\documentclass[11pt]{extarticle}
\usepackage{amsmath,amsfonts,amsthm,amssymb}
\usepackage{mathrsfs}
\usepackage{needspace}
\usepackage{hyperref}

\hypersetup{
  pdftitle={Global-in-time justification of the vortex-wave system in the vanishing-viscosity limit},
  pdfauthor={Trinh T. Nguyen},
  pdfsubject={Vanishing-viscosity limit of the two-dimensional Navier-Stokes equations},
  pdfkeywords={vortex-wave system, inviscid limit, Navier-Stokes, Lamb-Oseen vortex, weighted energy estimate}
}
\allowdisplaybreaks
\numberwithin{equation}{section}
\newtheorem{theorem}{Theorem}[section]
\newtheorem{proposition}[theorem]{Proposition}
\newtheorem{lemma}[theorem]{Lemma}
\newtheorem{corollary}[theorem]{Corollary}

\newcommand{\RR}{\mathbb{R}}
\newcommand{\Rt}{\mathbb{R}^{2}}

\newcommand{\w}{\omega}
\newcommand{\wb}{\overline{w}}
\newcommand{\vb}{\overline{v}}
\newcommand{\cL}{\mathcal{L}}
\newcommand{\cZ}{\mathcal{Z}}
\newcommand{\pw}{\mathsf{p}}
\newcommand{\Zn}{Z^{\nu}}
\newcommand{\wapp}{\w^{E,\nu}_{\text{app}}}
\newcommand{\vapp}{v^{E,\nu}_{\text{app}}}
\newcommand{\lbb}{L^{4}\cap L^{\frac{4}{3}}}
\newcommand{\Hz}{H^{z}}
\newcommand{\norm}[2][]{\left\lVert #2\right\rVert_{#1}}
\newcommand{\abs}[1]{\left\lvert #1\right\rvert}
\newcommand{\inner}[2]{\left\langle #1,#2\right\rangle}
\newcommand{\lsim}[1]{\mathrel{\underset{#1}{\lesssim}}}
\DeclareMathOperator{\supp}{supp}
\DeclareMathOperator{\dist}{dist}
\begin{document}
\title{Global-in-time inviscid limit of the Navier-Stokes equations for 
vortex-wave data}
\author{Trinh T. Nguyen\thanks{School of Mathematical Sciences,
Shanghai Jiao Tong University, Shanghai, China.
Email: \href{mailto:q0004349@sjtu.edu.cn}{\texttt{q0004349@sjtu.edu.cn}}.}}
\date{September 26, 2026}
\maketitle
\begin{abstract}
\noindent
We justify the vortex-wave system introduced by Marchioro and Pulvirenti in the 1990s as the vanishing-viscosity limit of the two-dimensional Navier-Stokes equations on every prescribed finite time interval. The initial vorticity consists of a point vortex and a smooth, compactly supported background that vanishes near the vortex. The regular component converges strongly to the background vorticity governed by the vortex-wave system, while the concentrated component is approximated by a Lamb-Oseen vortex centered on the limiting point-vortex trajectory. We obtain quantitative estimates for both components, with constants and a viscosity threshold depending on the chosen time interval, thereby providing a global-in-time justification of the vortex-wave system for this class of initial data.
\end{abstract}
\newpage 

\clearpage
\section{Introduction}\label{sec:intro}
Vortices are a prominent feature of atmospheric flows, from the swirling cloud patterns downstream of islands to the organized circulation of hurricanes. Their motion and persistence motivate the study of how concentrated vorticity interacts with the surrounding fluid. Two-dimensional incompressible flow provides an idealized setting in which to investigate this interaction. At large Reynolds numbers, advection acts on a much shorter time scale than viscous diffusion at the scale of the flow, suggesting that inviscid models can capture the leading dynamics. Justifying this description is a central objective of the inviscid-limit problem: understanding the behavior of Navier–Stokes solutions as viscosity tends to zero. When vorticity is concentrated in a small core, a natural approximation replaces the core by a point vortex while retaining the distributed vorticity of the surrounding fluid. The resulting interaction is mutual: the ambient flow transports and deforms the vortex, whose induced velocity in turn transports the ambient vorticity. Establishing the validity of this approximation over long time intervals requires controlling both this coupling and the cumulative effects of diffusion on the vortex core.

~\\
The vortex--wave system, introduced by Marchioro and Pulvirenti
\cite{MaPu91,MaPu94}, retains both the concentrated vortex and the
distributed ambient vorticity. We consider a single point vortex of
unit circulation, with position \(z(t)\in\Rt\), interacting with a
regular vorticity field \(\w^E(x,t)\). Their evolution is governed by
\begin{equation}\label{eq:VW}
\left\{
\begin{aligned}
\partial_t\w^E+(v^E+\Hz)\cdot\nabla\w^E&=0,\\
\partial_t z(t)&=v^E(t,z(t)),\\
v^E&=K\star\w^E,\\
\Hz(x,t)&=K(x-z(t)),\\
\w^E\big|_{t=0}&=\w_0^E,
\qquad z(0)=z_0.
\end{aligned}
\right.
\end{equation}
Here \(K\) is the planar Biot--Savart kernel,
\[
K(x)=\frac{1}{2\pi}\frac{x^\perp}{\abs{x}^{2}},
\qquad x^\perp=(-x_2,x_1),
\]
and \(\star\) denotes convolution in \(\Rt\). Thus \(v^E\) is the
velocity generated by the regular vorticity, and \(\Hz\) is the
velocity generated by the point vortex. The first equation transports
the regular vorticity by their combined velocity. The second equation
advects the point vortex by the regular velocity evaluated at its
position; its singular self-induced velocity is omitted. The initial
position and regular vorticity are denoted by \(z_0\) and \(\w_0^E\).

~\\
Marchioro and Pulvirenti established global existence for this system
in its Lagrangian formulation for a single point vortex and bounded,
integrable ambient vorticity \cite{MaPu91,MaPu94}. They also justified
the vortex--wave description through concentration limits of Euler
flows: under suitable localization assumptions, the point vortex is
replaced by a localized vorticity distribution of fixed circulation,
and its initial diameter is allowed to tend to zero \cite{MaPu93}.
Bjorland \cite{Bjor} subsequently obtained a concentration limit for
a single vortex in a bounded, integrable background, with subsequential
convergence to a weak vortex--wave solution in a velocity formulation.

~\\
For compactly supported \(L^1\cap L^\infty\) backgrounds initially
constant near the vortex, Lacave and Miot \cite{Lacave} proved
uniqueness and persistence of local constancy. In particular, if the
initial background vanishes near the vortex, its support remains
separated from the vortex on every finite time interval. For smooth
data, this separation also allows the usual transport estimates to
propagate Sobolev regularity.

~\\
In this paper, we justify the vortex--wave system as the
vanishing-viscosity limit of the two-dimensional Navier--Stokes
equations on every prescribed finite time interval. We start directly
with initial vorticity
\begin{equation}\label{eq:data}
\w_0=\delta_{z_0}+\w_0^E,
\end{equation}
where \(\delta_{z_0}\) is the unit Dirac mass at \(z_0\), and
\(\w_0^E\) is smooth, compactly supported, and vanishes in a
neighborhood of \(z_0\). For viscosity \(\nu>0\), the vorticity
\(\w^\nu=\partial_{x_1}u_2^\nu-\partial_{x_2}u_1^\nu\) solves
\begin{equation}\label{eq:NS}
\partial_t\w^\nu+u^\nu\cdot\nabla\w^\nu=\nu\Delta\w^\nu,
\qquad
u^\nu=K\star\w^\nu,
\qquad
\w^\nu\big|_{t=0}=\w_0.
\end{equation}
This problem has a unique global solution for finite-measure initial
vorticity, smooth for positive time \cite{GG,Kato}.

~\\
We prove that its
regular component converges strongly to \(\w^E\), while its
concentrated component remains close to a Lamb--Oseen vortex whose
center follows \(z(t)\), with quantitative bounds on both errors.
Our result extends the theorem of Nguyen and Nguyen \cite{NN} by
removing the stopping-time restrictions associated with the past wave
support and the ambient strain. The time interval \([0,T]\) is
arbitrary but fixed: the constants and the viscosity threshold may
depend on \(T\). This is the meaning of global-in-time justification
used here.

\subsection{Related results}

The inviscid limit in the whole space has a classical theory for
regular data. Swann \cite{Swann} proved convergence on a
viscosity-independent time interval for smooth three-dimensional
flows. Masmoudi \cite{Mas} established convergence in the Sobolev
space of the initial velocity on every compact interval within the
lifespan of the corresponding Euler solution. 
Constantin and Wu obtained quantitative estimates for vortex patches
\cite{CWu1} and strong vorticity convergence for nonsmooth data under
additional Besov regularity assumptions \cite{CWu2}. Chemin
\cite{Che} proved strong convergence of the velocity for
bounded-vorticity data. Sueur \cite{Sueur} further described the
internal viscous transition layer along a moving vortex-patch
interface through an asymptotic expansion. See also
\cite{PeterTarekTheo} for strong vorticity convergence in the periodic
Yudovich setting. The atomic component in
\eqref{eq:data} lies outside these vorticity classes.

~\\
In domains with a no-slip boundary, the mismatch between viscous and
inviscid boundary conditions introduces the additional problem of
boundary-layer formation. Sammartino and Caflisch
\cite{SammartinoCaflisch2} justified the Euler--Prandtl expansion for
well-prepared analytic data in a half-space. Maekawa \cite{Mae}
established the inviscid limit in the half-plane for sufficiently
regular initial vorticity supported away from the boundary. Nguyen
and Nguyen \cite{2N} gave a direct proof for analytic data in the
periodic half-plane using the boundary vorticity formulation, without
constructing Prandtl correctors. These results hold on short time
intervals independent of viscosity.
See also \cite{BNTT,ConstantinKuka,KVW,2N-ext} and the references
therein for further inviscid-limit results in domains with boundaries.

~\\
For a single point vortex in the half-plane with a no-slip boundary,
Dalibard and Gallay \cite{dalibard2026viscousevolutionpointvortex}
and Wang, Yue, and Zhang
\cite{wang2026navierstokesequationsmathbbr2} recently constructed
global solutions for arbitrary circulation and established uniqueness
within classes specified by their short-time asymptotic behavior.
Both approaches distinguish the concentrated vortex core from the
initial boundary layer generated by the no-slip condition.

~\\
For concentrated planar vorticity, Marchioro \cite{Ma90,Ma98}
studied initial distributions localized in small, separated regions
and justified point-vortex dynamics in joint limits of vanishing
core size and viscosity. These results concern families of regularized
concentrated initial data. Starting directly from a finite sum of
Dirac masses, Gallay \cite{Ga11} proved that the viscous solution is
approximated by a superposition of Lamb--Oseen vortices on any fixed
interval before collision of the limiting point vortices. His
analysis also describes the leading deformation of each core caused
by the other vortices. Our analysis of the concentrated component is also inspired by Gallay's treatment of vortex-core deformation and his use of weighted energy estimates. Building on these ideas, we develop estimates that remove the short-time restriction associated with the ambient strain.

~\\
For a point vortex interacting with a regular background,
Nguyen and Nguyen \cite{NN} proved convergence of the regular and
concentrated components to the vortex--wave evolution on a restricted
time interval, with a detailed asymptotic expansion. Their argument
requires the vortex to avoid the region
swept out by the earlier wave supports and imposes a smallness
condition involving time and the ambient strain. Although the limiting
vortex--wave solution remains regular and separated on every finite
interval, these two requirements restrict the duration of the viscous
approximation. Removing them requires estimates that follow the
moving wave support and control the continuing deformation of the core.

~\\
Recent work has also addressed time intervals that grow as viscosity
tends to zero for particular vortex configurations. Dolce and Gallay
\cite{DolceGallay2026} treated the
dipole formed by two vortices of equal and opposite circulations.
They justified an asymptotic approximation for times of order
\(T_{\mathrm{adv}}\mathrm{Re}^{\sigma}\), for every
\(0<\sigma<1\), and obtained the leading correction to the
translation speed due to the finite size of the cores. Here
\(T_{\mathrm{adv}}\) is the characteristic advection time and
\(\mathrm{Re}\) is the circulation Reynolds number.
Zhang and Zhang \cite{zhang2025longtimeevolutionpair} obtained
approximations on the same range of time scales for a pair with
nonzero total circulation, including vortices of the same sign and
vortices of opposite signs with unequal strengths. In that setting,
the approximation follows corrected rotating trajectories.

~\\
The behavior of an individual viscous core is closely related to the
stability of the Lamb--Oseen vortex. Gallay and Wayne \cite{GW02}
constructed finite-dimensional invariant manifolds to describe the
long-time asymptotics of planar vorticity equations. In \cite{GW05},
they proved convergence in self-similar variables to the Oseen vortex
for integrable initial vorticity of arbitrary size. Gallay
\cite{Gallay2018} subsequently showed that strongly localized
perturbations at large circulation Reynolds number return to
axisymmetry on a time scale shorter than the diffusive one, through
enhanced dissipation produced by differential rotation.

~\\
The analysis of Donati and Gallay \cite{DonatiGallay2026} is
particularly relevant to the present work. They study a concentrated
viscous vortex in a prescribed smooth, divergence-free external flow
and construct
an approximation of the vortex motion and deformation on every fixed
finite interval for sufficiently small viscosity. For sharply
concentrated Gaussian initial data, they also prove rapid relaxation
toward the corresponding deformed vortex profile. We adapt their
expansion about the vortex center and their strain-cancelling weighted
estimates to an ambient velocity generated by an evolving regular
vorticity. In our problem, this velocity is coupled to the core, so
the estimates for the two components must be closed together.

~\\
Related questions in three dimensions have also led to precise
descriptions of concentrated vortices. Gallay and Maekawa
\cite{GallayMaekawa2011} proved asymptotic stability of Burgers
vortices under small three-dimensional perturbations for arbitrary
Reynolds number. For axisymmetric vortex rings, Gallay and
\v{S}ver\'ak \cite{GS24} justified the Kelvin--Saffman formula for the
self-induced motion of a viscous ring originating from a circular
vortex filament, on intervals growing as a positive power of the
Reynolds number. 

\subsection{Main result}
For the global solution of \eqref{eq:VW}, write
\(\delta(t):=\dist(z(t),\supp\w^E(t))\) and
\(\delta_0:=\delta(0)\).
We recall the separation estimate from \cite{Lacave}.
\begin{lemma}
\label{lem:separation}
Let \(\w_0^E\in C_c^\infty(\Rt)\) vanish in a neighborhood of \(z_0\),
and let \((z,\w^E)\) be the global smooth solution of \eqref{eq:VW}
with \(z(0)=z_0\) and \(\w^E(0)=\w_0^E\).
If \(\w_0^E\not\equiv0\), there exists \(\kappa>0\), depending
only on \(\delta_0\) and
\(\norm[L^1]{\w_0^E}+\norm[L^\infty]{\w_0^E}\), such that
\(\delta(t)\ge \exp(-\kappa e^{\kappa t})>0\) for every \(t\ge0\).
If \(\w_0^E\equiv0\), the same conclusion holds with
\(\delta(t)=+\infty\).
Consequently, for every finite \(T>0\),
\begin{equation}\label{eq:dT}
d_T:=\min_{0\le t\le T}\delta(t)>0.
\end{equation}
\end{lemma}
Following \cite{Ga11,NN}, we decompose
\(\w^\nu=\w^{E,\nu}+\w^{B,\nu}\).  Both components are advected by
\(u^\nu=v^{E,\nu}+v^{B,\nu}\), where
\(v^{E,\nu}=K\star\w^{E,\nu}\) and
\(v^{B,\nu}=K\star\w^{B,\nu}\), and solve
\begin{equation}\label{eq:comp-eq}
\begin{aligned}
\partial_t\w^{E,\nu}+u^\nu\cdot\nabla\w^{E,\nu}
&=\nu\Delta\w^{E,\nu},
&
\w^{E,\nu}\big|_{t=0}
&=\w_0^E,
\\
\partial_t\w^{B,\nu}+u^\nu\cdot\nabla\w^{B,\nu}
&=\nu\Delta\w^{B,\nu},
&
\w^{B,\nu}(t)
&\to\delta_{z_0}
\quad\text{weak-* as }t\to0^+.
\end{aligned}
\end{equation}
Weak-* convergence below is understood in the space of finite Radon
measures, tested against continuous functions vanishing at infinity.
Their masses are conserved:
\(\int_{\Rt}\w^{E,\nu}(x,t)\,dx=\int_{\Rt}\w_0^E(x)\,dx\) and
\(\int_{\Rt}\w^{B,\nu}(x,t)\,dx=1\).
For \(s>0\) and \(a\in\Rt\), let
\begin{equation}\label{eq:Gamma-def}
\Gamma_{s,a}(x)
:=
\frac{1}{4\pi s}
\exp\left(-\frac{\abs{x-a}^2}{4s}\right).
\end{equation}
Thus, \(\Gamma_{s,a}\) is the two-dimensional heat kernel at time \(s\),
centered at \(a\), and has unit mass:
\[
\int_{\Rt}\Gamma_{s,a}(x)\,dx=1.
\]
\begin{theorem}\label{thm:main}
Let \(z_0\in\Rt\), and let \(\w_0^E\) be smooth and compactly supported.
Assume that \(\w_0^E\) vanishes in a neighborhood of \(z_0\).
Let \((z,\w^E)\) be the global smooth solution of \eqref{eq:VW} with
\(z(0)=z_0\) and \(\w^E(0)=\w_0^E\).
For
every finite \(T>0\), there exists a viscosity threshold \(\nu_T>0\)
such that, for each \(0<\nu\le\nu_T\), there is a modulated trajectory
\(\Zn\in C^1([0,T];\Rt)\) satisfying
\begin{equation}\label{eq:main-center}
\sup_{0<t\le T}
\frac{\abs{\Zn(t)-z(t)}}{t}
\lsim{T}\nu,
\end{equation}
and
\begin{equation}\label{eq:main-estimates}
\begin{aligned}
\sup_{0\le t\le T}
\norm[\lbb]{\w^{E,\nu}(t)-\w^E(t)}
&\lsim{T}\nu,
\\
\sup_{0<t\le T}
\frac1t
\norm[L^1]{\w^{B,\nu}(t)-\Gamma_{\nu t,\Zn(t)}}
&\lsim{T}\nu.
\end{aligned}
\end{equation}
In particular,
\begin{equation}\label{eq:main-z-centered}
\norm[L^1]{\w^{B,\nu}(t)-\Gamma_{\nu t,z(t)}}
\lsim{T}\sqrt{\nu t},
\qquad 0<t\le T.
\end{equation}
Consequently, for each fixed \(t>0\),
\[
\w^{E,\nu}(t)\to\w^E(t)
\quad\hbox{in }\lbb,
\qquad
\w^{B,\nu}(t)\to\delta_{z(t)}
\quad\hbox{weak-* as }\nu\to0.
\]
\end{theorem}
The theorem controls both the transport of the regular vorticity and
the shape of the diffusing core throughout \([0,T]\). The sharper
\(L^1\) error \(O_T(\nu t)\) is measured around the modulated center
\(\Zn(t)\). Replacing that center by \(z(t)\) gives the error
\(O_T(\sqrt{\nu t})\) in \eqref{eq:main-z-centered}, because the
displacement is measured relative to the core radius \(\sqrt{\nu t}\).

\subsection{Outline of the proof}
If \(\omega_0^E\equiv0\), the solution is the Lamb--Oseen vortex and
Theorem~\ref{thm:main} is immediate. Assume henceforth that
\(\omega_0^E\not\equiv0\), and set \(\ell=d_T/10\). The parabolic scaling
\begin{equation}\label{eq:parabolic-normalization}
\widetilde\omega(x,t)=\ell^2\omega(\ell x,\ell^2t),
\qquad
\widetilde v(x,t)=\ell v(\ell x,\ell^2t),
\qquad
\widetilde z(t)=\ell^{-1}z(\ell^2t)
\end{equation}
preserves the equations, the viscosity, and the circulation.
Dropping the tildes and relabelling \(T/\ell^2\) as \(T\), we may assume
\begin{equation}
\label{eq:normalized-limiting-separation}
S_E(t):=\supp\omega^E(t),
\qquad
\dist\bigl(z(t),S_E(t)\bigr)\ge10,
\qquad
0\le t\le T.
\end{equation}

~\\
The proof combines transport estimates for the regular component with
an expansion and weighted energy estimates for the concentrated core.
We describe the four steps that allow these estimates to hold on the
whole prescribed interval.

\paragraph{Transport of the regular corrector.}
The source of the regular corrector at time \(s\) is supported in
\(S_E(s)\). Its contribution at time \(t\) must therefore be located
after transport from \(s\) to \(t\). The limiting flow sends
\(S_E(s)\) exactly onto \(S_E(t)\). Comparing it with the flow of
the smooth localized velocity in
\eqref{eq:localized-flow-global-ODE} gives
\[
\sup_{\substack{0\le s\le t\le T\\y\in S_E(s)}}
\abs{\Phi^\nu_{t,s}(y)-\Phi_{t,s}(y)}\lsim{T}\nu.
\]
Thus all past source contributions lie in a \(C_T\nu\)-neighborhood
of the current support. The localization is introduced before this
comparison, so that the flow estimate does not assume the desired
support bound. Once the bound is proved, the localized fields agree
with the original ones on the corrector support. This recovers the
original corrector equation and keeps \(\supp\wapp(t)\) at distance
at least \(8\) from \(\Zn(t)\), by \eqref{eq:w1-support-tube}.
The supremum of a nonnegative quantity over an empty support is understood
to be zero.

\paragraph{Expansion of the concentrated core.}
With separation established, the ambient velocity is smooth near the
core. We expand it about \(\Zn(t)\) in the variables
\(\xi=(x-\Zn(t))/\sqrt{\nu t}\), following
\cite{Ga11,NN}. The nonradial corrections are obtained
by solving fixed angular-mode equations, and the radial correction
solves a linear evolution equation. These profiles describe the
deformation of the Gaussian core under the ambient flow. The expansion
\eqref{eq:high-app} has residual
\[
\abs{\Phi_{\mathrm{app}}(\xi,t)}
\lsim{\gamma,T}(\nu t)^{3/2}e^{-\gamma|\xi|^2/4},
\qquad \tfrac12<\gamma<1,
\]
by Proposition~\ref{prop:high-residual}. Its moment cancellations
also give the far-field velocity decay used in the wave estimate.

\paragraph{Cancellation of the ambient strain.}
The ambient strain
\(\Sigma(t)=\operatorname{sym}D_x\vapp(\Zn(t),t)\) contributes a
term of size \(t|\Sigma(t)|\) to a Gaussian energy estimate. Treating
this term only by its absolute value would restrict the time interval.
Following \cite{Ga11,DonatiGallay2026}, we instead use
\(\pw=\pw_0+\nu t q\), where \(\pw_0\) is a radial weight
with Gaussian, constant, and exponential regions. In the inner region,
the angular correction satisfies
\[
v^G\cdot\nabla q=-\Sigma(t)\xi\cdot\nabla\pw_0.
\]
This identity cancels the leading strain term in the inner region.
On the plateau, \(\nabla\pw_0=0\), so the transport term vanishes
from the energy identity. The exponential tail controls the remaining
transport and nonlocal terms. The time derivative of \(\pw_0\)
contributes to the coercivity estimate \eqref{eq:p0-pointwise-coercivity}.
Theorem~\ref{thm:linear-weighted-estimate}
gives coercivity without requiring
\(\sup_{0<t\le T}t|\Sigma(t)|\) to be small.

\paragraph{Estimates for the remainders.}
The remaining task is to control the feedback between the wave error
and the core error. The exact rescaled core starts from \(G\), and a
short-time estimate for the unscaled approximation error provides
zero initial remainder energy. Write \(E_{\pw}\) for the core
energy, \(\mathcal Q\) for its dissipation, and
\(\norm{\wb_1(t)}\) for the \(L^4\cap L^{4/3}\) norm of the regular
remainder defined in \eqref{eq:wave-remainder-norm}. Sections~\ref{sec:w2-estimate}
and~\ref{sec:w1-estimate} give the coupled estimates
\eqref{eq:w2-apriori} and \eqref{eq:w1-estimate}. With
\[
D(t)=\int_0^t\frac{\mathcal Q(s)}s\,ds,
\qquad Y_*(t)=E_{\pw}(t)+\kappa_1D(t),
\]
the wave estimate yields
\(\norm{\wb_1(t)}^2/t^2\lsim{T}\nu+t^3D(t)\).
Substitution into the core estimate gives
\[
Y_*'(t)\le C_T\bigl(\nu+Y_*(t)+Y_*(t)^2\bigr),
\qquad Y_*(0)=0.
\]
For sufficiently small \(\nu\), a continuity argument gives \mbox{\(Y_*(t)\lsim{T}\nu t\)} and
\mbox{\(\norm{\wb_1(t)}/t\lsim{T}\sqrt\nu\)}. The remainder scalings then give
Theorem~\ref{thm:main} on the prescribed interval.

\paragraph{Organization of the paper.}

Section~\ref{sec:approx} constructs the regular approximation.
Sections~\ref{sec:b-exp} and~\ref{sec:high-approx} expand the ambient
drift and construct the vortex profiles.
Section~\ref{sec:remainder-equations} gives the remainder equations,
and Section~\ref{sec:linearized-estimate} proves the weighted energy
estimate. Sections~\ref{sec:w2-estimate} and~\ref{sec:w1-estimate}
estimate the two remainders. Section~\ref{sec:closure} closes these
estimates and proves Theorem~\ref{thm:main}.
Appendix~\ref{app:corrector-estimates} contains the transport and
Biot--Savart estimates.

\paragraph{Notation.}
The notation \(A\lsim{\mathcal P}B\), for \(B\ge0\), means
\(A\le cB\) for a positive constant \(c\) depending only on the
listed parameters \(\mathcal P\), the fixed finite time \(T\), the
limiting solution, and the initial data.  The dependence on these fixed
background quantities is usually suppressed, so that \(A\lesssim B\)
has the same convention. In particular, \(A\lsim{\mathcal P}1\)
denotes a uniform upper bound. Unless dependence on \(\nu\) is explicitly
indicated, implicit constants are independent of \(\nu\), of time in the stated interval, and of the positive starting
time \(t_0\) in Section~\ref{sec:closure}. Dependence on parameters to be chosen in absorption or bootstrap
arguments is always indicated. We retain explicit constants when their
values determine these choices or occur in an exponential rate.
We write
\(\norm[\lbb]{f}:=\norm[L^4]{f}+\norm[L^{4/3}]{f}\),
\(\norm[L^2_q]{f}^2:=\int_{\Rt}\abs{f(\xi)}^2q(\xi)\,d\xi\), and
\(\inner{f}{g}_q:=\int_{\Rt}f(\xi)g(\xi)q(\xi)\,d\xi\).
The Oseen profile and its velocity are
\begin{equation}\label{eq:oseen}
G(\xi)=\frac{1}{4\pi}e^{-\frac{\abs{\xi}^{2}}{4}},
\qquad
v^G(\xi)
=
\frac{1}{2\pi}
\frac{\xi^\perp}{\abs{\xi}^{2}}
\left(1-e^{-\frac{\abs{\xi}^{2}}{4}}\right)
=K\star_\xi G.
\end{equation}
The Fokker--Planck and linearized self-interaction operators are
\begin{equation}\label{eq:LLambda}
\cL f
=
\Delta f+\frac12\xi\cdot\nabla f+f,
\qquad
\Lambda f
=
v^G\cdot\nabla f+(K\star_\xi f)\cdot\nabla G.
\end{equation}
They satisfy \(\cL G=0\) and \(v^G\cdot\nabla G=0\).
The weighted-space structure of \(\cL\) and \(\Lambda\) is recalled in
the profile analysis below.
\section{Approximate solution of the regular part}
\label{sec:approx}
\subsection{The leading regular corrector}

For \(t>0\), define the viscous Biot--Savart kernel
\begin{equation}\label{eq:viscous-Biot-Savart}
K^\nu(x,t)
:=
\frac1{\sqrt{\nu t}}\,
v^G\left(\frac{x}{\sqrt{\nu t}}\right)
=
K(x)
\left(
1-e^{-\frac{\abs{x}^2}{4\nu t}}
\right).
\end{equation}
Thus, if \(a(t)\) is a \(C^1\) curve, the velocity generated by a
unit Lamb--Oseen vortex centered at \(a(t)\) is
\[
K^\nu(x-a(t),t).
\]
For \(x\neq0\), define the continuous extension
\begin{equation}\label{eq:viscous-Biot-Savart-t0}
K^\nu(x,0):=K(x).
\end{equation}
In particular, for \(x\neq a(0)\),
\[
K^\nu(x-a(0),0)=K(x-a(0)).
\]

\begin{lemma}
\label{lem:regular-corrector-algebra}
Let \(w\) be smooth, let \(v^w=K\star w\), and set
\[
\w_{\mathrm{reg}}^{\mathrm{app}}
=
\w^E+\nu w,
\qquad
v_{\mathrm{reg}}^{\mathrm{app}}
=
v^E+\nu v^w.
\]
For every \(C^1\) curve \(\Zn\),
\begin{equation}\label{eq:regular-defect-general}
\begin{aligned}
&
\partial_t\w_{\mathrm{reg}}^{\mathrm{app}}
+
\bigl(
v_{\mathrm{reg}}^{\mathrm{app}}
+
K^\nu(\,\cdot-\Zn(t),t)
\bigr)
\cdot\nabla\w_{\mathrm{reg}}^{\mathrm{app}}
-
\nu\Delta\w_{\mathrm{reg}}^{\mathrm{app}}
\\
&\quad=
\bigl(
K^\nu(\,\cdot-\Zn(t),t)-\Hz
\bigr)\cdot\nabla\w^E
\\
&\qquad
+
\nu
\left[
\partial_tw
+
\bigl(
v^E+K^\nu(\,\cdot-\Zn(t),t)
\bigr)\cdot\nabla w
+
v^w\cdot\nabla\w^E
-
\Delta\w^E
\right]
\\
&\qquad
+
\nu^2
\left(
v^w\cdot\nabla w-\Delta w
\right).
\end{aligned}
\end{equation}
\end{lemma}

\begin{proof}
Substituting
$
\w_{\mathrm{reg}}^{\mathrm{app}}
=
\w^E+\nu w,\quad
v_{\mathrm{reg}}^{\mathrm{app}}
=
v^E+\nu v^w,
$
we obtain
\[
\partial_t\w_{\mathrm{reg}}^{\mathrm{app}}
=
\partial_t\w^E+\nu\partial_tw,
\]
and
\[
-\nu\Delta\w_{\mathrm{reg}}^{\mathrm{app}}
=
-\nu\Delta\w^E-\nu^2\Delta w.
\]
For the transport term,
\[
\begin{aligned}
&
\bigl(
v_{\mathrm{reg}}^{\mathrm{app}}
+
K^\nu(\,\cdot-\Zn(t),t)
\bigr)
\cdot\nabla\w_{\mathrm{reg}}^{\mathrm{app}}
\\
&\quad=
\bigl(
v^E+\nu v^w+K^\nu(\,\cdot-\Zn(t),t)
\bigr)
\cdot\nabla(\w^E+\nu w)
\\
&\quad=
\bigl(
v^E+K^\nu(\,\cdot-\Zn(t),t)
\bigr)\cdot\nabla\w^E
+
\nu v^w\cdot\nabla\w^E
\\
&\qquad
+
\nu
\bigl(
v^E+K^\nu(\,\cdot-\Zn(t),t)
\bigr)\cdot\nabla w
+
\nu^2v^w\cdot\nabla w.
\end{aligned}
\]
Consequently,
\[
\begin{aligned}
&
\partial_t\w_{\mathrm{reg}}^{\mathrm{app}}
+
\bigl(
v_{\mathrm{reg}}^{\mathrm{app}}
+
K^\nu(\,\cdot-\Zn(t),t)
\bigr)
\cdot\nabla\w_{\mathrm{reg}}^{\mathrm{app}}
-
\nu\Delta\w_{\mathrm{reg}}^{\mathrm{app}}
\\
&\quad=
\partial_t\w^E
+
\bigl(
v^E+K^\nu(\,\cdot-\Zn(t),t)
\bigr)\cdot\nabla\w^E
\\
&\qquad
+
\nu
\left[
\partial_tw
+
\bigl(
v^E+K^\nu(\,\cdot-\Zn(t),t)
\bigr)\cdot\nabla w
+
v^w\cdot\nabla\w^E
-
\Delta\w^E
\right]
\\
&\qquad
+
\nu^2
\left(
v^w\cdot\nabla w-\Delta w
\right).
\end{aligned}
\]
The leading term vanishes:
\[
\partial_t\w^E
+
(v^E+\Hz)\cdot\nabla\w^E
\stackrel{\eqref{eq:VW}}{=}
0.
\]
Hence
\[
\begin{aligned}
&
\partial_t\w^E
+
\bigl(
v^E+K^\nu(\,\cdot-\Zn(t),t)
\bigr)\cdot\nabla\w^E
\\
&\quad\stackrel{\eqref{eq:VW}}{=}
\bigl(
K^\nu(\,\cdot-\Zn(t),t)-\Hz
\bigr)\cdot\nabla\w^E.
\end{aligned}
\]
Substituting this identity into the preceding expansion proves
\eqref{eq:regular-defect-general}.
\end{proof}
We therefore define
\begin{equation}\label{eq:regular-approx-def}
\wapp=\w^E+\nu w_{1,a},
\qquad
\vapp=v^E+\nu v_{1,a},
\qquad
v_{1,a}=K\star_x w_{1,a},
\end{equation}
and determine \((w_{1,a},\Zn)\) from
\begin{equation}\label{eq:regular-corrector-system}
\left\{
\begin{aligned}
\partial_tw_{1,a}
{}&+\bigl(v^E+K^\nu(\,\cdot-\Zn(t),t)\bigr)\cdot\nabla w_{1,a}
+v_{1,a}\cdot\nabla\w^E
\\
&=
\Delta\w^E
-\frac1\nu
\bigl(K^\nu(\,\cdot-\Zn(t),t)-\Hz\bigr)\cdot\nabla\w^E,
\\
\partial_t\Zn(t)&=\vapp(\Zn(t),t),
\\
w_{1,a}\big|_{t=0}&=0,
\qquad
\Zn(0)=z_0.
\end{aligned}
\right.
\end{equation}

\begin{proposition}
\label{prop:regular-corrector}
For every finite \(T>0\), there is \(\nu_{0,T}>0\)
such that, for \(0<\nu\le\nu_{0,T}\), system
\eqref{eq:regular-corrector-system} has a solution satisfying
\begin{equation}\label{eq:regular-corrector-bounds}
\begin{aligned}
\sup_{0\le t\le T}
\Bigl(
\norm[W^{4,4}\cap W^{3,1}]{w_{1,a}(t)}
+\norm[W^{3,\infty}]{v_{1,a}(t)}
\Bigr)
&\lsim{T}1,
\\
\sup_{0\le t\le T}
\Bigl(
\norm[W^{2,4}\cap W^{2,1}]{\partial_tw_{1,a}(t)}
+\norm[W^{2,\infty}]{\partial_tv_{1,a}(t)}
\Bigr)
&\lsim{T}1,
\\
\sup_{0\le t\le T}
\Bigl(
\norm[L^4\cap L^1]{\partial_t^2w_{1,a}(t)}
+\abs{\partial_t^2\Zn(t)}
\Bigr)
&\lsim{T}1,
\\
\abs{\Zn(t)-z(t)}&\lsim{T}\nu t.
\end{aligned}
\end{equation}
Moreover,
\begin{equation}\label{eq:approx-separation}
\dist\bigl(\Zn(t),\supp\wapp(t)\bigr)\ge8,
\qquad 0\le t\le T,
\end{equation}
and
\begin{equation}\label{eq:regular-approx-residual}
\partial_t\wapp
+(\vapp+V_G^{\Zn,\nu})\cdot\nabla\wapp
-\nu\Delta\wapp
=
\nu^2\left(v_{1,a}\cdot\nabla w_{1,a}-\Delta w_{1,a}\right).
\end{equation}
Finally,
\begin{equation}\label{eq:w1a-zero-mass}
\int_{\Rt}w_{1,a}(x,t)\,dx=0.
\end{equation}
\end{proposition}

~\\
To prove \eqref{eq:approx-separation}, we show that
\begin{equation}\label{eq:key-support-mechanism}
\sup_{x\in\supp w_{1,a}(t)}
\dist\bigl(x,S_E(t)\bigr)\lsim{T}\nu,
\qquad
S_E(t):=\supp\omega^E(t).
\end{equation}
The point vortex remains separated from \(S_E(t)\), but need not remain
separated from
\[
\bigcup_{0\le s\le t}S_E(s).
\]
We therefore compare the two flows from each source time \(s\)
to time \(t\).
We localize the stream functions to obtain smooth divergence-free fields
that agree with the original fields near \(S_E(t)\). This allows us to
construct the corrector before proving that its support remains separated
from the vortex.

~\\
Let \(\Phi_{t,s}\) be the flow of
\[
v^E+H^z,
\]
and \(\Phi^\nu_{t,s}\) the flow of the localized velocity
\[
v^E+\mathcal V_{\nu,\zeta}.
\]
Since \(\omega^E\) is transported by the limiting flow,
\begin{equation}\label{eq:current-support-flow-idea}
\Phi_{t,s}\bigl(S_E(s)\bigr)=S_E(t).
\end{equation}
We prove
\begin{equation}\label{eq:key-flow-comparison-idea}
\sup_{\substack{0\le s\le t\le T\\ y\in S_E(s)}}
\left|
\Phi^\nu_{t,s}(y)-\Phi_{t,s}(y)
\right|
\lsim{T}\nu.
\end{equation}
We compare the smooth fields
\(v^E+\mathcal V_{\nu,\zeta}\) and \(v^E+\mathcal H_\chi\).
Their difference is \(O_T(\nu)\) in \(L^\infty\), and the latter
has a uniform Lipschitz bound; see
\eqref{eq:localized-field-bounds} and
\eqref{eq:auxiliary-uniform-bound}. Since the limiting trajectories
starting in \(S_E(s)\) remain in \(S_E(r)\), they also solve the
equation for \(v^E+\mathcal H_\chi\). Gronwall's inequality gives
\[
\abs{\Phi^\nu_{t,s}(y)-\Phi_{t,s}(y)}
\lsim{T}\nu(t-s),\qquad y\in S_E(s),
\]
and hence \eqref{eq:key-flow-comparison-idea}.

~\\
The corrector has zero initial data and forcing supported in \(S_E(s)\).
The characteristic formula gives
\[
\supp w_{1,a}(t)
\subset
\overline{
\bigcup_{0\le s\le t}
\Phi^\nu_{t,s}\bigl(S_E(s)\bigr)
}.
\]
For each source point,
\[
\begin{aligned}
\dist\left(
\Phi^\nu_{t,s}(y),S_E(t)
\right)
&\stackrel{\eqref{eq:current-support-flow-idea}}{\le}
\left|
\Phi^\nu_{t,s}(y)-\Phi_{t,s}(y)
\right|
\\
&\stackrel{\eqref{eq:key-flow-comparison-idea}}{\lsim{T}}\nu,
\qquad y\in S_E(s),
\end{aligned}
\]
and hence
\[
\sup_{x\in\supp w_{1,a}(t)}\dist(x,S_E(t))\lsim{T}\nu.
\]
For small \(\nu\), the localized and original fields agree on
\(\supp\nabla w_{1,a}(t)\). They also agree on
\(\supp\nabla\omega^E(t)\subset S_E(t)\).
Thus the auxiliary solution satisfies the original corrector equation
on \([0,T]\).

~\\
The transport and Biot--Savart estimates used in the local construction
are proved in Appendix~\ref{app:corrector-estimates}.
\subsection{Divergence-free localization near the wave support}
For \(0\le s,r\le T\) and \(x\in S_E(s)\), define
\(\Phi_{r,s}(x)\) by
\begin{equation}\label{eq:Euler-flow-ODE}
\frac{d}{dr}\Phi_{r,s}(x)
=
\bigl(v^E+\Hz\bigr)\bigl(\Phi_{r,s}(x),r\bigr),
\qquad
\Phi_{s,s}(x)=x.
\end{equation}
There exist \(\epsilon_T>0\) and a smooth compactly supported function
\(\chi:\Rt\times(-\epsilon_T,T+\epsilon_T)\to\mathbb R\),
with \(0\le\chi\le1\), such that
\begin{equation}\label{eq:chi-tubes}
\chi(x,t)=\begin{cases} 1 & \text{if } \dist(x,S_E(t))\le 2,\\ 0 & \text{if } \dist(x,S_E(t))\ge 3. \end{cases}
\end{equation}
The flow of \eqref{eq:Euler-flow-ODE} is continuous in \((t,y)\), and
\(S_E(t)=\Phi_{t,0}\bigl(S_E(0)\bigr)\) with \(S_E(0)\) compact. Thus
\(t\mapsto S_E(t)\) is continuous in the Hausdorff metric.  The closed space--time tube of
radius \(2\) and the closed complement of the tube of radius
\(3\) are separated.  A smooth Urysohn cutoff, followed by a compact
spatial cutoff which is one on the bounded union of the tubes, gives
\eqref{eq:chi-tubes} (after extending the time interval slightly).
For \(t>0\), define
\begin{equation}\label{eq:localized-stream-functions}
\begin{aligned}
\Psi_H^a(x)
&:=\frac1{2\pi}\log\abs{x-a},
\\
\Psi_G^{a,\nu}(x,t)
&:=\frac1{2\pi}\log\abs{x-a}
 +\frac1{4\pi}E_1\left(\frac{\abs{x-a}^2}{4\nu t}\right),
\\
E_1(q)&:=\int_q^\infty\frac{e^{-s}}s\,ds.
\end{aligned}
\end{equation}

~\\
For every \(x\neq a\),
\[
\lim_{t\to0^+}\Psi_G^{a,\nu}(x,t)
=
\Psi_H^a(x),
\]
and therefore we define
\[
\Psi_G^{a,\nu}(x,0):=\Psi_H^a(x).
\]
Moreover, \(\Psi_H^a\) and \(\Psi_G^{a,\nu}\) are the stream functions
of the unit point vortex and the Lamb--Oseen vortex centered at \(a\),
respectively:
\[
\nabla^\perp\Psi_H^a=K(\cdot-a),
\qquad
\nabla^\perp\Psi_G^{a,\nu}=V_G^{a,\nu}.
\]
For \(\zeta\in\Rt\) satisfying \(\nu\abs\zeta<7\), let
\begin{equation}\label{eq:localized-fields}
\begin{aligned}
\mathcal V_{\nu,\zeta}(x,t)
&:=\nabla^\perp\left[
\chi(x,t)\Psi_G^{z(t)+\nu\zeta,\nu}(x,t)
\right],
\\
\mathcal H_\chi(x,t)
&:=\nabla^\perp\left[
\chi(x,t)\Psi_H^{z(t)}(x)
\right],
\\
\mathcal R_{\nu,\zeta}^{\mathrm{fld}}(x,t)
&:=\frac{\mathcal V_{\nu,\zeta}(x,t)-\mathcal H_\chi(x,t)}\nu.
\end{aligned}
\end{equation}
Here \(\chi\) is given by \eqref{eq:chi-tubes} and \(\nu\zeta\) is the
displacement from \(z(t)\). Since the cutoff acts on the stream functions,
the fields are divergence free. They are smooth on \(\Rt\), because
\(\supp\chi\) stays at distance at least \(7\) from \(z(t)\) and at
positive distance from \(z(t)+\nu\zeta\).
\begin{lemma}
\label{lem:divfree-localization}
Let \(M\ge1\), and suppose that \(0<\nu\le1\) and \(\nu M\le1\).
Then, for every \(\zeta\in\Rt\) with \(\abs\zeta\le M\), the fields in
\eqref{eq:localized-fields} are smooth in \(x\) and divergence free and,
for every \(t\in[0,T]\),
\begin{equation}\label{eq:localized-field-bounds}
\begin{aligned}
\norm[W^{5,\infty}]{\mathcal V_{\nu,\zeta}(t)}
+\norm[W^{5,\infty}]{\mathcal H_\chi(t)}
&\lsim{T}1,
\\
\norm[W^{4,\infty}]
 {\mathcal R_{\nu,\zeta}^{\mathrm{fld}}(t)}
&\lsim{T}1+\abs{\zeta}.
\end{aligned}
\end{equation}
Moreover, for every \(\zeta_1,\zeta_2\in\Rt\) with
\(\abs{\zeta_i}\le M\), \(i=1,2\), and every \(t\in[0,T]\),
\begin{equation}\label{eq:localized-field-Lipschitz}
\begin{aligned}
\norm[W^{4,\infty}]
 {\mathcal V_{\nu,\zeta_1}(t)-\mathcal V_{\nu,\zeta_2}(t)}
&\lsim{T}\nu\abs{\zeta_1-\zeta_2},
\\
\norm[W^{3,\infty}]
 {\mathcal R_{\nu,\zeta_1}^{\mathrm{fld}}(t)
 -\mathcal R_{\nu,\zeta_2}^{\mathrm{fld}}(t)}
&\lsim{T}\abs{\zeta_1-\zeta_2}.
\end{aligned}
\end{equation}
On \(\{\dist(x,S_E(t))\le2\}\), one has, for \(t>0\), the exact identities
\begin{equation}\label{eq:localized-fields-equal}
\mathcal V_{\nu,\zeta}=V_G^{z+\nu\zeta,\nu},
\qquad
\mathcal H_\chi=\Hz,
\qquad
\mathcal R_{\nu,\zeta}^{\mathrm{fld}}
=\frac{V_G^{z+\nu\zeta,\nu}-\Hz}{\nu}.
\end{equation}
Finally, let \(\zeta\in C^1([0,T];\Rt)\) be any path satisfying
\(\abs{\zeta(t)}\le M\).  Along this path, for \(0<t\le T\),
\begin{equation}\label{eq:localized-field-time-bounds}
\begin{aligned}
\left\|\frac d{dt}\mathcal V_{\nu,\zeta(t)}(t)\right\|_{W^{4,\infty}}
+\norm[W^{4,\infty}]{\partial_t\mathcal H_\chi(t)}
&\lsim{T}1+\nu\abs{\partial_t\zeta(t)},
\\
\left\|\frac d{dt}
\mathcal R_{\nu,\zeta(t)}^{\mathrm{fld}}(t)\right\|_{W^{3,\infty}}
&\lsim{T}1+\abs{\zeta(t)}
                    +\abs{\partial_t\zeta(t)}.
\end{aligned}
\end{equation}
The fields and their derivatives in
\eqref{eq:localized-field-time-bounds} extend continuously to \(t=0\).
With the derivatives at \(t=0\) defined by these limits, the same
bounds hold on \([0,T]\).
\end{lemma}
\begin{proof}
For \(\tau>0\), write
\[
\mathscr E_\tau^a(x)
:=
\frac{1}{4\pi}
E_1\left(\frac{|x-a|^2}{4\tau}\right),
\]
so that
\begin{equation}\label{eq:PsiG-PsiH-tail}
\Psi_G^{a,\nu}(x,t)
=
\Psi_H^a(x)+\mathscr E_{\nu t}^a(x).
\end{equation}

~\\
If
\(x\in\supp\chi(\cdot,t)\), \(|\zeta|\le M\), and
\(0\le\theta\le1\), then
\[
\begin{aligned}
|x-z(t)-\theta\nu\zeta|
&\ge
\dist(z(t),S_E(t))
-\dist(x,S_E(t))
-\nu|\zeta|
\\
&\stackrel{\substack{\eqref{eq:normalized-limiting-separation}\\\eqref{eq:chi-tubes}}}{\ge} 10-3-1=6.
\end{aligned}
\]
Since \(\chi\) has compact space--time support and \(z([0,T])\) is
bounded, we also have
\begin{equation}\label{eq:localized-annulus}
6
\le
|x-z(t)-\theta\nu\zeta|
\lsim{T}1,
\qquad
x\in\supp\chi(\cdot,t).
\end{equation}
The same estimate holds for all centers on the segment joining
\(z(t)+\nu\zeta_1\) and \(z(t)+\nu\zeta_2\).

~\\
The cutoff therefore vanishes near the centers, so the localized fields
are smooth and divergence free. Moreover,
\[
\nabla^\perp\Psi_H^a(x)=K(x-a),
\qquad
\nabla^\perp\mathscr E_\tau^a(x)
=
-K(x-a)e^{-\frac{|x-a|^2}{4\tau}},
\]
and therefore
\[
\nabla^\perp\Psi_G^{a,\nu}(x,t)
=
K(x-a)
\left(
1-e^{-\frac{|x-a|^2}{4\nu t}}
\right)
=
V_G^{a,\nu}(x,t).
\]
On
\[
\{x:\dist(x,S_E(t))\le2\},
\]
we have \(\chi=1\) and \(\nabla\chi=0\).  Hence
\[
\mathcal V_{\nu,\zeta}
=
V_G^{z+\nu\zeta,\nu},
\qquad
\mathcal H_\chi
=
\Hz,
\qquad
\mathcal R_{\nu,\zeta}^{\mathrm{fld}}
=
\frac{V_G^{z+\nu\zeta,\nu}-\Hz}{\nu},
\]
which proves \eqref{eq:localized-fields-equal} for \(t>0\).

~\\
On the region \eqref{eq:localized-annulus},
\begin{equation}\label{eq:localized-log-bound}
\max_{|\alpha|+|\beta|\le6}
\sup_{x\in\supp\chi(\cdot,t)}
\left|
\partial_x^\alpha D_a^\beta\Psi_H^a(x)
\right|
\lsim{T}1.
\end{equation}
For \(\ell\in\{0,1\}\) and
\(|\alpha|+|\beta|\le6\),
\begin{equation}\label{eq:localized-tail-bound}
\sup_{x\in\supp\chi(\cdot,t)}
\left|
\partial_x^\alpha D_a^\beta
\partial_\tau^\ell\mathscr E_\tau^a(x)
\right|
\lsim{T}\tau^{-N}
e^{-\frac{c_{\mathrm{loc}}}{\tau}},
\qquad
0<\tau\le T,
\end{equation}
for some \(N=N(\alpha,\beta,\ell)\) and a constant
\(c_{\mathrm{loc}}>0\).
This follows directly from
\[
E_1'(q)=-q^{-1}e^{-q},
\qquad
\partial_\tau\mathscr E_\tau^a(x)
=
\frac{1}{4\pi\tau}
e^{-\frac{|x-a|^2}{4\tau}},
\]
and the lower bound \(|x-a|\ge6\).
For every
\(m\ge0\) and \(\delta>0\),
\begin{equation}\label{eq:localized-flatness}
\tau^{-m}e^{-\delta/\tau}
\lsim{m,\delta,T}1,
\qquad
\tau^{-m}e^{-\delta/\tau}\to0
\quad\text{as }\tau\to0.
\end{equation}
When \(\tau=\nu t\),
\begin{equation}\label{eq:localized-nu-flatness}
\nu^{-1}\tau^{-m}e^{-\delta/\tau}
=
t\tau^{-m-1}e^{-\delta/\tau}
\le
T\tau^{-m-1}e^{-\delta/\tau}.
\end{equation}

~\\
By Leibniz' rule,
\begin{equation}\label{eq:localized-product-estimate}
\left\|
\nabla^\perp(\chi F)
\right\|_{W^{k,\infty}}
\lsim{k,T}
\max_{|\alpha|\le k+1}
\sup_{x\in\supp\chi(\cdot,t)}
|\partial_x^\alpha F(x,t)|.
\end{equation}
By \eqref{eq:localized-log-bound}, \eqref{eq:localized-tail-bound},
and \eqref{eq:localized-product-estimate},
\[
\|\mathcal V_{\nu,\zeta}(t)\|_{W^{5,\infty}}
+
\|\mathcal H_\chi(t)\|_{W^{5,\infty}}
\lsim{T}1.
\]

~\\
For the normalized difference, write
\[
\frac{
\Psi_G^{z+\nu\zeta,\nu}-\Psi_H^z
}{\nu}
=
A_{\nu,\zeta}+B_{\nu,\zeta},
\]
where
\[
A_{\nu,\zeta}
:=
\frac{
\Psi_H^{z+\nu\zeta}-\Psi_H^z
}{\nu},
\qquad
B_{\nu,\zeta}
:=
\frac{\mathscr E_{\nu t}^{z+\nu\zeta}}{\nu}.
\]
By the fundamental theorem of calculus,
\[
A_{\nu,\zeta}
=
\int_0^1
D_a\Psi_H^{z+\theta\nu\zeta}\cdot\zeta\,d\theta.
\]
Thus
\[
\|A_{\nu,\zeta}\|_{W^{5,\infty}(\supp\chi)}
\stackrel{\eqref{eq:localized-log-bound}}{\lsim{T}}|\zeta|,
\]
while \eqref{eq:localized-tail-bound}, \eqref{eq:localized-flatness},
and \eqref{eq:localized-nu-flatness} give
\[
\|B_{\nu,\zeta}\|_{W^{5,\infty}(\supp\chi)}
\lsim{T}1.
\]
Since
\[
\mathcal R_{\nu,\zeta}^{\mathrm{fld}}
=
\nabla^\perp
\left[
\chi(A_{\nu,\zeta}+B_{\nu,\zeta})
\right],
\]
we obtain
\[
\|\mathcal R_{\nu,\zeta}^{\mathrm{fld}}(t)\|_{W^{4,\infty}}
\stackrel{\eqref{eq:localized-product-estimate}}{\lsim{T}}
1+|\zeta|,
\]
which proves \eqref{eq:localized-field-bounds}.

~\\
The Lipschitz estimates follow in the same way.  If
\[
\zeta_\theta
=
(1-\theta)\zeta_2+\theta\zeta_1,
\]
then
\[
\Psi_G^{z+\nu\zeta_1,\nu}
-
\Psi_G^{z+\nu\zeta_2,\nu}
=
\nu
\int_0^1
D_a\Psi_G^{z+\nu\zeta_\theta,\nu}
\cdot(\zeta_1-\zeta_2)\,d\theta.
\]
By \eqref{eq:localized-log-bound},
\eqref{eq:localized-tail-bound}, and
\eqref{eq:localized-product-estimate},
\[
\|\mathcal V_{\nu,\zeta_1}
-\mathcal V_{\nu,\zeta_2}\|_{W^{4,\infty}}
\lsim{T}
\nu|\zeta_1-\zeta_2|.
\]
Since \(\mathcal H_\chi\) is independent of \(\zeta\),
\[
\mathcal R_{\nu,\zeta_1}^{\mathrm{fld}}
-
\mathcal R_{\nu,\zeta_2}^{\mathrm{fld}}
=
\nu^{-1}
\left(
\mathcal V_{\nu,\zeta_1}
-
\mathcal V_{\nu,\zeta_2}
\right),
\]
and therefore
\[
\|\mathcal R_{\nu,\zeta_1}^{\mathrm{fld}}
-\mathcal R_{\nu,\zeta_2}^{\mathrm{fld}}\|_{W^{4,\infty}}
\lsim{T}|\zeta_1-\zeta_2|.
\]
This is slightly stronger than
\eqref{eq:localized-field-Lipschitz}.

~\\
It remains to estimate time derivatives.  Let
\[
a(t)=z(t)+\nu\zeta(t),
\qquad
\tau=\nu t.
\]
Then
\[
\frac d{dt}\Psi_G^{a(t),\nu}
=
D_a\Psi_G^{a(t),\nu}
\cdot(\dot z+\nu\dot\zeta)
+
\nu\partial_\tau\mathscr E_\tau^{a(t)}.
\]
Since \(\dot z\) is bounded, the preceding estimates imply
\[
\left\|
\frac d{dt}\mathcal V_{\nu,\zeta(t)}(t)
\right\|_{W^{4,\infty}}
+
\|\partial_t\mathcal H_\chi(t)\|_{W^{4,\infty}}
\lsim{T}1+\nu|\dot\zeta(t)|.
\]

~\\
For the normalized field, the representation above gives
\[
A_{\nu,\zeta(t)}
=
\int_0^1
D_a\Psi_H^{z+\theta\nu\zeta(t)}
\cdot\zeta(t)\,d\theta,
\]
and hence
\[
\left\|
\frac d{dt}A_{\nu,\zeta(t)}
\right\|_{W^{4,\infty}(\supp\chi)}
\lsim{T}|\zeta(t)|+|\dot\zeta(t)|,
\]
because \(\nu|\zeta(t)|\le1\).
Moreover,
\[
B_{\nu,\zeta(t)}
=
\nu^{-1}\mathscr E_{\nu t}^{a(t)},
\]
so
\[
\frac d{dt}B_{\nu,\zeta(t)}
=
\nu^{-1}
D_a\mathscr E_{\nu t}^{a(t)}\cdot\dot z
+
D_a\mathscr E_{\nu t}^{a(t)}\cdot\dot\zeta
+
\partial_\tau\mathscr E_{\nu t}^{a(t)}.
\]
By \eqref{eq:localized-tail-bound}, \eqref{eq:localized-flatness},
and \eqref{eq:localized-nu-flatness},
\[
\left\|
\frac d{dt}B_{\nu,\zeta(t)}
\right\|_{W^{4,\infty}(\supp\chi)}
\lsim{T}
1+|\dot\zeta(t)|.
\]
Applying Leibniz' rule to
\[
\mathcal R_{\nu,\zeta(t)}^{\mathrm{fld}}
=
\nabla^\perp
\left[
\chi(A_{\nu,\zeta(t)}+B_{\nu,\zeta(t)})
\right]
\]
yields
\[
\left\|
\frac d{dt}
\mathcal R_{\nu,\zeta(t)}^{\mathrm{fld}}(t)
\right\|_{W^{3,\infty}}
\lsim{T}
1+|\zeta(t)|+|\dot\zeta(t)|,
\]
which proves \eqref{eq:localized-field-time-bounds} for \(t>0\).

~\\
Finally, the Gaussian terms above, including those divided by \(\nu\),
are bounded in the stated norms by finite sums of terms of the form
\[
\max\left\{\tau^{-m}e^{-\delta/\tau},\,
\nu^{-1}\tau^{-m}e^{-\delta/\tau}\right\}
\stackrel{\substack{\eqref{eq:localized-flatness}\\\eqref{eq:localized-nu-flatness}}}{\to}0,
\qquad \tau=\nu t\to0.
\]
The remaining terms involve the point-vortex stream functions away from
their centers, \(\chi\), \(z\), and \(\zeta\), and are continuous at
\(t=0\). This gives the stated traces and bounds at \(t=0\).
\end{proof}
\subsection{Local well-posedness and continuation}

Introduce the rescaled center displacement
\[
\zeta(t):=\frac{\Zn(t)-z(t)}{\nu},
\qquad
v^w:=K\star w.
\]
Define
\begin{equation}\label{eq:Anu-definition}
A_\nu^E(t,\zeta)
:=
\int_0^1
D_xv^E\bigl(t,z(t)+\theta\nu\zeta\bigr)\,d\theta.
\end{equation}
The fundamental theorem of calculus gives
\[
v^E(t,z+\nu\zeta)-v^E(t,z)
=
\nu A_\nu^E(t,\zeta)\zeta.
\]
Thus the center equations
\(\partial_t\Zn=v^E(t,\Zn)+\nu v^w(t,\Zn)\) and
\(\partial_tz=v^E(t,z)\) are equivalent to
\begin{equation}\label{eq:zeta-equation}
\partial_t\zeta
=
A_\nu^E(t,\zeta)\zeta
+
v^w\bigl(t,z(t)+\nu\zeta\bigr).
\end{equation}

~\\
Fix \(M\ge1\) with \(\nu M\le1\).  As long as
\(\abs{\zeta(t)}<M\), consider the localized system
\begin{equation}\label{eq:auxiliary-corrector-system}
\left\{
\begin{aligned}
\partial_tw
&+
(v^E+\mathcal V_{\nu,\zeta})\cdot\nabla w
+
v^w\cdot\nabla\w^E
\\
&=
\Delta\w^E
-
\mathcal R_{\nu,\zeta}^{\mathrm{fld}}\cdot\nabla\w^E,
\qquad
v^w=K\star w,
\\
\partial_t\zeta
&=
A_\nu^E(t,\zeta)\zeta
+
v^w\bigl(t,z(t)+\nu\zeta\bigr),
\\
w(0)&=0,
\qquad
\zeta(0)=0.
\end{aligned}
\right.
\end{equation}
For this subsection, set
\begin{equation}\label{eq:corrector-Sobolev-spaces}
X:=W^{4,4}(\Rt)\cap W^{3,1}(\Rt),
\qquad
X_{\mathrm{low}}:=W^{2,4}(\Rt)\cap W^{1,1}(\Rt),
\end{equation}
with the sum norms.

\begin{lemma}\label{lem:corrector-local-theory}
Let \(M\ge1\), \(0<\nu\le1\), and \(\nu M\le1\).
There exists \(\tau=\tau(M,T)\in(0,T]\), independent of \(\nu\),
such that \eqref{eq:auxiliary-corrector-system} has a unique solution
\[
w\in C([0,\tau];X),
\qquad
\zeta\in C^1([0,\tau];\Rt),
\]
satisfying
\[
\sup_{0\le t\le\tau}
\bigl(\norm[X]{w(t)}+\abs{\zeta(t)}\bigr)
\le M.
\]
Moreover, throughout any interval of existence on which
\(\abs{\zeta}<M\),
\begin{equation}\label{eq:corrector-local-apriori}
\sup_{0\le s\le t}
\bigl(\norm[X]{w(s)}+\abs{\zeta(s)}\bigr)
\le e^{c_{\mathrm{Gr},T}t}-1,
\end{equation}
where \(c_{\mathrm{Gr},T}>0\) is independent of \(M\) and \(\nu\).
\end{lemma}

\begin{proof}
\emph{A priori estimates.}
For the spaces in \eqref{eq:corrector-Sobolev-spaces},
\[
\begin{aligned}
\norm[W^{5,\infty}]{v^E+\mathcal V_{\nu,\zeta}}
&\stackrel{\eqref{eq:localized-field-bounds}}{\lsim{T}}1,
\\
\norm[W^{4,\infty}]{
\mathcal R_{\nu,\zeta}^{\mathrm{fld}}
}
&\stackrel{\eqref{eq:localized-field-bounds}}{\lsim{T}}1+\abs{\zeta}.
\end{aligned}
\]
Since \(\w^E\) is smooth and supported in a fixed compact set,
\[
\begin{aligned}
&\left\|
\Delta\w^E
-
\mathcal R_{\nu,\zeta}^{\mathrm{fld}}\cdot\nabla\w^E
-
v^w\cdot\nabla\w^E
\right\|_X
\\
&\qquad\stackrel{\substack{\eqref{eq:localized-field-bounds}\\\eqref{eq:nonlocal-product-high}}}{\lsim{T}}
1+\abs{\zeta}+\norm[X]{w}.
\end{aligned}
\]
Therefore,
\begin{equation}\label{eq:corrector-wave-apriori}
\frac d{dt}\norm[X]{w(t)}
\stackrel{\eqref{eq:transport-high}}{\lsim{T}}
1+\norm[X]{w(t)}+\abs{\zeta(t)}.
\end{equation}
Similarly, the center equation \eqref{eq:zeta-equation},
the definition \eqref{eq:Anu-definition}, and the Biot--Savart bound
\eqref{eq:BS-high} give
\begin{equation}\label{eq:corrector-center-apriori}
\abs{\partial_t\zeta(t)}
\lsim{T}
\abs{\zeta(t)}+\norm[X]{w(t)}.
\end{equation}
Adding these inequalities gives, for almost every time,
\[
\frac d{dt}
\bigl(\norm[X]{w(t)}+\abs{\zeta(t)}\bigr)
\lsim{T}
1+\norm[X]{w(t)}+\abs{\zeta(t)}.
\]
Since \(w(0)=0\) and \(\zeta(0)=0\), Gronwall's inequality proves
\eqref{eq:corrector-local-apriori}.

\smallskip
\emph{Local existence and uniqueness.}
For a prescribed continuous path \(\eta\) with
\(\abs{\eta}\le M\), the first equation of
\eqref{eq:auxiliary-corrector-system}, with \(\zeta=\eta\), is a
linear transport equation. Its velocity is uniformly Lipschitz,
and the nonlocal operator
\[
f\longmapsto (K\star f)\cdot\nabla\w^E
\]
is bounded on \(X\) by \eqref{eq:nonlocal-product-high}.
The characteristic formula therefore gives a unique solution
\(w[\eta]\in C([0,\tau];X)\), with bounds uniform in \(\nu\).

~\\
For two prescribed paths \(\eta_1,\eta_2\), the difference equation
contains the additional transport term
\[
(\mathcal V_{\nu,\eta_1}-\mathcal V_{\nu,\eta_2})
\cdot\nabla w[\eta_2].
\]
In \(X_{\mathrm{low}}\),
\[
\begin{aligned}
&\norm[X_{\mathrm{low}}]{
(\mathcal V_{\nu,\eta_1}-\mathcal V_{\nu,\eta_2})
\cdot\nabla w[\eta_2]}
\\
&\qquad\stackrel{\eqref{eq:localized-field-Lipschitz}}{\lsim{T}}
\nu\abs{\eta_1-\eta_2}\norm[X]{w[\eta_2]}.
\end{aligned}
\]
By \eqref{eq:localized-field-Lipschitz} and
\eqref{eq:nonlocal-product-low}, followed by the transport estimate
\eqref{eq:transport-low} and Gronwall's inequality,
\[
\begin{aligned}
&\sup_{0\le s\le t}
\norm[X_{\mathrm{low}}]{w[\eta_1](s)-w[\eta_2](s)}
\\
&\qquad\lsim{M,T}
t
\sup_{0\le s\le t}\abs{\eta_1(s)-\eta_2(s)}.
\end{aligned}
\]
The bound \eqref{eq:BS-low} and the mean value theorem show that the
right-hand side of the center equation is locally Lipschitz in
\(X_{\mathrm{low}}\times\Rt\), uniformly in \(\nu\).
Substituting \(w[\eta]\) into the integral form of that equation
therefore defines a contraction on a sufficiently short interval
among continuous paths satisfying
\(\eta(0)=0\) and \(\sup_t\abs{\eta(t)}\le M\).
This proves local existence and uniqueness, with a lifespan
independent of \(\nu\).
The center equation gives \(\zeta\in C^1\), while the linear
transport construction gives \(w\in C([0,\tau];X)\).
Finally, decrease \(\tau\) so that
\(e^{c_{\mathrm{Gr},T}\tau}-1\le M/2\).  Then
\[
\sup_{0\le s\le\tau}
\bigl(\norm[X]{w(s)}+\abs{\zeta(s)}\bigr)
\stackrel{\eqref{eq:corrector-local-apriori}}{\le}
e^{c_{\mathrm{Gr},T}\tau}-1\le M/2,
\]
which gives the stated local bound.
\end{proof}

~\\
For later use, denote the transport flow of the auxiliary solution by
\begin{equation}\label{eq:localized-flow-ODE}
\frac d{dr}\Phi^\nu_{r,s}(x)
=
\bigl(v^E+\mathcal V_{\nu,\zeta(r)}\bigr)
\bigl(\Phi^\nu_{r,s}(x),r\bigr),
\qquad
\Phi^\nu_{s,s}(x)=x,
\end{equation}
and set
\begin{equation}\label{eq:localized-evolution-family}
U(t,s)f:=f\circ\Phi^\nu_{s,t},
\qquad
0\le s\le t\le\tau.
\end{equation}

\begin{corollary}\label{cor:corrector-continuation}
Let \(t_0<T\) belong to the interval of existence, and suppose that
\[
\norm[X]{w(t_0)}\le R,
\qquad
\abs{\zeta(t_0)}\le M-\eta,
\qquad
R\ge1,\quad \eta>0.
\]
Then the solution extends to
\[
[t_0,\min\{t_0+\sigma,T\}],
\]
where \(\sigma=\sigma(R,M,\eta,T)>0\) is independent of
\(t_0\) and \(\nu\), provided \(\nu M\le1\).
Consequently, if the maximal existence time \(t_*<T\), then
\[
\limsup_{t\to t_*^-}\norm[X]{w(t)}=\infty
\quad\text{or}\quad
\limsup_{t\to t_*^-}\abs{\zeta(t)}=M.
\]
\end{corollary}

\begin{proof}
Restart the preceding local construction at \(t_0\), keeping the
original time-dependent coefficients.
The transport and difference estimates are uniform on \([0,T]\).
Moreover, \eqref{eq:corrector-wave-apriori} and
\eqref{eq:corrector-center-apriori} control the Sobolev norm and the
center displacement on a time interval depending only on
\(R,M,\eta,T\). Choosing this interval sufficiently short keeps
\(\abs{\zeta}<M\) and gives the stated extension.
If neither alternative occurred at \(t_*<T\), the same argument
would provide a uniform positive extension time from times
approaching \(t_*\), contradicting maximality.
\end{proof}

\subsection{Support persistence and removal of the localization}

\begin{proof}[Proof of Proposition~\ref{prop:regular-corrector}]
\noindent\emph{Step 1: uniform a priori bounds.}
Let \((w,\zeta)\) be the local solution of
\eqref{eq:auxiliary-corrector-system}, and set
\(\Zn=z+\nu\zeta\).
As long as \(\abs{\zeta}<M\) and \(\nu M\le1\),
\begin{align*}
&\frac{d}{dt}\Bigl(\norm[W^{4,4}]{w} + \norm[W^{3,1}]{w} + \abs{\zeta}\Bigr)
\\
&\qquad\stackrel{\substack{\eqref{eq:corrector-wave-apriori}\\\eqref{eq:corrector-center-apriori}}}{\lsim{T}}
1 + \norm[W^{4,4}]{w} + \norm[W^{3,1}]{w} + \abs{\zeta}.
\end{align*}
The implicit constant is independent of \(M\) and \(\nu\) under
\(\nu M\le1\).
Gronwall's inequality and the zero initial data therefore bound
the displayed sum uniformly on the interval of existence.
Choose \(M\) larger than twice this bound, and then choose
\(\nu_{0,T}\) so that \(\nu_{0,T}M\le1\).
Corollary~\ref{cor:corrector-continuation} extends the solution
to \([0,T]\), with
\begin{equation}\label{eq:auxiliary-uniform-bound}
\sup_{0\le t\le T}
\left(
\norm[W^{4,4}]{w(t)}
+\norm[W^{3,1}]{w(t)}
+\abs{\zeta(t)}
\right)
\lsim{T}1.
\end{equation}
The center equation and the Biot--Savart bound also give
\(\abs{\partial_t\zeta(t)}\lsim{T}1\).
Since \(\zeta(0)=0\),
\begin{equation}\label{eq:zeta-linear-time}
\abs{\zeta(t)}\lsim{T}t,
\qquad
\abs{\Zn(t)-z(t)}
=\nu\abs{\zeta(t)}
\lsim{T}\nu t.
\end{equation}

\smallskip
\noindent\emph{Step 2: comparison of the two flows.}
The limiting flow transports the wave support exactly:
\begin{equation}\label{eq:Euler-support-flow}
\Phi_{t,s}\bigl(S_E(s)\bigr)=S_E(t),
\qquad 0\le s\le t\le T.
\end{equation}
Let \(\Phi^\nu_{r,s}\) be the flow of the auxiliary velocity,
defined for \(0\le r,s\le T\) by
\begin{equation}\label{eq:localized-flow-global-ODE}
\frac{d}{dr}\Phi^\nu_{r,s}(x)
=
\bigl(v^E+\mathcal V_{\nu,\zeta(r)}\bigr)
\bigl(\Phi^\nu_{r,s}(x),r\bigr),
\qquad
\Phi^\nu_{s,s}(x)=x.
\end{equation}

~\\
We compare the smooth fields
\(v^E+\mathcal V_{\nu,\zeta(t)}\) and
\(v^E+\mathcal H_\chi\).
Indeed,
\[
\norm[L^\infty]{
D_x\bigl(v^E+\mathcal H_\chi\bigr)(t)}
\stackrel{\eqref{eq:localized-field-bounds}}{\lsim{T}}1
\]
and
\[
\begin{aligned}
\norm[L^\infty]{
\mathcal V_{\nu,\zeta(t)}(t)-\mathcal H_\chi(t)}
&\stackrel{\eqref{eq:localized-fields}}{=}
\nu\norm[L^\infty]{
\mathcal R_{\nu,\zeta(t)}^{\mathrm{fld}}(t)}
\\
&\stackrel{\substack{\eqref{eq:localized-field-bounds}\\\eqref{eq:auxiliary-uniform-bound}}}{\lsim{T}}\nu,
\qquad 0\le t\le T.
\end{aligned}
\]

~\\
Fix \(s\in[0,T]\) and \(y\in S_E(s)\), and write
\[
Y^\nu(r):=\Phi^\nu_{r,s}(y),
\qquad
Y(r):=\Phi_{r,s}(y),
\qquad s\le r\le T.
\]
The limiting trajectory satisfies
\[
Y(r)=\Phi_{r,s}(y)
\stackrel{\eqref{eq:Euler-support-flow}}{\in}S_E(r).
\]
Hence
\[
\dot Y(r)
\stackrel{\substack{\eqref{eq:Euler-flow-ODE}\\\eqref{eq:localized-fields-equal}}}{=}
\bigl(v^E+\mathcal H_\chi\bigr)(Y(r),r).
\]
Subtracting this equation from the equation for \(Y^\nu\) gives
\[
\begin{aligned}
\dot Y^\nu-\dot Y
={}&
\bigl(v^E+\mathcal H_\chi\bigr)(Y^\nu,r)
-\bigl(v^E+\mathcal H_\chi\bigr)(Y,r)
\\
&+
\bigl(\mathcal V_{\nu,\zeta(r)}-\mathcal H_\chi\bigr)
(Y^\nu,r).
\end{aligned}
\]
Since \(Y^\nu(s)=Y(s)=y\), the preceding coefficient bounds imply
\[
\abs{Y^\nu(r)-Y(r)}
\lsim{T}
\int_s^r\abs{Y^\nu(q)-Y(q)}\,dq
+\nu(r-s).
\]
Gronwall's inequality yields
\(\abs{Y^\nu(t)-Y(t)}\lsim{T}\nu(t-s)\), and therefore
\begin{equation}\label{eq:uniform-flow-O-nu}
\sup_{\substack{0\le s\le t\le T\\ y\in S_E(s)}}
\abs{\Phi^\nu_{t,s}(y)-\Phi_{t,s}(y)}
\lsim{T}\nu.
\end{equation}
The coefficient bounds are global, so no localization of \(Y^\nu\)
is needed.

\smallskip
\noindent\emph{Step 3: support localization and recovery of the original system.}
Write the auxiliary equation as
\[
\partial_tw
+\bigl(v^E+\mathcal V_{\nu,\zeta(t)}\bigr)\cdot\nabla w
=F_w,
\]
where
\[
F_w
:=
\Delta\w^E
-\bigl(
\mathcal R_{\nu,\zeta(t)}^{\mathrm{fld}}+v^w
\bigr)\cdot\nabla\w^E.
\]
Every term in \(F_w(t)\) is supported in \(S_E(t)\).

~\\
Since \(w(0)=0\), the characteristic formula gives
\begin{equation}\label{eq:corrector-characteristic-formula}
w(x,t)
=
\int_0^t
F_w\bigl(\Phi^\nu_{s,t}(x),s\bigr)\,ds.
\end{equation}
Consequently,
\begin{equation}\label{eq:corrector-support-preliminary}
\supp w(t)
\subset
\overline{
\bigcup_{0\le s\le t}
\Phi^\nu_{t,s}\bigl(S_E(s)\bigr)
}.
\end{equation}
For every \(y\in S_E(s)\),
\[
\begin{aligned}
\dist\bigl(\Phi^\nu_{t,s}(y),S_E(t)\bigr)
&\stackrel{\eqref{eq:Euler-support-flow}}{\le}
\abs{\Phi^\nu_{t,s}(y)-\Phi_{t,s}(y)}
\\
&\stackrel{\eqref{eq:uniform-flow-O-nu}}{\lsim{T}}\nu.
\end{aligned}
\]
Together with \eqref{eq:corrector-support-preliminary}, this gives
\begin{equation}\label{eq:w1-support-tube}
d_w(t):=\sup_{x\in\supp w(t)}\dist(x,S_E(t))\lsim{T}\nu,
\qquad 0\le t\le T.
\end{equation}

~\\
Decrease \(\nu_{0,T}\) so that
\(\sup_{0\le t\le T}d_w(t)<2\) for every \(0<\nu\le\nu_{0,T}\).
Then both \(\nabla w(t)\) and \(\nabla\w^E(t)\) are supported
where the localized fields agree with the original ones.
Consequently,
\[
\bigl(
\mathcal V_{\nu,\zeta(t)}-V_G^{\Zn,\nu}
\bigr)\cdot\nabla w
\stackrel{\eqref{eq:localized-fields-equal}}{=}0
\]
and
\[
\left(
\mathcal R_{\nu,\zeta(t)}^{\mathrm{fld}}
-\frac{V_G^{\Zn,\nu}-\Hz}{\nu}
\right)\cdot\nabla\w^E
\stackrel{\eqref{eq:localized-fields-equal}}{=}0
\]
globally for \(t>0\), with the products extended by zero at
the singular centers.
Together with the center equation, these identities show that
\[
w_{1,a}:=w,
\qquad
v_{1,a}:=v^w,
\qquad
\Zn=z+\nu\zeta
\]
satisfy \eqref{eq:regular-corrector-system} exactly.

~\\
Moreover,
\(\supp\wapp(t)\subset S_E(t)\cup\supp w(t)\).
Therefore,
\[
\begin{aligned}
\dist\bigl(\Zn(t),\supp\wapp(t)\bigr)
&\stackrel{\eqref{eq:w1-support-tube}}{\ge}
\dist\bigl(z(t),S_E(t)\bigr)
-\abs{\Zn(t)-z(t)}-d_w(t)
\\
&\stackrel{\eqref{eq:normalized-limiting-separation}}{\ge}
10-\abs{\Zn(t)-z(t)}-d_w(t)
\ge8
\end{aligned}
\]
after decreasing \(\nu_{0,T}\) once more so that
\(\sup_{0\le t\le T}(\abs{\Zn(t)-z(t)}+d_w(t))\le2\),
by \eqref{eq:zeta-linear-time} and \eqref{eq:w1-support-tube}.
This proves \eqref{eq:approx-separation}.

\smallskip
\noindent\emph{Step 4: the remaining bounds and identities.}
The spatial bounds in
\eqref{eq:regular-corrector-bounds} follow from
\eqref{eq:auxiliary-uniform-bound} and
Lemma~\ref{lem:corrector-estimates}.
For the time derivatives,
\begin{equation}\label{eq:corrector-time-first}
\partial_tw
\stackrel{\eqref{eq:auxiliary-corrector-system}}{=}
F_w
-\bigl(v^E+\mathcal V_{\nu,\zeta(t)}\bigr)\cdot\nabla w.
\end{equation}
The spatial bounds and the Biot--Savart estimates imply
\[
\sup_{0\le t\le T}
\left(
\norm[W^{2,4}\cap W^{2,1}]{\partial_tw(t)}
+\norm[W^{2,\infty}]{\partial_tv^w(t)}
\right)
\lsim{T}1.
\]
Since \(\abs{\partial_t\zeta}\lsim{T}1\),
\begin{equation}\label{eq:corrector-field-time-bounds}
\begin{aligned}
&
\left\|
\frac{d}{dt}
\bigl(v^E+\mathcal V_{\nu,\zeta(t)}\bigr)(t)
\right\|_{W^{4,\infty}}
\\
&\qquad+
\left\|
\frac{d}{dt}
\mathcal R_{\nu,\zeta(t)}^{\mathrm{fld}}(t)
\right\|_{W^{3,\infty}}
\stackrel{\eqref{eq:localized-field-time-bounds}}{\lsim{T}}1.
\end{aligned}
\end{equation}
Here the derivatives of the localized fields are taken along
the path \(\zeta(t)\).

~\\
The formula for \(F_w\) now gives
\(\norm[L^4\cap L^1]{\partial_tF_w(t)}\lsim{T}1\).
Differentiation in time gives
\begin{equation}\label{eq:corrector-time-second}
\begin{aligned}
\partial_t^2w
\stackrel{\eqref{eq:corrector-time-first}}{=}&
\partial_tF_w
-\bigl(v^E+\mathcal V_{\nu,\zeta(t)}\bigr)
\cdot\nabla\partial_tw
\\
&-
\frac{d}{dt}
\bigl(v^E+\mathcal V_{\nu,\zeta(t)}\bigr)
\cdot\nabla w,
\end{aligned}
\end{equation}
and therefore
\[
\sup_{0\le t\le T}
\norm[L^4\cap L^1]{\partial_t^2w(t)}
\stackrel{\substack{\eqref{eq:corrector-time-second}\\\eqref{eq:corrector-field-time-bounds}}}{\lsim{T}}1.
\]
By Lemma~\ref{lem:divfree-localization}, all terms in these identities
are continuous in the indicated spaces up to \(t=0\).
Thus they also give the claimed \(C^1\) and \(C^2\) regularity
of \(w\).

~\\
Differentiating the center equation gives
\[
\partial_t^2\Zn(t)
=
(\partial_t\vapp)(\Zn(t),t)
+
D_x\vapp(\Zn(t),t)\,\partial_t\Zn(t).
\]
The preceding bounds imply
\(\sup_{0\le t\le T}\abs{\partial_t^2\Zn(t)}\lsim{T}1\)
and \(\Zn\in C^2([0,T];\Rt)\).
For each fixed \(\nu>0\), higher-order transport estimates
and differentiation of the coupled center equation give
\(w_{1,a}\in C^\infty((0,T]\times\Rt)\).
This proves \eqref{eq:regular-corrector-bounds} and the
asserted regularity.

~\\
Finally, \(w(t)\) is compactly supported by
\eqref{eq:w1-support-tube}.
Integration of the auxiliary equation and the divergence-free property
of its velocity fields give
\[
\frac{d}{dt}\int_{\Rt}w(x,t)\,dx
\stackrel{\eqref{eq:auxiliary-corrector-system}}{=}0.
\]
Since \(w(0)=0\), this proves \eqref{eq:w1a-zero-mass}.
Moreover,
\[
\begin{aligned}
&\partial_t\wapp
+(\vapp+V_G^{\Zn,\nu})\cdot\nabla\wapp
-\nu\Delta\wapp
\\
&\qquad\stackrel{\substack{\eqref{eq:regular-defect-general}\\\eqref{eq:regular-corrector-system}}}{=}
\nu^2\left(
v_{1,a}\cdot\nabla w_{1,a}
-\Delta w_{1,a}
\right),
\end{aligned}
\]
which is \eqref{eq:regular-approx-residual}.
\end{proof}
\section{Vortex variables and expansion of the ambient drift}
\label{sec:b-exp}

~\\
Fix \(T>0\) and \(0<\nu\le\nu_{0,T}\), where \(\nu_{0,T}\) is
given by Proposition~\ref{prop:regular-corrector}.
The modulated center and the approximate regular vorticity satisfy
the separation and uniform bounds established there.
All implicit constants below are independent of \(\nu\).

~\\
We introduce the vortex variables
\begin{equation}\label{eq:inner-variable}
\xi=\frac{x-\Zn(t)}{\sqrt{\nu t}},
\qquad
w_2(\xi,t)=\nu t\,\w^{B,\nu}(x,t),
\qquad
v_2=K\star_\xi w_2.
\end{equation}

\begin{lemma}
\label{lem:exact-inner-equation}
Let \(V(x,t)\) be a smooth divergence-free field, and suppose that
\(\w^{B,\nu}\) is a smooth integrable solution, at positive times, of
\[
\partial_t\w^{B,\nu}
+\bigl(V+K\star_x\w^{B,\nu}\bigr)
 \cdot\nabla_x\w^{B,\nu}
=
\nu\Delta_x\w^{B,\nu}.
\]
Under \eqref{eq:inner-variable},
\begin{equation}\label{eq:exact-inner-equation}
\begin{aligned}
0={}&
(t\partial_t-\cL)w_2
+\frac1\nu v_2\cdot\nabla_\xi w_2
\\
&+
\sqrt{\frac t\nu}
\left(
V\bigl(\Zn(t)+\sqrt{\nu t}\,\xi,t\bigr)
-\partial_t\Zn(t)
\right)\cdot\nabla_\xi w_2,
\end{aligned}
\end{equation}
where
\(\cL=\Delta_\xi+\frac12\xi\cdot\nabla_\xi+1\), and
\(\partial_t w_2\) is taken at fixed \(\xi\).
\end{lemma}

\begin{proof}
At fixed \(x\),
\[
\partial_t\xi
\stackrel{\eqref{eq:inner-variable}}{=}
-\frac{\partial_t\Zn(t)}{\sqrt{\nu t}}
-\frac{\xi}{2t}.
\]
Since
\(\w^{B,\nu}(x,t)=(\nu t)^{-1}w_2(\xi,t)\), the chain rule gives
\[
\begin{aligned}
\nu t^2\partial_t\w^{B,\nu}
={}&
t\partial_t w_2-w_2
-\frac12\xi\cdot\nabla_\xi w_2
\\
&-
\sqrt{\frac t\nu}\,
\partial_t\Zn(t)\cdot\nabla_\xi w_2.
\end{aligned}
\]
Moreover,
\[
\nabla_x\w^{B,\nu}
=
(\nu t)^{-3/2}\nabla_\xi w_2,
\qquad
\Delta_x\w^{B,\nu}
=
(\nu t)^{-2}\Delta_\xi w_2.
\]
By the homogeneity of the Biot--Savart kernel,
\[
(K\star_x\w^{B,\nu})(x,t)
\stackrel{\eqref{eq:inner-variable}}{=}
\frac1{\sqrt{\nu t}}v_2(\xi,t).
\]
Substituting these identities into the physical equation and
multiplying by \(\nu t^2\) proves
\eqref{eq:exact-inner-equation}.
\end{proof}

~\\
For a profile \(w\) and an ambient velocity \(V\), define
\begin{equation}\label{eq:core-residual-operator}
\begin{aligned}
\Phi(w,V)
:={}&
(t\partial_t-\cL)w
+\frac1\nu(K\star_\xi w)\cdot\nabla_\xi w
\\
&+
\sqrt{\frac t\nu}
\left(
V\bigl(\Zn(t)+\sqrt{\nu t}\,\xi,t\bigr)
-\partial_t\Zn(t)
\right)\cdot\nabla_\xi w.
\end{aligned}
\end{equation}
Thus \(\Phi(w_2,v^{E,\nu})=0\).

\paragraph{The ambient drift and its strain.}

Following the Taylor expansion about the vortex center in
\cite{DonatiGallay2026}, we isolate the symmetric first-order term
of the ambient velocity. Here the ambient field is the approximate
regular velocity constructed in Section~\ref{sec:approx}, whose local
derivative bounds follow from separation.

~\\
Proposition~\ref{prop:regular-corrector} gives, for \(0\le t\le T\),
\begin{equation}\label{eq:b-hyp}
\begin{aligned}
\partial_t\Zn(t)&=\vapp(\Zn(t),t),
\\
\vapp&=K\star_x\wapp,
\\
\dist\bigl(\Zn(t),\supp\wapp(t)\bigr)
&\ge8.
\end{aligned}
\end{equation}
Define
\begin{equation}\label{eq:strain-field}
\begin{aligned}
b(\xi,t)
&:=
\vapp\bigl(\Zn(t)+\sqrt{\nu t}\,\xi,t\bigr)
-\partial_t\Zn(t)
\\
&\stackrel{\eqref{eq:b-hyp}}{=}
\vapp\bigl(\Zn(t)+\sqrt{\nu t}\,\xi,t\bigr)
-\vapp(\Zn(t),t).
\end{aligned}
\end{equation}
In particular, \(b(0,t)=0\), and
\[
\nabla_\xi\cdot b(\xi,t)
=
\sqrt{\nu t}\,
(\nabla_x\cdot\vapp)
\bigl(\Zn(t)+\sqrt{\nu t}\,\xi,t\bigr)
=
0.
\]

~\\
Set
\begin{equation}\label{eq:Sigma-def}
\Sigma(t)
:=
\operatorname{sym}D_x\vapp(\Zn(t),t).
\end{equation}
Since \(\wapp(t)\) vanishes in \(B(\Zn(t),8)\), the velocity
\(\vapp\) is both divergence free and curl free in that ball.
Its Jacobian is therefore symmetric and trace free there. In particular,
\begin{equation}\label{eq:strain-jacobian}
D_x\vapp(\Zn(t),t)=\Sigma(t),
\qquad
\Sigma(t)^T=\Sigma(t),
\qquad
\operatorname{tr}\Sigma(t)=0.
\end{equation}
Proposition~\ref{prop:regular-corrector} also bounds the strain and its
time derivative:
\begin{equation}\label{eq:Sigma-time-bound}
\sup_{0\le t\le T}
\left(
\abs{\Sigma(t)}
+\abs{\partial_t\Sigma(t)}
\right)
\lesssim 1.
\end{equation}

\paragraph{Multipole expansion.}

For \(\xi\neq0\) and \(y\neq\Zn(t)\), let
\(\psi=\psi(\xi,y,t)\) be the oriented angle from
\(\Zn(t)-y\) to \(\xi\):
\begin{equation}\label{eq:psi-def}
\cos\psi
=
\frac{\xi\cdot(\Zn(t)-y)}
     {\abs{\xi}\abs{\Zn(t)-y}},
\qquad
\sin\psi
=
\frac{\xi\cdot(\Zn(t)-y)^\perp}
     {\abs{\xi}\abs{\Zn(t)-y}}.
\end{equation}
At \(\xi=0\), every expression of the form
\(\abs{\xi}^m\sin(m\psi)\) is understood by continuous extension,
with value zero. The kernels below are evaluated on
\(\supp\wapp(t)\), where \(\abs{\Zn(t)-y}\ge8\).

\begin{lemma}
\label{lem:b-exp}
For every \(0<t\le T\) and every \(\xi\in\Rt\) satisfying
\(\sqrt{\nu t}\abs{\xi}\le1\),
\begin{equation}\label{eq:b-multipole}
\begin{aligned}
b(\xi,t)\cdot\xi
={}&
-\frac{\sqrt{\nu t}\abs{\xi}^2}{2\pi}
\int_{\Rt}
\frac{\sin(2\psi)}{\abs{\Zn(t)-y}^2}
\wapp(y,t)\,dy
\\
&+
\frac{\nu t\abs{\xi}^3}{2\pi}
\int_{\Rt}
\frac{\sin(3\psi)}{\abs{\Zn(t)-y}^3}
\wapp(y,t)\,dy
\\
&-
\frac{(\nu t)^{3/2}\abs{\xi}^4}{2\pi}
\int_{\Rt}
\frac{\sin(4\psi)}{\abs{\Zn(t)-y}^4}
\wapp(y,t)\,dy
+r_b(\xi,t),
\end{aligned}
\end{equation}
where
\begin{equation}\label{eq:rb-explicit-tail}
r_b(\xi,t)
:=
\frac1{2\pi}
\sum_{m=5}^{\infty}
(-1)^{m+1}(\nu t)^{(m-1)/2}\abs{\xi}^m
\int_{\Rt}
\frac{\sin(m\psi)}{\abs{\Zn(t)-y}^m}
\wapp(y,t)\,dy.
\end{equation}
Moreover, the quadratic coefficient satisfies
\begin{equation}\label{eq:b-strain-id}
-\frac{\abs{\xi}^2}{2\pi}
\int_{\Rt}
\frac{\sin(2\psi)}{\abs{\Zn(t)-y}^2}
\wapp(y,t)\,dy
=
\xi\cdot\Sigma(t)\xi.
\end{equation}
\end{lemma}

\begin{proof}
Fix \(t>0\). The case \(\xi=0\) follows from the conventions above,
so assume \(\xi\neq0\).

\smallskip
\noindent\emph{Multipole expansion.}
The drift has the integral representation
\begin{equation}\label{eq:drift-Biot-Savart-difference}
b(\xi,t)
\stackrel{\substack{\eqref{eq:b-hyp}\\\eqref{eq:strain-field}}}{=}
\int_{\Rt}
\left[
K\bigl(\Zn(t)-y+\sqrt{\nu t}\,\xi\bigr)
-K(\Zn(t)-y)
\right]
\wapp(y,t)\,dy.
\end{equation}
Since \(\xi\cdot\xi^\perp=0\), its radial component is
\begin{equation}\label{eq:b-radial}
\begin{aligned}
b(\xi,t)\cdot\xi
\stackrel{\substack{\eqref{eq:NS}\\
\eqref{eq:drift-Biot-Savart-difference}}}{=}{}&
\frac1{2\pi}
\int_{\Rt}
\xi\cdot(\Zn(t)-y)^\perp
\\
&\quad\times
\left(
\frac1{\abs{\Zn(t)-y+\sqrt{\nu t}\,\xi}^2}
-\frac1{\abs{\Zn(t)-y}^2}
\right)
\wapp(y,t)\,dy.
\end{aligned}
\end{equation}

~\\
For \(z_1,z_2\neq0\), \(\abs{z_1}\le\frac12\abs{z_2}\), and
\(\varphi\) the angle from \(z_2\) to \(z_1\), one has, initially
for \(\sin\varphi\neq0\),
\begin{equation}\label{eq:kernel-exp}
\frac1{\abs{z_1+z_2}^2}
-\frac1{\abs{z_2}^2}
=
\frac1{\abs{z_2}^2}
\sum_{n=1}^{\infty}
(-1)^n
\frac{\abs{z_1}^n}{\abs{z_2}^n}
\frac{\sin((n+1)\varphi)}{\sin\varphi}.
\end{equation}
The quotient is understood by continuous extension when
\(\sin\varphi=0\); when \(z_1=0\), both sides are zero.
For \(y\in\supp\wapp(t)\),
\[
\frac{\sqrt{\nu t}\abs{\xi}}{\abs{\Zn(t)-y}}
\le\frac18.
\]
Set \(z_1=\sqrt{\nu t}\,\xi\), \(z_2=\Zn(t)-y\), and
\(\varphi=\psi\). Then
\[
\xi\cdot(\Zn(t)-y)^\perp
\stackrel{\eqref{eq:psi-def}}{=}
\abs{\xi}\abs{\Zn(t)-y}\sin\psi.
\]
Hence
\begin{equation}\label{eq:b-series}
\begin{aligned}
&\xi\cdot(\Zn(t)-y)^\perp
\left(
\frac1{\abs{\Zn(t)-y+\sqrt{\nu t}\,\xi}^2}
-\frac1{\abs{\Zn(t)-y}^2}
\right)
\\
&\qquad\stackrel{\substack{\eqref{eq:kernel-exp}\\\eqref{eq:psi-def}}}{=}
\sum_{m=2}^{\infty}
(-1)^{m+1}
(\nu t)^{(m-1)/2}
\frac{\abs{\xi}^m}{\abs{\Zn(t)-y}^m}
\sin(m\psi).
\end{aligned}
\end{equation}
The terms \(m=2,3,4\) give \eqref{eq:b-multipole}, and the
remaining terms give \eqref{eq:rb-explicit-tail}.

\smallskip
\noindent\emph{Identification of the strain.}
Separation permits differentiation under the Biot--Savart integral.
For \(a\neq0\) and \(h\in\Rt\),
\begin{equation}\label{eq:Biot-Savart-kernel-derivative}
DK(a)[h]
=
\frac1{2\pi}
\left(
\frac{h^\perp}{\abs a^2}
-
2\frac{(a\cdot h)a^\perp}{\abs a^4}
\right).
\end{equation}
Since \(\xi\cdot\xi^\perp=0\),
\[
\begin{aligned}
\xi\cdot\Sigma(t)\xi
&\stackrel{\substack{\eqref{eq:b-hyp}\\\eqref{eq:strain-jacobian}}}{=}
\int_{\Rt}
\xi\cdot DK(\Zn(t)-y)[\xi]\,
\wapp(y,t)\,dy
\\
&\stackrel{\eqref{eq:Biot-Savart-kernel-derivative}}{=}
-\frac1\pi
\int_{\Rt}
\frac{
\bigl(\xi\cdot(\Zn(t)-y)\bigr)
\bigl(\xi\cdot(\Zn(t)-y)^\perp\bigr)
}{
\abs{\Zn(t)-y}^4
}
\wapp(y,t)\,dy.
\end{aligned}
\]
Finally,
\[
\bigl(\xi\cdot(\Zn(t)-y)\bigr)
\bigl(\xi\cdot(\Zn(t)-y)^\perp\bigr)
\stackrel{\eqref{eq:psi-def}}{=}
\frac12\abs{\xi}^2\abs{\Zn(t)-y}^2\sin(2\psi),
\]
which proves \eqref{eq:b-strain-id}.
\end{proof}

\paragraph{Local and global bounds.}

\begin{lemma}
\label{lem:b-local}
For every \(0<t\le T\) and every \(\xi\in\Rt\) satisfying
\(\sqrt{\nu t}\abs{\xi}\le1\), the remainder in
\eqref{eq:b-multipole} satisfies
\begin{equation}\label{eq:b-tail}
\abs{r_b(\xi,t)}
\le
\frac1\pi
\norm[L^1]{\wapp(t)}
(\nu t)^2\abs{\xi}^5.
\end{equation}
Moreover,
\begin{equation}\label{eq:b-local-lipschitz}
\abs{b(\xi,t)}
\le
\frac{\norm[L^1]{\wapp(t)}}{\pi}
\sqrt{\nu t}\abs{\xi},
\end{equation}
and
\begin{equation}\label{eq:b-vector-Taylor}
\left|
\sqrt{\frac t\nu}\,b(\xi,t)
-t\Sigma(t)\xi
\right|
\lesssim t\sqrt{\nu t}\abs{\xi}^2.
\end{equation}
\end{lemma}

\begin{proof}
\noindent\emph{The multipole remainder.}
For \(y\in\supp\wapp(t)\), the ratio
\(\sqrt{\nu t}\abs{\xi}/\abs{\Zn(t)-y}\) is at most \(1/8\).
Hence
\[
\sum_{m=5}^{\infty}
(\nu t)^{(m-1)/2}
\frac{\abs{\xi}^m}{\abs{\Zn(t)-y}^m}
\le
2\frac{(\nu t)^2\abs{\xi}^5}{\abs{\Zn(t)-y}^5}.
\]
After integration against \((2\pi)^{-1}\abs{\wapp(y,t)}\,dy\),
\[
|r_b(\xi,t)|
\stackrel{\substack{\eqref{eq:rb-explicit-tail}\\\eqref{eq:b-hyp}}}{\le}
\frac1\pi\|\wapp(t)\|_{L^1}(\nu t)^2|\xi|^5,
\]
which proves \eqref{eq:b-tail}.

\smallskip
\noindent\emph{The local velocity bound.}
For \(a,a'\neq0\), the Biot--Savart kernel satisfies
\[
\abs{K(a)-K(a')}
=
\frac{\abs{a-a'}}{2\pi\abs a\abs{a'}}.
\]
On \(\supp\wapp(t)\),
\[
\abs{\Zn(t)-y}\ge8,
\qquad
\abs{\Zn(t)-y+\sqrt{\nu t}\,\xi}\ge7.
\]
Therefore,
\[
\begin{aligned}
\abs{b(\xi,t)}
&\le
\frac{\sqrt{\nu t}\abs{\xi}}{2\pi}
\int_{\Rt}
\frac{\abs{\wapp(y,t)}}
{
\abs{\Zn(t)-y+\sqrt{\nu t}\,\xi}
\abs{\Zn(t)-y}
}\,dy
\\
&\le
\frac{\norm[L^1]{\wapp(t)}}{\pi}
\sqrt{\nu t}\abs{\xi},
\end{aligned}
\]
which is \eqref{eq:b-local-lipschitz}.

\smallskip
\noindent\emph{The vector Taylor estimate.}
For \(\abs{x-\Zn(t)}\le1\), separation gives
\(\abs{x-y}\ge7\) on \(\supp\wapp(t)\). Differentiating the
Biot--Savart integral twice therefore yields
\[
\sup_{\abs{x-\Zn(t)}\le1}
\abs{D_x^2\vapp(x,t)}
\lesssim \norm[L^1]{\wapp(t)}
\lesssim 1.
\]
Since \(D_x\vapp(\Zn(t),t)=\Sigma(t)\), Taylor's formula gives
\[
\begin{aligned}
b(\xi,t)-\sqrt{\nu t}\,\Sigma(t)\xi
={}&
\nu t\int_0^1(1-\theta)
\\
&\quad\times
D_x^2\vapp
\bigl(\Zn(t)+\theta\sqrt{\nu t}\,\xi,t\bigr)
[\xi,\xi]\,d\theta.
\end{aligned}
\]
Consequently,
\[
\abs{
b(\xi,t)-\sqrt{\nu t}\,\Sigma(t)\xi
}
\lesssim \nu t\abs{\xi}^2.
\]
Multiplication by \(\sqrt{t/\nu}\) proves
\eqref{eq:b-vector-Taylor}.
\end{proof}

~\\
The velocity bound in Proposition~\ref{prop:regular-corrector} yields
the global estimate
\begin{equation}\label{eq:b-global}
\abs{b(\xi,t)}
\stackrel{\eqref{eq:strain-field}}{\le}
2\norm[L^\infty]{\vapp(t)}
\lesssim 1,
\qquad
\xi\in\Rt,\quad 0<t\le T.
\end{equation}

\section{Construction of the concentrated approximation}
\label{sec:high-approx}
Decreasing \(\nu_{0,T}\) if necessary, we assume throughout this section that
\[
\nu_{0,T}\le1,
\qquad
\nu_{0,T}T\le1.
\]
\subsection{Gaussian spaces and angular modes}\label{sec:Gaussian-function-spaces}

The real Gaussian Hilbert space is
\begin{equation}\label{eq:Gaussian-space-Y}
\begin{gathered}
\mathsf p_G(\xi):=e^{|\xi|^2/4},
\qquad
Y:=L^2(\mathbb R^2,\mathsf p_G\,d\xi;\mathbb R)
=\left\{f:\int_{\mathbb R^2}e^{|\xi|^2/4}|f(\xi)|^2\,d\xi<\infty\right\},
\\
\langle f,g\rangle_Y:=\int_{\mathbb R^2}\mathsf p_G f g\,d\xi,
\qquad
\|f\|_Y^2:=\langle f,f\rangle_Y.
\end{gathered}
\end{equation}
For \(n\in\mathbb N\), let \(Y_n=L^2_{\mathsf p_G,n}\) denote the
closed subspace consisting of the \(n\)-th angular mode:
\begin{equation}\label{eq:Gaussian-angular-spaces}
Y_n:=\left\{f\in Y:
f(r,\theta)=a(r)\cos(n\theta)+b(r)\sin(n\theta)\right\}.
\end{equation}
We also write \(L^2_{\mathsf p_G,\mathrm{rad}}\) for the radial
subspace of \(Y\) in \eqref{eq:Gaussian-space-Y}, and set
\begin{equation}\label{eq:Gaussian-energy-space}
\mathcal V
:=
\{f\in Y:\nabla f\in Y^2\text{ in the distributional sense}\},
\qquad
\|f\|_{\mathcal V}^2
:=
\|f\|_Y^2+\|\nabla f\|_Y^2.
\end{equation}
With this graph norm, \(\mathcal V\) is a Hilbert space.
For \(0<\gamma<1\), define the weighted supremum norm
\begin{equation}\label{eq:Gaussian-supremum-norm}
\|f\|_{L^\infty_\gamma}
:=
\sup_{\xi\in\mathbb R^2}
e^{\gamma|\xi|^2/4}|f(\xi)|.
\end{equation}
The smooth Gaussian class is
\begin{equation}\label{eq:Gaussian-class-Z}
\mathcal Z:=\left\{
e^{-|\xi|^2/4}u(\xi):
\begin{array}{l}
u\in C^\infty(\mathbb R^2;\mathbb R),\quad
\forall\alpha\in\mathbb N_0^2\ \exists M_\alpha,N_\alpha\ge0,\\
|\partial^\alpha u(\xi)|\le M_\alpha(1+|\xi|)^{N_\alpha}
\quad\forall\xi\in\mathbb R^2
\end{array}
\right\}.
\end{equation}
Thus every derivative of the Gaussian quotient has at most polynomial
growth.  The angular projections are
\begin{equation}\label{eq:Gaussian-angular-projections}
\Pi_{\mathrm{rad}}f(r)
:=\frac1{2\pi}\int_0^{2\pi}f(r,\theta)\,d\theta,
\qquad
\Pi_{\ne0}:=I-\Pi_{\mathrm{rad}}.
\end{equation}
With this notation,
\begin{equation}\label{eq:Gaussian-linearized-operators}
\begin{gathered}
G(\xi)=\frac1{4\pi}e^{-|\xi|^2/4},
\qquad
\mathcal L f=\Delta f+\frac12\xi\cdot\nabla f+f,
\\
\Lambda f=v^G\cdot\nabla f+v^f\cdot\nabla G,
\end{gathered}
\end{equation}
where $v^f=K*f$ and $K(\xi)=(2\pi)^{-1}\xi^\perp/|\xi|^2$.
The sign convention is
$\xi^\perp=(-\xi_2,\xi_1)$ and $v^f=\nabla^\perp\Psi$ when
$\Delta\Psi=f$.
Set
\begin{equation}\label{eq:fixed-mode-radial-coefficients}
\varphi(r):=\frac{1-e^{-r^2/4}}{2\pi r^2},
\qquad
h(r):=\frac{r^2/4}{e^{r^2/4}-1}
    .
\end{equation}
Both functions extend smoothly as radial functions at zero, with
$\varphi(0)=1/(8\pi)$ and $h(0)=1$. For $m\in\mathbb N_0$ and
$0<\gamma<1$, define
\begin{equation}\label{eq:finite-Gaussian-seminorm}
\mathfrak n_{m,\gamma}(f)
:=\sum_{|\alpha|\le m}
\sup_{\xi\in\mathbb R^2}
 e^{\gamma|\xi|^2/4}|\partial^\alpha f(\xi)|.
\end{equation}
Gaussian decay dominates polynomial growth. In particular,
\begin{equation}\label{eq:Z-Gaussian-inclusions}
f\in\mathcal Z
\quad\stackrel{\substack{\eqref{eq:Gaussian-space-Y}\\
\eqref{eq:Gaussian-class-Z}}}{\Longrightarrow}\quad
f\in Y,\qquad
\mathfrak n_{m,\gamma}(f)<\infty
\quad(m\in\mathbb N_0,\ 0<\gamma<1).
\end{equation}
The constants below may depend on the fixed angular mode $n$.
No uniformity as $n\to\infty$ is asserted.
The notation $\Lambda^{-1}$ denotes inversion on $Y_n\cap\mathcal Z$;
it does not denote a bounded inverse on all of $Y_n$.
\begin{lemma}\label{lem:fixed-mode-source-origin}
Let $n\ge2$ be an integer and let $q\in C^\infty(\mathbb R^2)$ satisfy
\[
q(r\cos\theta,r\sin\theta)
=a(r)\cos(n\theta)+b(r)\sin(n\theta),
\qquad r>0.
\]
Then
\begin{equation}\label{eq:angular-mode-vanishing-jets}
\partial^\alpha q(0)=0,
\qquad |\alpha|<n.
\end{equation}
In particular, $q(0)=0$ and $\nabla q(0)=0$.

~\\
The coefficients extend smoothly to $r\in\mathbb R$, with
\[
a(-r)=(-1)^na(r),
\qquad
b(-r)=(-1)^nb(r).
\]
Moreover, there exist $A,B\in C^\infty([0,\infty))$ such that
\begin{equation}\label{eq:angular-coefficient-factorization}
a(r)=r^nA(r^2),
\qquad
b(r)=r^nB(r^2),
\qquad r\ge0.
\end{equation}
Thus $a(r),b(r)=O(r^n)$ as $r\to0$.

~\\
Let $0<\gamma_1<1$. If
\[
\mathfrak n_{2,\gamma_1}(q)
:=
\sum_{|\alpha|\le2}
\sup_{\xi\in\mathbb R^2}
e^{\gamma_1|\xi|^2/4}|\partial^\alpha q(\xi)|
<\infty,
\]
then
\begin{equation}\label{eq:fixed-mode-a-derivatives}
\begin{aligned}
|a(r)|+|b(r)|
&\lesssim \mathfrak n_{2,\gamma_1}(q)e^{-\gamma_1r^2/4},\\
|a'(r)|+|b'(r)|
&\lesssim \mathfrak n_{2,\gamma_1}(q)e^{-\gamma_1r^2/4},\\
|a''(r)|+|b''(r)|
&\lesssim \mathfrak n_{2,\gamma_1}(q)e^{-\gamma_1r^2/4},
\end{aligned}
\qquad r\ge0.
\end{equation}
Moreover,
\begin{equation}\label{eq:fixed-mode-a-origin}
\begin{aligned}
|a(r)|+|b(r)|
&\lesssim \mathfrak n_{2,\gamma_1}(q)r^2,\\
|a'(r)|+|b'(r)|
&\lesssim \mathfrak n_{2,\gamma_1}(q)r,\\
|a''(r)|+|b''(r)|
&\lesssim \mathfrak n_{2,\gamma_1}(q),
\end{aligned}
\qquad 0<r\le1.
\end{equation}
The implicit constants in these coefficient bounds are absolute.
\end{lemma}
\begin{proof}
The Fourier coefficient formulas
\[
a(r)=\frac1\pi\int_0^{2\pi}
q(r\cos\theta,r\sin\theta)\cos(n\theta)\,d\theta,
\qquad
b(r)=\frac1\pi\int_0^{2\pi}
q(r\cos\theta,r\sin\theta)\sin(n\theta)\,d\theta
\]
extend $a,b$ smoothly to $r\in\mathbb R$. Changing
$\theta$ to $\theta+\pi$ gives
$a(-r)=(-1)^na(r)$ and $b(-r)=(-1)^nb(r)$.

~\\
We next consider the behavior at the origin. Every homogeneous
polynomial of degree $j<n$, restricted to the unit circle, contains
only angular frequencies of absolute value at most $j$. Hence it is
orthogonal to $\cos(n\theta)$ and $\sin(n\theta)$. Applying this
observation to the Taylor expansion of $q$ at the origin gives
$a(r),b(r)=O(r^n)$. Since
$q(r,\theta)=a(r)\cos(n\theta)+b(r)\sin(n\theta)$, it follows that
$q(\xi)=O(|\xi|^n)$, and therefore
\[
\partial^\alpha q(0)=0,
\qquad |\alpha|<n.
\]
In particular,
$a^{(j)}(0)=b^{(j)}(0)=0$ for $0\le j<n$.

~\\
Thus $a(r)/r^n$ and $b(r)/r^n$ extend smoothly through $r=0$.
By the parity above, these extensions are even. The standard
factorization of smooth even functions then gives
$A,B\in C^\infty([0,\infty))$ such that
\[
a(r)=r^nA(r^2),
\qquad
b(r)=r^nB(r^2).
\]

~\\
It remains to prove the quantitative bounds. Evaluating $q$ on the
rays $\theta=0$ and $\theta=\pi/(2n)$ gives
\[
a(r)=q(r,0),
\qquad
b(r)=q\left(
r\cos\frac{\pi}{2n},
r\sin\frac{\pi}{2n}
\right).
\]
Differentiating along these unit rays up to order two immediately
yields \eqref{eq:fixed-mode-a-derivatives}, with an absolute
constant.

~\\
Finally, since $n\ge2$, we have
$a(0)=b(0)=a'(0)=b'(0)=0$. Hence, for $0<r\le1$,
\[
a'(r)=\int_0^r a''(s)\,ds,
\qquad
a(r)=\int_0^r(r-s)a''(s)\,ds,
\]
and similarly for $b$. Hence
\[
|a(r)|+|b(r)|
\stackrel{\eqref{eq:fixed-mode-a-derivatives}}{\lesssim} \mathfrak n_{2,\gamma_1}(q)r^2,
\qquad
|a'(r)|+|b'(r)|
\stackrel{\eqref{eq:fixed-mode-a-derivatives}}{\lesssim} \mathfrak n_{2,\gamma_1}(q)r,
\]
while the corresponding bound for $a''$ and $b''$ follows directly
from \eqref{eq:fixed-mode-a-derivatives}. This proves
\eqref{eq:fixed-mode-a-origin}.
\end{proof}

\subsection{Inversion on fixed angular modes}
\begin{lemma}\label{lem:fixed-mode-fundamental-solutions}
Let $n\ge2$ be an integer. There are unique real-valued functions
$\psi_-,\psi_+\in C^\infty((0,\infty))$ solving
\begin{equation}\label{eq:fixed-mode-homogeneous-equation-proof}
-\psi''(r)-\frac1r\psi'(r)
+\left(\frac{n^2}{r^2}-h(r)\right)\psi(r)=0,
\qquad r>0,
\end{equation}
with the respective normalizations
\begin{equation}\label{eq:fixed-mode-fundamental-normalizations}
\lim_{r\to0}\frac{\psi_-(r)}{r^n}=1,
\qquad
\lim_{r\to\infty}r^n\psi_+(r)=1.
\end{equation}

~\\
Both solutions are strictly positive, with $\psi_-$ strictly
increasing and $\psi_+$ strictly decreasing:
\begin{equation}\label{eq:fixed-mode-fundamental-positivity}
\psi_-(r)>0,
\qquad
\psi_+(r)>0,
\qquad
\psi_-'(r)>0,
\qquad
\psi_+'(r)<0,
\qquad r>0.
\end{equation}
For every $r>0$, they satisfy the bounds
\begin{subequations}\label{eq:fixed-mode-fundamental-solution-bounds}
\begin{equation}
0<\psi_-(r)\le r^n,
\end{equation}
\begin{equation}
0<\psi_+(r)\le r^{-n},
\end{equation}
\begin{equation}
0<\psi_-'(r)\le n r^{n-1},
\end{equation}
\begin{equation}
0<-\psi_+'(r)\le n r^{-n-1},
\end{equation}
\begin{equation}
|\psi_-''(r)|\le n(n+1)r^{n-2},
\end{equation}
\begin{equation}
|\psi_+''(r)|\le n(n+1)r^{-n-2}.
\end{equation}
\end{subequations}

~\\
The quantity
\begin{equation}\label{eq:fixed-mode-Wronskian-convention}
w_0
:=
r\bigl(
\psi_-'(r)\psi_+(r)-\psi_-(r)\psi_+'(r)
\bigr)
\end{equation}
is independent of $r$, depends only on $n$, and satisfies
\[
0<w_0<2n.
\]
Moreover, for every $r>0$,
\begin{subequations}\label{eq:fundamental-two-sided-bounds}
\begin{equation}
\frac{w_0}{2n}\,r^n
\le \psi_-(r)\le r^n,
\end{equation}
\begin{equation}
\frac{w_0}{2n}\,r^{-n}
\le \psi_+(r)\le r^{-n}.
\end{equation}
\end{subequations}
At the opposite endpoints, the solutions satisfy
\[
\lim_{r\to\infty}\frac{\psi_-(r)}{r^n}
=\frac{w_0}{2n},
\]
and
\[
\lim_{r\to0}r^n\psi_+(r)
=\frac{w_0}{2n}.
\]
\end{lemma}
\begin{proof}
The radial coefficient satisfies
\[
h(r)\stackrel{\eqref{eq:fixed-mode-radial-coefficients}}{=}O(1)
\quad(r\to0),
\qquad
h(r)\stackrel{\eqref{eq:fixed-mode-radial-coefficients}}{\lesssim}
r^2e^{-r^2/4}\quad(r\to\infty).
\]
Moreover, with \(s=r^2/4\),
\[
0<r^2h(r)=\frac{4s^2}{e^s-1}<4\le n^2.
\]
Hence
\begin{equation}\label{eq:fixed-mode-positive-potential}
0<
V_n(r):=\frac{n^2}{r^2}-h(r)
\le \frac{n^2}{r^2}.
\end{equation}

~\\
Write
\[
\psi_-(r)=r^n u_-(r),
\qquad
\psi_+(r)=r^{-n}u_+(r).
\]
The transformed equations are
\[
\bigl(r^{2n+1}u_-'\bigr)'
=-r^{2n+1}h\,u_-,
\qquad
\bigl(r^{1-2n}u_+'\bigr)'
=-r^{1-2n}h\,u_+.
\]
Standard Volterra theory at \(0\) and \(\infty\), followed by ordinary
ODE continuation, gives the unique normalized solutions. Indeed,
boundedness at the normalized endpoints forces
\[
\lim_{r\to0}r^{2n+1}u_-'(r)=0,
\qquad
\lim_{r\to\infty}r^{1-2n}u_+'(r)=0,
\]
and therefore
\begin{equation}\label{eq:fundamental-u-derivatives}
\begin{aligned}
u_-'(r)
&=-r^{-2n-1}\int_0^r s^{2n+1}h(s)u_-(s)\,ds,\\
u_+'(r)
&=r^{2n-1}\int_r^\infty s^{1-2n}h(s)u_+(s)\,ds.
\end{aligned}
\end{equation}
In particular,
\[
u_-(r)=1+O(r^2),\qquad u_-'(r)=O(r)
\quad(r\to0),
\]
and
\[
u_+(r)\to1,\qquad
u_+'(r)=O\!\left(re^{-r^2/4}\right)
\quad(r\to\infty).
\]

~\\
Since \(V_n>0\), the maximum principle and the endpoint
normalizations imply
\[
\psi_-(r)>0,\qquad \psi_+(r)>0,\qquad r>0.
\]
Writing the equation as
\[
(r\psi')'=rV_n(r)\psi
\]
and integrating from the corresponding normalized endpoint gives
\begin{equation}\label{eq:fundamental-flux-identities}
\begin{aligned}
r\psi_-'(r)
&=\int_0^r
\left(\frac{n^2}{s}-sh(s)\right)\psi_-(s)\,ds>0,\\
r\psi_+'(r)
&=-\int_r^\infty
\left(\frac{n^2}{s}-sh(s)\right)\psi_+(s)\,ds<0.
\end{aligned}
\end{equation}
Thus \(\psi_-\) is strictly increasing and \(\psi_+\) is strictly
decreasing.

~\\
The positivity of \(h\) and \(\psi_\pm\) in
\eqref{eq:fundamental-u-derivatives} also gives
\[
u_-'<0,\qquad u_+'>0.
\]
Since \(u_-(0)=1\) and \(u_+(\infty)=1\), it follows that
\[
0<u_-(r)\le1,\qquad 0<u_+(r)\le1,
\]
and hence
\[
0<\psi_-(r)\le r^n,
\qquad
0<\psi_+(r)\le r^{-n}.
\]
Consequently,
\[
0<\psi_-'(r)
\stackrel{\eqref{eq:fundamental-flux-identities}}{\le}
\frac{n^2}{r}\int_0^r s^{n-1}\,ds
=nr^{n-1},
\]
and
\[
0<-\psi_+'(r)
\stackrel{\eqref{eq:fundamental-flux-identities}}{\le}
\frac{n^2}{r}\int_r^\infty s^{-n-1}\,ds
=nr^{-n-1}.
\]
For the second derivatives,
\[
|\psi_\pm''|
\stackrel{\substack{\eqref{eq:fixed-mode-homogeneous-equation-proof}\\
\eqref{eq:fixed-mode-positive-potential}}}{\le}
\frac{|\psi_\pm'|}{r}
+\frac{n^2}{r^2}\psi_\pm,
\]
which proves the stated second-derivative bounds.

~\\
Abel's identity shows that
\[
w_0
=r\bigl(\psi_-'\psi_+-\psi_-\psi_+'\bigr)
\]
is constant, and the signs of \(\psi_\pm\) and \(\psi_\pm'\) give \(w_0>0\).
In terms of \(u_\pm\),
\begin{equation}\label{eq:fundamental-Wronskian-u}
w_0
=
2nu_-u_+
+ru_-'u_+
-ru_-u_+'.
\end{equation}

~\\
At the opposite endpoints,
\[
|u_-'(r)|
\stackrel{\eqref{eq:fundamental-u-derivatives}}{\le}
r^{-2n-1}
\int_0^\infty s^{2n+1}h(s)\,ds
\lsim{n} r^{-2n-1},
\qquad r\ge1,
\]
and, since \(n\ge2\),
\[
0<u_+'(r)
\lsim{n}
r^{2n-1}
\left(\int_r^1s^{1-2n}\,ds+1\right)
\lsim{n} r,
\qquad 0<r\le1.
\]
Thus \(ru_-'\to0\) at infinity and \(ru_+'\to0\) at zero.
The endpoint limits therefore satisfy
\[
\lim_{r\to\infty}u_-(r)
=
\lim_{r\to0}u_+(r)
\stackrel{\eqref{eq:fundamental-Wronskian-u}}{=}
\frac{w_0}{2n}.
\]
Since \(u_-\) is strictly decreasing from \(1\), we have
\(0<w_0<2n\). The monotonicity of \(u_\pm\) now yields
\[
\frac{w_0}{2n}r^n
\le\psi_-(r)\le r^n,
\qquad
\frac{w_0}{2n}r^{-n}
\le\psi_+(r)\le r^{-n},
\]
as well as the stated limits.
\end{proof}
\begin{lemma}\label{lem:fixed-mode-Green-inverse}
Let $n\ge2$ and $q\in Y_n\cap\mathcal Z$.
There exists a unique $f^0\in Y_n\cap\mathcal Z$ satisfying
$\Lambda f^0=q$.

~\\
For a sine source $q(r,\theta)=a(r)\sin(n\theta)$, the solution is
\begin{equation}\label{eq:fixed-mode-inverse}
f^0(r,\theta)=-\omega(r)\cos(n\theta),
\qquad
\omega(r)=h(r)\Omega(r)+\frac{a(r)}{n\varphi(r)},
\end{equation}
where $\Omega$ is the unique solution of
\begin{equation}\label{eq:fixed-mode-radial-ODE}
-\Omega''-\frac1r\Omega'
+\left(\frac{n^2}{r^2}-h(r)\right)\Omega
=\frac{a(r)}{n\varphi(r)},
\qquad r>0,
\end{equation}
satisfying $\Omega(r)=O(r^n)$ as $r\to0$ and
$\Omega(r)=O(r^{-n})$ as $r\to\infty$.
It is given by
\begin{equation}\label{eq:fixed-mode-Green-representation}
\begin{split}
\Omega(r)
={}&
\frac{\psi_+(r)}{w_0}
\int_0^r y\psi_-(y)\frac{a(y)}{n\varphi(y)}\,dy
\\
&+
\frac{\psi_-(r)}{w_0}
\int_r^\infty y\psi_+(y)\frac{a(y)}{n\varphi(y)}\,dy.
\end{split}
\end{equation}
For a cosine source $q(r,\theta)=a(r)\cos(n\theta)$,
the solution is $f^0(r,\theta)=\omega(r)\sin(n\theta)$,
with the same radial formulas.

~\\
For every integer $0\le\ell\le n$ and $0<\gamma_1<1$,
\begin{equation}\label{eq:fixed-mode-weighted-radial-estimate}
\begin{split}
&
\sup_{0<r\le1}
\sum_{k=0}^2 r^{k-\ell}|\Omega^{(k)}(r)|
+
\sup_{r\ge1}
\sum_{k=0}^2 r^{n+k}|\Omega^{(k)}(r)|
\\
&\qquad\lsim{n,\ell,\gamma_1} 
\left(
\sup_{0<r\le1}\frac{|a(r)|}{r^\ell}
+
\sup_{r\ge1}e^{\gamma_1r^2/4}|a(r)|
\right).
\end{split}
\end{equation}
The constant is independent of $a$.
In particular, taking $\ell=n$ gives the endpoint bounds
$|\Omega^{(k)}(r)|=O(r^{n-k})$ at zero and
$|\Omega^{(k)}(r)|=O(r^{-n-k})$ at infinity, for $k=0,1,2$.

~\\
For $m\in\{0,1,2\}$, the source coefficient satisfies
\begin{equation}\label{eq:fixed-mode-coefficient-Gaussian-control}
\sup_{0<r\le1}\frac{|a(r)|}{r^m}
+
\sup_{r\ge1}e^{\gamma_1r^2/4}|a(r)|
\lesssim \mathfrak n_{m,\gamma_1}(q).
\end{equation}
Consequently, the right-hand side of
\eqref{eq:fixed-mode-weighted-radial-estimate}, with $\ell=m$,
is \(\lsim{n,m,\gamma_1}\mathfrak n_{m,\gamma_1}(q)\).
\end{lemma}

\begin{proof}
Existence, uniqueness, and the representation follow from
\cite{Ga11}.
We prove the quantitative estimates for the sine source;
the cosine case follows by rotation.

~\\
Fix $\ell$ and $\gamma_1$, and let
\[
N:=
\sup_{0<r\le1}\frac{|a(r)|}{r^\ell}
+
\sup_{r\ge1}e^{\gamma_1r^2/4}|a(r)|,
\qquad
s(r):=\frac{a(r)}{n\varphi(r)}.
\]
The coefficient bounds give $|s(r)|\lsim{n} Nr^\ell$ for
$0<r\le1$ and
$|s(r)|\lsim{n} Nr^2e^{-\gamma_1r^2/4}$ for $r\ge1$.

~\\
Differentiating \eqref{eq:fixed-mode-Green-representation},
the terms from the integration limits cancel in $\Omega'$ and contribute
$-s(r)$ to $\Omega''$, by the Wronskian identity.
Consequently,
\begin{equation}\label{eq:fixed-mode-Green-basic-bound}
\begin{split}
\sum_{k=0}^2 r^k|\Omega^{(k)}(r)|
\stackrel{\substack{\eqref{eq:fixed-mode-Green-representation}\\
\eqref{eq:fixed-mode-fundamental-solution-bounds}}}{\lsim{n}} \biggl(
&r^{-n}\int_0^r y^{n+1}|s(y)|\,dy
\\
&+r^n\int_r^\infty y^{1-n}|s(y)|\,dy
+r^2|s(r)|
\biggr).
\end{split}
\end{equation}

~\\
We first consider $0<r\le1$ and divide this estimate by $r^\ell$.
The first and last terms are \(\lsim{n}Nr^2\).
Splitting the remaining integral at one, its exterior part is
\(\lsim{n,\gamma_1}N\), since \(r^{n-\ell}\le1\).
For its interior part, the condition $\ell\le n$ gives
\[
r^{n-\ell}\int_r^1 y^{\ell+1-n}\,dy
=
\int_r^1 y\left(\frac ry\right)^{n-\ell}\,dy
\le\frac12.
\]
This proves the interior estimate in
\eqref{eq:fixed-mode-weighted-radial-estimate}.

~\\
For $r\ge1$, the source bounds imply
$\int_0^\infty y^{n+1}|s(y)|\,dy
\lsim{n,\ell,\gamma_1} N$, while
\[
r^{2n}\int_r^\infty y^{1-n}|s(y)|\,dy
\le
\int_r^\infty y^{n+1}|s(y)|\,dy
\lsim{n,\ell,\gamma_1} N.
\]
Moreover,
$r^{n+2}|s(r)|
\lsim{n} Nr^{n+4}e^{-\gamma_1r^2/4}
\lsim{n,\gamma_1} N$.
Multiplying \eqref{eq:fixed-mode-Green-basic-bound} by $r^n$
therefore proves the exterior estimate.

~\\
Finally, evaluate $a$ along a unit ray on which the angular
factor equals one. Its radial derivatives through order $m$
are then bounded by the corresponding Cartesian derivatives of $q$.
For $m=1,2$, Lemma~\ref{lem:fixed-mode-source-origin} gives
$a^{(j)}(0)=0$ for $j<m$, so Taylor's formula yields
$|a(r)|\lesssim  r^m\mathfrak n_{m,\gamma_1}(q)$ on $0<r\le1$.
The case $m=0$ and the exterior Gaussian bound follow directly
by evaluating $a$ along the same ray. This proves
\eqref{eq:fixed-mode-coefficient-Gaussian-control}.
\end{proof}
\begin{lemma}\label{lem:fixed-mode-unregularized-estimates}
Let $n\ge2$, $q\in Y_n\cap\mathcal Z$, and
$f^0=\Lambda^{-1}q$ as in
Lemma~\ref{lem:fixed-mode-Green-inverse}.
For $m\in\{0,1,2\}$ and $0<\gamma<\gamma_1<1$,
\begin{equation}\label{eq:fixed-mode-higher-unregularized}
\mathfrak n_{m,\gamma}(f^0)
\lsim{n,m,\gamma,\gamma_1} 
\mathfrak n_{m,\gamma_1}(q).
\end{equation}
In particular,
\begin{equation}\label{eq:fixed-mode-unregularized-Gaussian-estimate}
\|f^0\|_{L^\infty_\gamma}
+\|\nabla f^0\|_{L^\infty_\gamma}
\lsim{n,\gamma,\gamma_1} 
\bigl(
\|q\|_{L^\infty_{\gamma_1}}
+\|\nabla q\|_{L^\infty_{\gamma_1}}
\bigr).
\end{equation}
The constants are independent of $q$.
\end{lemma}

\begin{proof}
Reflection separates the sine and cosine components, and rotation
interchanges them. These operations preserve the Gaussian weights
and act boundedly on each finite derivative seminorm.
Within this proof, dependence on \(n,m,\gamma,\gamma_1\) is suppressed;
the implicit constants are independent of the source \(q\).
By linearity, it therefore suffices to consider
$q(r,\theta)=a(r)\sin(n\theta)$.
Set $N:=\mathfrak n_{m,\gamma_1}(q)$ and
$s:=a/(n\varphi)$, and write
$f^0=-\omega(r)\cos(n\theta)$ with $\omega=h\Omega+s$.

~\\
We first estimate the radial source from derivatives of $q$
through order $m$. The identity
$a(r)=q(r\cos(\pi/(2n)),r\sin(\pi/(2n)))$
bounds $a^{(k)}$ by the corresponding directional derivative of $q$.
Moreover, Lemma~\ref{lem:fixed-mode-source-origin} gives
$a^{(j)}(0)=0$ for $j<m$.
Taylor's formula and the coefficient bounds for $\varphi^{-1}$
therefore imply, for $0\le k\le m$,
\begin{equation}\label{eq:fixed-mode-source-origin}
|s^{(k)}(r)|\lesssim Nr^{m-k},
\qquad 0<r\le1,
\end{equation}
and
\begin{equation}\label{eq:fixed-mode-source-exterior}
|s^{(k)}(r)|\lesssim Nr^2e^{-\gamma_1r^2/4},
\qquad r\ge1.
\end{equation}

~\\
For $\Omega$, the choice $\ell=m$ is admissible because $m\le n$:
\begin{equation}\label{eq:fixed-mode-Omega-origin}
|\Omega^{(k)}(r)|
\stackrel{\substack{\eqref{eq:fixed-mode-weighted-radial-estimate}\\
\eqref{eq:fixed-mode-coefficient-Gaussian-control}}}{\lesssim} Nr^{m-k},
\qquad 0<r\le1,\quad 0\le k\le m,
\end{equation}
and
\begin{equation}\label{eq:fixed-mode-Omega-exterior}
|\Omega^{(k)}(r)|
\stackrel{\substack{\eqref{eq:fixed-mode-weighted-radial-estimate}\\
\eqref{eq:fixed-mode-coefficient-Gaussian-control}}}{\lesssim} Nr^{-n-k},
\qquad r\ge1,\quad 0\le k\le m.
\end{equation}

~\\
We now show that the Cartesian derivatives remain bounded near
the origin. Since the derivatives of $h$ through order two are
bounded on $[0,1]$, Leibniz' rule and $\omega=h\Omega+s$ yield
\begin{equation}\label{eq:fixed-mode-origin-bound}
|\omega^{(k)}(r)|
\stackrel{\substack{\eqref{eq:fixed-mode-source-origin}\\
\eqref{eq:fixed-mode-Omega-origin}}}{\lesssim} Nr^{m-k},
\qquad 0<r\le1,\quad 0\le k\le m.
\end{equation}
Polar differentiation gives
\[
|\nabla^\ell f^0|
\lsim{n,\ell} \sum_{k=0}^{\ell}
r^{k-\ell}|\omega^{(k)}(r)|,
\qquad r>0,\quad 0\le\ell\le m.
\]
Consequently, $|\nabla^\ell f^0|\lesssim Nr^{m-\ell}\lesssim N$
on $0<r\le1$. The bounds extend to the origin by smoothness,
and the Gaussian weight is bounded on this disk.

~\\
Finally, we show Gaussian decay of the derivatives for $r\ge1$.
The bounds $|h^{(j)}(r)|\lesssim r^{j+2}e^{-r^2/4}$ for $j\le2$,
together with \eqref{eq:fixed-mode-Omega-exterior}, imply
$|(h\Omega)^{(k)}(r)|\lesssim Nr^{k+2-n}e^{-r^2/4}$.
Since $k\le m\le n$ and $\gamma_1<1$, we obtain
\[
|\omega^{(k)}(r)|
\lesssim N\left(
r^{k+2-n}e^{-r^2/4}
+r^2e^{-\gamma_1r^2/4}
\right)
\lesssim Nr^2e^{-\gamma_1r^2/4},
\qquad r\ge1,\quad 0\le k\le m.
\]
The factor $r^2e^{-(\gamma_1-\gamma)r^2/4}$ is bounded.
Applying the same polar differentiation estimate, now with
$r^{k-\ell}\le1$ for $k\le\ell$, therefore gives
\[
\sup_{|\xi|\ge1}
e^{\gamma|\xi|^2/4}|\nabla^\ell f^0(\xi)|
\lesssim N,
\qquad 0\le\ell\le m.
\]
Combining the interior and exterior estimates proves
\eqref{eq:fixed-mode-higher-unregularized}.
Taking $m=1$ gives
\eqref{eq:fixed-mode-unregularized-Gaussian-estimate}.
\end{proof}
\subsection{Uniform regularized inversion and velocity bounds}
\begin{lemma}
\label{lem:fixed-mode-elementary-tools}
Fix $n\ge2$. The closed quadratic form associated with
\(1-\mathcal L\) has domain \(\mathcal V\) and satisfies
\begin{equation}\label{eq:fixed-mode-energy-form}
\mathfrak q(f,f)
:=
\int_{\mathbb R^2}\mathsf p_G|\nabla f|^2
\ge \|f\|_Y^2,
\qquad f\in\mathcal V.
\end{equation}
Whenever \(f\in\mathcal V\) and the distribution
\((1-\mathcal L)f\) belongs to \(Y\), this form agrees with the
operator pairing:
\[
\langle(1-\mathcal L)f,f\rangle_Y=\mathfrak q(f,f).
\]
On $Y_n$, the operator $\Lambda$ extends to a bounded operator and
\begin{equation}\label{eq:fixed-mode-Lambda-skew}
\langle\Lambda f,g\rangle_Y
=
-\langle f,\Lambda g\rangle_Y,
\qquad f,g\in Y_n.
\end{equation}
For $m\in\mathbb N_0$, $0<\gamma<1$, and every smooth $f$
with finite $\mathfrak n_{m,\gamma}(f)$,
\begin{equation}\label{eq:Biot-Savart-finite-derivatives}
\|\nabla^jv^f\|_{L^\infty}
\lsim{j,\gamma} \mathfrak n_{j,\gamma}(f),
\qquad 0\le j\le m.
\end{equation}
If, in addition, $f\in Y_n$, then
\begin{equation}\label{eq:Lambda-finite-Gaussian-bound}
\mathfrak n_{m,\gamma}(\Lambda f)
\lsim{n,m,\gamma} \mathfrak n_{m,\gamma}(f).
\end{equation}
Both $\mathcal L$ and $\Lambda$ preserve $Y_n\cap\mathcal Z$.
\end{lemma}

\begin{proof}
We first consider the quadratic form.
Since $1-\mathcal L
=-\mathsf p_G^{-1}\nabla\cdot(\mathsf p_G\nabla)$,
integration by parts gives the pairing identity for smooth functions
in the operator domain.  For the lower bound, write
\(\rho=e^{-|\xi|^2/4}\) and \(f=\rho u\).  For compactly supported
smooth functions, expansion and integration by parts yield
\[
\int_{\mathbb R^2}\mathsf p_G|\nabla f|^2
=
\int_{\mathbb R^2}\rho\bigl(|\nabla u|^2+|u|^2\bigr)
\ge \|f\|_Y^2.
\]
Closing the form extends the estimate to \(\mathcal V\), and the
pairing identity holds whenever the distribution
\((1-\mathcal L)f\) belongs to \(Y\), by the definition of the
associated operator.

~\\
We next show that $\Lambda$ is bounded and skew-symmetric on $Y_n$.
The Gaussian weight gives
$\|f\|_{L^1}+\|f\|_{L^2}\lesssim \|f\|_Y$.
Splitting $K$ at radius one, its near part belongs to $L^{4/3}$
and its far part to $L^4$. Young's inequality therefore gives
$\|v^f\|_{L^4}\lesssim \|f\|_Y$, and hence
\[
\|v^f\cdot\nabla G\|_Y
\le
\|v^f\|_{L^4}
\|\mathsf p_G^{1/2}\nabla G\|_{L^4}
\lesssim \|f\|_Y.
\]
Since $v^G\cdot\nabla f=\varphi\partial_\theta f$,
$\varphi$ is bounded, and
$\|\partial_\theta f\|_Y=n\|f\|_Y$ on $Y_n$, we obtain
$\|\Lambda f\|_Y\lsim{n} \|f\|_Y$.

~\\
The transport term is skew-symmetric because $v^G$ is divergence
free and tangent to the level sets of $\mathsf p_G$.
For the nonlocal term,
$\mathsf p_G\nabla G=-\xi/(8\pi)$ and
\[
\begin{split}
&\int_{\mathbb R^2}(\xi\cdot v^f(\xi))g(\xi)\,d\xi
+
\int_{\mathbb R^2}(\xi\cdot v^g(\xi))f(\xi)\,d\xi
\\
&\qquad=
\iint_{\mathbb R^2\times\mathbb R^2}
(\xi-\eta)\cdot K(\xi-\eta)\,
g(\xi)f(\eta)\,d\eta\,d\xi
=0.
\end{split}
\]
This calculation is justified first for smooth compactly
supported functions. Boundedness and density then give
\eqref{eq:fixed-mode-Lambda-skew}.

~\\
We now estimate derivatives of the velocity. Moving derivatives
onto $f$ and splitting the kernel at radius one gives
\[
\|\partial^\alpha v^f\|_{L^\infty}
\lesssim \bigl(
\|\partial^\alpha f\|_{L^\infty}
+\|\partial^\alpha f\|_{L^1}
\bigr)
\lsim{\gamma} 
\|\partial^\alpha f\|_{L^\infty_\gamma},
\]
where
\[
\|\partial^\alpha f\|_{L^1}
\stackrel{\eqref{eq:Gaussian-supremum-norm}}{\le}
(4\pi/\gamma)\|\partial^\alpha f\|_{L^\infty_\gamma}.
\]
Summing over $|\alpha|=j$ proves
\eqref{eq:Biot-Savart-finite-derivatives}.

~\\
To estimate $\Lambda$ without losing derivatives, let $R$ be
rotation through $\pi/(2n)$.
On mode $n$, the identity
$\partial_\theta f(\xi)=n f(R\xi)$ gives
\[
\mathfrak n_{m,\gamma}(\partial_\theta f)
\lsim{n,m} \mathfrak n_{m,\gamma}(f).
\]
The integral representation
\[
\varphi(|\xi|)=\frac1{8\pi}\int_0^1e^{-s|\xi|^2/4}\,ds
\]
shows that every Cartesian derivative of $\varphi$ is bounded.
Leibniz' rule therefore controls
$\mathfrak n_{m,\gamma}(\varphi\partial_\theta f)$.
For $v^f\cdot\nabla G$, the same rule, the velocity estimates,
and the Gaussian decay of every derivative of $G$ give the
required bound because $\gamma<1$.
This proves \eqref{eq:Lambda-finite-Gaussian-bound}.

~\\
Finally, writing $f=\rho u$, we have
\[
\rho^{-1}\mathcal L(\rho u)
=
\Delta u-\frac12\xi\cdot\nabla u,
\]
and
\[
\rho^{-1}\Lambda(\rho u)
=
\varphi\partial_\theta u
-\frac1{8\pi}\xi\cdot v^{\rho u}.
\]
If $u$ and all its derivatives have polynomial growth, so do
the right-hand sides, by the velocity estimates.
Thus both operators preserve the class $\mathcal Z$ in
\eqref{eq:Gaussian-class-Z}.
Their rotational invariance preserves the angular mode,
which completes the proof.
\end{proof}
\begin{lemma}
\label{lem:OU-Gaussian-smoothing}
Let $m\in\mathbb N_0$ and $0<\gamma<\gamma_1<1$.
For every smooth $g\in Y$ with finite
$\mathfrak n_{m,\gamma_1}(g)$,
\begin{equation}\label{eq:OU-resolvent-one-derivative}
\mathfrak n_{m+1,\gamma}
\bigl((1-\mathcal L)^{-1}g\bigr)
\lsim{m,\gamma,\gamma_1} \mathfrak n_{m,\gamma_1}(g).
\end{equation}
\end{lemma}

\begin{proof}
For related Gaussian semigroup estimates, see
\cite{Gallay2018,GallayMaekawa2011}.

~\\
Let $P_\tau=e^{\tau\mathcal L}$.
The rescaling of the heat equation
$P_\tau g(\xi)
=e^\tau(e^{(e^\tau-1)\Delta}g)(e^{\tau/2}\xi)$
gives
\begin{equation}\label{eq:OU-kernel-explicit}
P_\tau g(\xi)
=
\frac1{4\pi a_\tau}
\int_{\mathbb R^2}
e^{-|\xi-b_\tau\eta|^2/(4a_\tau)}g(\eta)\,d\eta,
\end{equation}
where $a_\tau=1-e^{-\tau}$ and $b_\tau=e^{-\tau/2}$.
The semigroup is contractive:
\[
\|P_\tau g\|_Y
\stackrel{\eqref{eq:fixed-mode-energy-form}}{\le}\|g\|_Y.
\]
Consequently,
\begin{equation}\label{eq:OU-resolvent-integral}
(1-\mathcal L)^{-1}g
=
\int_0^\infty e^{-\tau}P_\tau g\,d\tau,
\end{equation}
with convergence in $Y$.

~\\
We first show a Gaussian bound for one derivative of $P_\tau g$.
Set $d_\tau=b_\tau^2+\gamma_1a_\tau$, so that
$\gamma_1\le d_\tau\le1$.
Completion of the square gives
\[
\frac{|\xi-b_\tau\eta|^2}{a_\tau}
+\gamma_1|\eta|^2
=
\frac{d_\tau}{a_\tau}
\left|\eta-\frac{b_\tau}{d_\tau}\xi\right|^2
+
\frac{\gamma_1}{d_\tau}|\xi|^2.
\]
Integrating the centered Gaussian yields
\[
|P_\tau g(\xi)|
\le
d_\tau^{-1}\mathfrak n_{0,\gamma_1}(g)
e^{-\gamma_1|\xi|^2/(4d_\tau)}.
\]
For the differentiated kernel, the identity
$\xi-b_\tau\eta
=(\gamma_1a_\tau/d_\tau)\xi
-b_\tau(\eta-b_\tau\xi/d_\tau)$ gives
\[
|\nabla P_\tau g(\xi)|
\lsim{\gamma_1} \mathfrak n_{0,\gamma_1}(g)
\bigl(|\xi|+a_\tau^{-1/2}\bigr)
e^{-\gamma_1|\xi|^2/4}.
\]
Since $\gamma<\gamma_1$, the factor $|\xi|$ is absorbed by the Gaussian.
Together with $a_\tau\le1$, this proves
\begin{equation}\label{eq:OU-one-derivative-Gaussian}
\mathfrak n_{1,\gamma}(P_\tau g)
\lsim{\gamma,\gamma_1} a_\tau^{-1/2}
\mathfrak n_{0,\gamma_1}(g).
\end{equation}

~\\
We next estimate the higher derivatives for $0<\tau\le1$.
Integration by parts in the kernel gives
$\partial^\alpha P_\tau g
=e^{|\alpha|\tau/2}P_\tau\partial^\alpha g$.
After transferring all but one derivative to $g$,
\[
\mathfrak n_{m+1,\gamma}(P_\tau g)
\stackrel{\eqref{eq:OU-one-derivative-Gaussian}}{\lsim{m,\gamma,\gamma_1}} a_\tau^{-1/2}
\mathfrak n_{m,\gamma_1}(g),
\qquad 0<\tau\le1,
\]
because the commutation factors are bounded on this interval.

~\\
For $\tau\ge1$, we instead place all derivatives on the kernel.
Since $a_\tau$ is bounded below, each differentiated kernel
is a polynomial in $\xi-b_\tau\eta$ times the original Gaussian,
with coefficients bounded uniformly in $\tau$.
The same completion of the square and Gaussian moment estimates
therefore give
\[
|\partial^\alpha P_\tau g(\xi)|
\lsim{m,\gamma_1} \mathfrak n_{0,\gamma_1}(g)
(1+|\xi|)^{m+1}e^{-\gamma_1|\xi|^2/4},
\qquad |\alpha|\le m+1.
\]
Since $\gamma<\gamma_1$, the polynomial factor is absorbed by the Gaussian, and
\[
\mathfrak n_{m+1,\gamma}(P_\tau g)
\lsim{m,\gamma,\gamma_1} \mathfrak n_{0,\gamma_1}(g),
\qquad \tau\ge1.
\]

~\\
Since $a_\tau\le1$, these estimates give
\[
\mathfrak n_{m+1,\gamma}(P_\tau g)
\lsim{m,\gamma,\gamma_1} a_\tau^{-1/2}
\mathfrak n_{m,\gamma_1}(g),
\qquad \tau>0.
\]
Finally,
\[
\int_0^\infty e^{-\tau}a_\tau^{-1/2}\,d\tau
=
\int_0^1 s^{-1/2}\,ds
=2.
\]
This integrable bound justifies differentiation under
\eqref{eq:OU-resolvent-integral} and proves
\eqref{eq:OU-resolvent-one-derivative}.
\end{proof}

\begin{lemma}
\label{lem:regularized-fixed-mode-inverse}
Fix $n\ge2$ and $q\in Y_n\cap\mathcal Z$ independent of
$\nu$. For every $0<\nu\le1$, the equation
\begin{equation}\label{eq:regularized-fixed-mode-equation}
\bigl[\nu(1-\mathcal L)+\Lambda\bigr]f^\nu=q
\end{equation}
has a unique weak solution \(f^\nu\in\mathcal V\cap Y_n\), and this
solution belongs to \(Y_n\cap\mathcal Z\).
We denote the solution by
\[
f^\nu=R_\nu q
:=
\bigl[\nu(1-\mathcal L)+\Lambda\bigr]^{-1}q.
\]
Then, for every $m\in\mathbb N_0$ and $0<\gamma<1$,
\begin{equation}\label{eq:fixed-mode-higher-fixed-source}
\sup_{0<\nu\le1}\mathfrak n_{m,\gamma}(f^\nu)
\lsim{q,n,m,\gamma} 1.
\end{equation}
Moreover, if $f^0=\Lambda^{-1}q$, then
\begin{equation}\label{eq:fixed-mode-strong-fixed-source-convergence}
\mathfrak n_{m,\gamma}(f^\nu-f^0)
\lsim{q,n,m,\gamma} \nu,
\qquad 0<\nu\le1.
\end{equation}
In particular,
\begin{equation}\label{eq:nu-fixed-mode-first-order}
\|f^\nu-f^0\|_Y\lsim{q,n} \nu,
\qquad 0<\nu\le1,
\end{equation}
and
\begin{equation}\label{eq:nu-fixed-mode-Gaussian}
\|f^\nu\|_{L^\infty_\gamma}
+\|\nabla f^\nu\|_{L^\infty_\gamma}
\lsim{q,n,\gamma} 1,
\qquad 0<\nu\le1.
\end{equation}
\end{lemma}
\begin{proof}
Let \(\mathcal V_n:=\mathcal V\cap Y_n\), a closed subspace of
\(\mathcal V\).  On this space consider
the bilinear form
\[
\mathfrak a_\nu(f,g)
:=
\nu\int_{\mathbb R^2}\mathsf p_G\nabla f\cdot\nabla g
+\langle\Lambda f,g\rangle_Y.
\]
This form is bounded on \(\mathcal V_n\) by
Lemma~\ref{lem:fixed-mode-elementary-tools}. Moreover,
\[
\mathfrak a_\nu(f,f)
\stackrel{\eqref{eq:fixed-mode-Lambda-skew}}{=}\nu\|\nabla f\|_Y^2
\stackrel{\substack{\eqref{eq:fixed-mode-energy-form}\\
\eqref{eq:Gaussian-energy-space}}}{\ge}
\frac{\nu}{2}\|f\|_{\mathcal V}^2.
\]
Thus Lax--Milgram gives a unique \(f^\nu\in\mathcal V_n\) satisfying
\[
\mathfrak a_\nu(f^\nu,g)=\langle q,g\rangle_Y,
\qquad g\in\mathcal V_n,
\]
which is the weak form of
\eqref{eq:regularized-fixed-mode-equation}.
The fixed-mode regularity argument in
\cite{Ga11} shows that this solution belongs
to \(Y_n\cap\mathcal Z\).  We now prove the stronger seminorm bounds
that are uniform as \(\nu\to0\).

~\\
We first obtain a Gaussian bound with a polynomial loss in $\nu$.
Testing against $f^\nu$ in the Gaussian inner product
\eqref{eq:Gaussian-space-Y},
\[
\nu\|\nabla f^\nu\|_Y^2
\stackrel{\substack{\eqref{eq:regularized-fixed-mode-equation}\\
\eqref{eq:fixed-mode-Lambda-skew}}}{=}\langle q,f^\nu\rangle_Y
\le \|q\|_Y\|f^\nu\|_Y
\stackrel{\eqref{eq:fixed-mode-energy-form}}{\le}
\|q\|_Y\|\nabla f^\nu\|_Y.
\]
Consequently,
\begin{equation}\label{eq:regularized-preliminary-energy}
\|f^\nu\|_Y+\|\nabla f^\nu\|_Y
\le 2\nu^{-1}\|q\|_Y.
\end{equation}
Splitting the Biot--Savart kernel at radius one, with
$H^1(\mathbb R^2)\hookrightarrow L^4(\mathbb R^2)$, gives
\[
\|v^{f^\nu}\|_{L^\infty}
\lesssim \bigl(\|f^\nu\|_{L^4}+\|f^\nu\|_{L^1}\bigr)
\stackrel{\eqref{eq:regularized-preliminary-energy}}{\lesssim} \nu^{-1}\|q\|_Y.
\]

~\\
Fix $0<\gamma<1$ and choose
$\max\{\gamma,\tfrac12\}<\eta<1$.
Since $q\in\mathcal Z$ as in \eqref{eq:Gaussian-class-Z} and
$0<\nu\le1$,
\[
\left|
\bigl[\nu(1-\mathcal L)+v^G\cdot\nabla\bigr]f^\nu
\right|
=
|q-v^{f^\nu}\cdot\nabla G|
\lsim{q,\eta} \nu^{-1}e^{-\eta|\xi|^2/4}.
\]
On the other hand, radiality of the Gaussian gives
\[
\bigl[\nu(1-\mathcal L)+v^G\cdot\nabla\bigr]
e^{-\eta|\xi|^2/4}
=
\nu\left(
\eta+\frac{\eta(1-\eta)}4|\xi|^2
\right)e^{-\eta|\xi|^2/4}
\ge \nu\eta e^{-\eta|\xi|^2/4}.
\]
Choose \(M_{q,\eta}>0\), independent of \(\nu\), sufficiently large,
and define
\[
u(\xi):=
f^\nu(\xi)-M_{q,\eta}\nu^{-2}e^{-\eta|\xi|^2/4},
\qquad
u_+:=\max\{u,0\}.
\]
The preceding inequalities imply
\[
\bigl[\nu(1-\mathcal L)+v^G\cdot\nabla\bigr]u\le0.
\]
Since
\[
1-\mathcal L
=
-\mathsf p_G^{-1}\nabla\cdot(\mathsf p_G\nabla),
\]
this is equivalent to
\[
-\nu\nabla\cdot(\mathsf p_G\nabla u)
+\mathsf p_Gv^G\cdot\nabla u
\le0.
\]
Moreover, $\eta>1/2$ and
\eqref{eq:regularized-preliminary-energy} give $u,\nabla u\in Y$.
The positive part therefore satisfies
\[
u_+,\nabla u_+\in Y,
\qquad
\nabla u_+=\mathbf 1_{\{u>0\}}\nabla u
\quad\text{almost everywhere}.
\]

~\\
We show that $u_+=0$.
Let $\chi_R\in C_c^\infty(\mathbb R^2)$ be radial, with
$0\le\chi_R\le1$, equal to one on $B(0,R)$, supported in
$B(0,2R)$, and satisfying $|\nabla\chi_R|\lesssim R^{-1}$.
Testing the preceding differential inequality with
$\chi_R^2u_+$, justified by compactly supported Sobolev
approximation, gives
\[
\nu\int_{\mathbb R^2}
\mathsf p_G\nabla u\cdot\nabla(\chi_R^2u_+)\,d\xi
+
\int_{\mathbb R^2}
\mathsf p_G\chi_R^2(v^G\cdot\nabla u)u_+\,d\xi
\le0.
\]
Because $v^G$ is divergence free and tangent to circles,
while $\mathsf p_G$ and $\chi_R$ are radial,
\[
\nabla\cdot(\mathsf p_G\chi_R^2v^G)=0.
\]
Thus the transport term is
\[
\begin{split}
\int_{\mathbb R^2}
\mathsf p_G\chi_R^2(v^G\cdot\nabla u)u_+\,d\xi
&=
\frac12\int_{\mathbb R^2}
\mathsf p_G\chi_R^2v^G\cdot\nabla(u_+^2)\,d\xi
\\
&=
-\frac12\int_{\mathbb R^2}
\nabla\cdot(\mathsf p_G\chi_R^2v^G)u_+^2\,d\xi
=0.
\end{split}
\]
Since $\nabla u\cdot\nabla u_+=|\nabla u_+|^2$,
\[
\nu\int_{\mathbb R^2}
\mathsf p_G\chi_R^2|\nabla u_+|^2\,d\xi
\le
-2\nu\int_{\mathbb R^2}
\mathsf p_G\chi_Ru_+\nabla u_+\cdot\nabla\chi_R\,d\xi.
\]
Young's inequality then gives
\[
\begin{split}
\nu\int_{\mathbb R^2}
\mathsf p_G\chi_R^2|\nabla u_+|^2\,d\xi
&\le
\frac{\nu}{2}\int_{\mathbb R^2}
\mathsf p_G\chi_R^2|\nabla u_+|^2\,d\xi
\\
&\quad+
2\nu\int_{\mathbb R^2}
\mathsf p_Gu_+^2|\nabla\chi_R|^2\,d\xi.
\end{split}
\]
Absorbing the first term and dividing by $\nu>0$ yields
\[
\int_{\mathbb R^2}
\mathsf p_G\chi_R^2|\nabla u_+|^2\,d\xi
\lesssim
\frac{1}{R^2}
\int_{\{R\le|\xi|\le2R\}}\mathsf p_Gu_+^2\,d\xi.
\]
Since $u_+\in Y$, the right-hand side tends to zero as
$R\to\infty$. Fatou's lemma therefore gives
\[
\int_{\mathbb R^2}\mathsf p_G|\nabla u_+|^2\,d\xi=0.
\]
Hence $\nabla u_+=0$ almost everywhere, so $u_+$ is constant
on $\mathbb R^2$. Its membership in $Y$ forces this constant
to be zero, because
$\int_{\mathbb R^2}\mathsf p_G\,d\xi=\infty$.
Consequently,
\[
f^\nu(\xi)\lsim{q,\eta} \nu^{-2}e^{-\eta|\xi|^2/4}.
\]
Repeating the same argument with $-f^\nu$ in place of $f^\nu$
gives the opposite bound. Thus
\[
|f^\nu(\xi)|
\lsim{q,\eta} \nu^{-2}e^{-\eta|\xi|^2/4},
\]
and hence
\begin{equation}\label{eq:regularized-crude-zero-order}
\mathfrak n_{0,\gamma}(f^\nu)
\lsim{q,\gamma} \nu^{-2},
\qquad 0<\nu\le1.
\end{equation}

~\\
To estimate higher derivatives, write
\[
f^\nu
=
\nu^{-1}(1-\mathcal L)^{-1}
\bigl(q-\Lambda f^\nu\bigr).
\]
For every $k\ge0$ and $0<\gamma<\gamma_1<1$,
\[
\mathfrak n_{k+1,\gamma}(f^\nu)
\stackrel{\substack{\eqref{eq:OU-resolvent-one-derivative}\\
\eqref{eq:Lambda-finite-Gaussian-bound}}}{\lsim{n,k,\gamma,\gamma_1}} \nu^{-1}
\left(
\mathfrak n_{k,\gamma_1}(q)
+\mathfrak n_{k,\gamma_1}(f^\nu)
\right).
\]
Starting from \eqref{eq:regularized-crude-zero-order} and choosing
an intermediate Gaussian exponent at each application, induction
gives
\begin{equation}\label{eq:regularized-derivative-polynomial-loss}
\mathfrak n_{m,\gamma}(f^\nu)
\lsim{q,n,m,\gamma} \nu^{-(m+2)},
\qquad 0<\nu\le1.
\end{equation}

~\\
We now remove this loss by a finite expansion.
Fix $m$ and define
\[
f_0:=f^0=\Lambda^{-1}q,
\qquad
f_j:=-\Lambda^{-1}(1-\mathcal L)f_{j-1},
\qquad 1\le j\le m+2.
\]
We have the identity
\begin{equation}\label{eq:finite-resolvent-expansion-exact}
\bigl[\nu(1-\mathcal L)+\Lambda\bigr]
\left(
f^\nu-\sum_{j=0}^{m+2}\nu^j f_j
\right)
=
-\nu^{m+3}(1-\mathcal L)f_{m+2}.
\end{equation}
Set \(g:=(1-\mathcal L)f_{m+2}\in Y_n\cap\mathcal Z\).  By
linearity and uniqueness, the difference on the left of
\eqref{eq:finite-resolvent-expansion-exact} is
\(-\nu^{m+3}R_\nu g\).  For this fixed source \(g\),
\[
\mathfrak n_{m,\gamma}
\left(
f^\nu-\sum_{j=0}^{m+2}\nu^j f_j
\right)
\stackrel{\eqref{eq:regularized-derivative-polynomial-loss}}{\lsim{q,n,m,\gamma}} \nu^{m+3}\nu^{-(m+2)}
=
\nu.
\]
Moreover, since all $f_j$ are fixed and $\nu^j\le\nu$ for $j\ge1$,
\[
\mathfrak n_{m,\gamma}
\left(
\sum_{j=1}^{m+2}\nu^j f_j
\right)
\le
\nu\sum_{j=1}^{m+2}\mathfrak n_{m,\gamma}(f_j)
\lsim{q,n,m,\gamma} \nu.
\]
Since $f_0=f^0$, the last two estimates imply
\[
\mathfrak n_{m,\gamma}(f^\nu-f^0)
\lsim{q,n,m,\gamma} \nu.
\]
This proves
\eqref{eq:fixed-mode-strong-fixed-source-convergence}.
Adding $\mathfrak n_{m,\gamma}(f^0)$ gives
\eqref{eq:fixed-mode-higher-fixed-source}.

~\\
For $m=0$ and $\gamma=3/4$,
\[
\|f^\nu-f^0\|_Y^2
\stackrel{\substack{\eqref{eq:Gaussian-space-Y}\\
\eqref{eq:fixed-mode-strong-fixed-source-convergence}}}{\lsim{q,n}} \nu^2
\int_{\mathbb R^2}e^{-|\xi|^2/8}\,d\xi
\lsim{q,n} \nu^2,
\]
which proves \eqref{eq:nu-fixed-mode-first-order}.
Taking $m=1$ in
\eqref{eq:fixed-mode-higher-fixed-source} gives
\eqref{eq:nu-fixed-mode-Gaussian}.

\end{proof}
\begin{corollary}
\label{cor:regularized-finite-dimensional-sources}
Fix \(n\ge2\), a finite-dimensional subspace
\(\mathscr E\subset Y_n\cap\mathcal Z\), and any norm
\(\|\cdot\|_{\mathscr E}\) on \(\mathscr E\). Let \(R_\nu\) be the
solution operator in
Lemma~\ref{lem:regularized-fixed-mode-inverse}.
For every \(m\in\mathbb N_0\) and \(0<\gamma<1\),
\begin{equation}\label{eq:regularized-finite-dimensional-bound}
\mathfrak n_{m,\gamma}(R_\nu q)
+\nu^{-1}\mathfrak n_{m,\gamma}
\bigl(R_\nu q-\Lambda^{-1}q\bigr)
\lsim{\mathscr E,n,m,\gamma} \|q\|_{\mathscr E},
\qquad q\in\mathscr E,\quad 0<\nu\le1.
\end{equation}
If \(q^\nu\in C^k([0,T];\mathscr E)\), then
\begin{equation}\label{eq:regularized-finite-dimensional-time-bound}
\mathfrak n_{m,\gamma}
\bigl(\partial_t^jR_\nu q^\nu(t)\bigr)
\lsim{\mathscr E,n,m,\gamma} 
\|\partial_t^jq^\nu(t)\|_{\mathscr E},
\qquad 0\le j\le k.
\end{equation}
\end{corollary}
\begin{proof}
Choose a basis \(q_1,\dots,q_d\) of \(\mathscr E\) and write
\(q=\sum_{\ell=1}^dc_\ell q_\ell\).  Uniqueness and linearity give
\[
R_\nu q=\sum_{\ell=1}^dc_\ell R_\nu q_\ell,
\qquad
\Lambda^{-1}q=\sum_{\ell=1}^dc_\ell\Lambda^{-1}q_\ell.
\]
The finitely many basis estimates and equivalence of norms give
\[
\begin{aligned}
&\mathfrak n_{m,\gamma}(R_\nu q)
+\nu^{-1}\mathfrak n_{m,\gamma}(R_\nu q-\Lambda^{-1}q)
\\
&\qquad\stackrel{\substack{\eqref{eq:fixed-mode-higher-fixed-source}\\
\eqref{eq:fixed-mode-strong-fixed-source-convergence}}}{\lsim{\mathscr E,n,m,\gamma}}
\sum_{\ell=1}^d |c_\ell|
\lsim{\mathscr E,n,m,\gamma} \|q\|_{\mathscr E}.
\end{aligned}
\]
This proves \eqref{eq:regularized-finite-dimensional-bound}.
Since \(R_\nu\) is independent of time,
\[
\mathfrak n_{m,\gamma}(\partial_t^jR_\nu q^\nu)
=\mathfrak n_{m,\gamma}(R_\nu\partial_t^jq^\nu)
\stackrel{\eqref{eq:regularized-finite-dimensional-bound}}{\lsim{\mathscr E,n,m,\gamma}} \|\partial_t^jq^\nu\|_{\mathscr E},
\]
which is \eqref{eq:regularized-finite-dimensional-time-bound}.
\end{proof}
\begin{proposition}
\label{prop:fixed-mode-velocity}
Let $n\ge2$ and $f\in Y_n\cap\mathcal Z$. Then $v^f=K*f$ is smooth and, for every $0<\gamma<1$,
\begin{equation}\label{eq:fixed-mode-velocity-decay}
|v^f(\xi)|
\lsim{n,\gamma}
\frac{1}{(1+|\xi|)^{n+1}}
\|f\|_{L^\infty_\gamma},
\qquad \xi\in\mathbb R^2.
\end{equation}
\end{proposition}

\begin{proof}
The derivative bounds in \eqref{eq:Biot-Savart-finite-derivatives}
ensure smoothness of $v^f$, and
\[
\|v^f\|_{L^\infty}
\stackrel{\eqref{eq:Biot-Savart-finite-derivatives}}{\lsim{\gamma}} \|f\|_{L^\infty_\gamma}.
\]
This proves the desired estimate on $|\xi|\le2$, after increasing
the constant. We therefore fix $\xi$ with $r:=|\xi|\ge2$.
Recall that
\[
|f(\eta)|
\le
\|f\|_{L^\infty_\gamma}e^{-\gamma|\eta|^2/4}.
\]

~\\
We first show that the angular mode of $f$ forces all moments
of degree less than $n$ to vanish on every disk centered at the
origin. In polar coordinates,
\[
\eta=(\rho\cos\theta,\rho\sin\theta),
\qquad
f(\rho,\theta)
=
a(\rho)\cos(n\theta)+b(\rho)\sin(n\theta).
\]
For a multi-index $\alpha=(\alpha_1,\alpha_2)$,
\[
\eta^\alpha
=
\rho^{|\alpha|}
(\cos\theta)^{\alpha_1}(\sin\theta)^{\alpha_2},
\qquad
|\alpha|=\alpha_1+\alpha_2.
\]
The identities
\[
\cos\theta=\frac{e^{i\theta}+e^{-i\theta}}2,
\qquad
\sin\theta=\frac{e^{i\theta}-e^{-i\theta}}{2i},
\]
show that
$(\cos\theta)^{\alpha_1}(\sin\theta)^{\alpha_2}$
is a finite linear combination of $e^{ik\theta}$ with
$|k|\le|\alpha|$.
If $|\alpha|<n$, then $k+n$ and $k-n$ are both nonzero for
every such frequency. Hence
\[
\int_0^{2\pi}e^{i(k+n)\theta}\,d\theta
=
\int_0^{2\pi}e^{i(k-n)\theta}\,d\theta
=0.
\]
It follows that, for every $\rho>0$,
\[
\int_0^{2\pi}
(\cos\theta)^{\alpha_1}(\sin\theta)^{\alpha_2}
f(\rho,\theta)\,d\theta
=0.
\]
Integrating this identity against
$\rho^{|\alpha|+1}\,d\rho$ on $(0,R)$ gives
\begin{equation}\label{eq:fixed-mode-truncated-moments}
\int_{|\eta|\le R}\eta^\alpha f(\eta)\,d\eta=0,
\qquad R>0,\quad |\alpha|<n.
\end{equation}
We next estimate the contribution to the Biot--Savart integral
from $|\eta|\le r/2$.
For every such $\eta$ and every $0\le s\le1$,
\[
|\xi-s\eta|
\ge|\xi|-s|\eta|
\ge r-\frac r2
=\frac r2.
\]
Taylor's theorem therefore applies to $K$ along the segment
joining $\xi$ to $\xi-\eta$:
\[
\frac{d^k}{ds^k}K(\xi-s\eta)
=
(-1)^k
\sum_{|\alpha|=k}
\frac{k!}{\alpha!}\eta^\alpha
\partial^\alpha K(\xi-s\eta),
\qquad
\alpha!:=\alpha_1!\alpha_2!.
\]
Applying the one-dimensional Taylor formula at $s=0$ and
evaluating at $s=1$ therefore gives
\[
K(\xi-\eta)
=
\sum_{|\alpha|<n}
\frac{(-1)^{|\alpha|}}{\alpha!}
\partial^\alpha K(\xi)\eta^\alpha
+\mathcal R_n(\xi,\eta),
\]
where the integral remainder is
\[
\mathcal R_n(\xi,\eta)
=
(-1)^n n
\sum_{|\alpha|=n}
\frac{\eta^\alpha}{\alpha!}
\int_0^1
(1-s)^{n-1}\partial^\alpha K(\xi-s\eta)\,ds.
\]

~\\
The Biot--Savart kernel is homogeneous of degree $-1$:
\[
K(\lambda x)=\lambda^{-1}K(x),
\qquad \lambda>0,\quad x\ne0.
\]
Differentiating with respect to $x$ yields
\[
(\partial^\alpha K)(\lambda x)
=
\lambda^{-1-|\alpha|}(\partial^\alpha K)(x).
\]
Since these derivatives are bounded on the unit circle,
\[
|\partial^\alpha K(x)|
\lsim{n} |x|^{-n-1},
\qquad |\alpha|=n,\quad x\ne0.
\]
In particular, throughout the segment considered above,
\[
|\partial^\alpha K(\xi-s\eta)|
\lsim{n} |\xi-s\eta|^{-n-1}
\lsim{n}  r^{-n-1}.
\]
Since $|\eta^\alpha|\le|\eta|^n$ for $|\alpha|=n$ and
$\int_0^1(1-s)^{n-1}\,ds=1/n$, the remainder satisfies
\[
|\mathcal R_n(\xi,\eta)|
\lsim{n}  r^{-n-1}|\eta|^n,
\qquad |\eta|\le r/2.
\]

~\\
The moment cancellations eliminate the entire Taylor polynomial:
\[
\begin{split}
&\int_{|\eta|\le r/2}K(\xi-\eta)f(\eta)\,d\eta
\\
&\quad=
\sum_{|\alpha|<n}
\frac{(-1)^{|\alpha|}}{\alpha!}
\partial^\alpha K(\xi)
\int_{|\eta|\le r/2}\eta^\alpha f(\eta)\,d\eta
\\
&\qquad\quad+
\int_{|\eta|\le r/2}
\mathcal R_n(\xi,\eta)f(\eta)\,d\eta
\\
&\quad\stackrel{\eqref{eq:fixed-mode-truncated-moments}}{=}
\int_{|\eta|\le r/2}
\mathcal R_n(\xi,\eta)f(\eta)\,d\eta,
\end{split}
\]
with the truncation radius $R=r/2$.
Consequently,
\[
\left|
\int_{|\eta|\le r/2}K(\xi-\eta)f(\eta)\,d\eta
\right|
\lsim{n}  r^{-n-1}
\int_{\mathbb R^2}|\eta|^n|f(\eta)|\,d\eta.
\]
The Gaussian bound controls the remaining absolute moment:
\[
\begin{split}
\int_{\mathbb R^2}|\eta|^n|f(\eta)|\,d\eta
&\le
2\pi\|f\|_{L^\infty_\gamma}
\int_0^\infty
\rho^{n+1}e^{-\gamma\rho^2/4}\,d\rho
\\
&\lsim{n,\gamma} \|f\|_{L^\infty_\gamma}.
\end{split}
\]
We have therefore proved
\[
\left|
\int_{|\eta|\le r/2}K(\xi-\eta)f(\eta)\,d\eta
\right|
\lsim{n,\gamma} r^{-n-1}\|f\|_{L^\infty_\gamma}.
\]

~\\
It remains to estimate the contribution from $|\eta|>r/2$.
We split this region into $|\xi-\eta|<1$ and $|\xi-\eta|\ge1$.

~\\
On $|\eta|>r/2$, we have
\[
|f(\eta)|
\le
e^{-\gamma r^2/16}\|f\|_{L^\infty_\gamma}.
\]
Since $|K(z)|=(2\pi|z|)^{-1}$ is integrable on the unit disk,
\[
\begin{split}
&\int_{\substack{|\eta|>r/2\\|\xi-\eta|<1}}
|K(\xi-\eta)|\,|f(\eta)|\,d\eta
\\
&\quad\le
e^{-\gamma r^2/16}\|f\|_{L^\infty_\gamma}
\int_{|\xi-\eta|<1}|K(\xi-\eta)|\,d\eta
\\
&\quad\lesssim  e^{-\gamma r^2/16}\|f\|_{L^\infty_\gamma}.
\end{split}
\]
On the complementary region $|\xi-\eta|\ge1$,
the kernel is bounded by $1/(2\pi)$. Moreover,
\[
\int_{|\eta|>r/2}e^{-\gamma|\eta|^2/4}\,d\eta
=
2\pi\int_{r/2}^\infty
\rho e^{-\gamma\rho^2/4}\,d\rho
=
\frac{4\pi}{\gamma}e^{-\gamma r^2/16}.
\]
Therefore,
\[
\begin{split}
&\int_{\substack{|\eta|>r/2\\|\xi-\eta|\ge1}}
|K(\xi-\eta)|\,|f(\eta)|\,d\eta
\\
&\quad\le
\frac1{2\pi}\|f\|_{L^\infty_\gamma}
\int_{|\eta|>r/2}e^{-\gamma|\eta|^2/4}\,d\eta
\\
&\quad\lsim{\gamma}  e^{-\gamma r^2/16}\|f\|_{L^\infty_\gamma}.
\end{split}
\]
Combining these two estimates gives
\[
\left|
\int_{|\eta|>r/2}K(\xi-\eta)f(\eta)\,d\eta
\right|
\lsim{\gamma}  e^{-\gamma r^2/16}\|f\|_{L^\infty_\gamma}.
\]
Finally,
\[
\sup_{r\ge2}r^{n+1}e^{-\gamma r^2/16}<\infty,
\]
so this contribution is also
\(\lsim{n,\gamma}r^{-n-1}\|f\|_{L^\infty_\gamma}\).
Thus
\[
|v^f(\xi)|
\lsim{n,\gamma} r^{-n-1}\|f\|_{L^\infty_\gamma},
\qquad r=|\xi|\ge2.
\]
Since $r$ and $1+r$ are comparable on $r\ge2$, combining
this estimate with the initial bound on $|\xi|\le2$ proves
\eqref{eq:fixed-mode-velocity-decay}.
\end{proof}

\subsection{Construction of the profiles}
\paragraph{The Gaussian profile.}
Set \(w^{(0)}:=G\) and \(v^{(0)}:=v^G\).
Since \(\cL G=0\), \(v^G\cdot\nabla G=0\), and
\(\nabla G=-\tfrac{1}{2}\xi G\),
\(\Phi(w^{(0)},\vapp)=\Phi(G,\vapp)
=-\frac12\sqrt{t/\nu}\,(b(\xi,t)\cdot\xi)G\).
Define
\begin{equation}\label{eq:ABC-def}
\begin{aligned}
A(\xi,t)
&:=\bigl(\Sigma(t)\xi\bigr)\cdot\nabla G(\xi)
\\
&\stackrel{\eqref{eq:b-strain-id}}{=}
\frac{1}{4\pi}\abs{\xi}^2G(\xi)
\int_{\Rt}
\frac{\sin(2\psi)}{\abs{\Zn-y}^2}\wapp(y,t)\,dy,
\\
B(\xi,t)
&:=
-\frac{1}{4\pi}\abs{\xi}^3G(\xi)
\int_{\Rt}
\frac{\sin(3\psi)}{\abs{\Zn-y}^3}\wapp(y,t)\,dy,
\\
C(\xi,t)
&:=
\frac{1}{4\pi}\abs{\xi}^4G(\xi)
\int_{\Rt}
\frac{\sin(4\psi)}{\abs{\Zn-y}^4}\wapp(y,t)\,dy.
\end{aligned}
\end{equation}
The formulas show that \(A\), \(B\), and \(C\) are pure angular
harmonics of orders two, three, and four, respectively.
Define $\mathcal R_0$ by
\begin{equation}\label{eq:naive-residual-expansion}
\Phi(w^{(0)},\vapp)
=t\left(A+(\nu t)^{\frac{1}{2}}B+(\nu t) C\right)+\mathcal R_0.
\end{equation}
\begin{lemma}
For every \(\gamma\in(\tfrac12,1)\),
\begin{equation}\label{eq:R0-target}
\norm[L^\infty_\gamma]{\mathcal R_0(\cdot,t)}
\lsim{\gamma} (\nu t)^{\frac{3}{2}},
\qquad
0<t\le T,\quad 0<\nu\le\nu_{0,T}.
\end{equation}
\end{lemma}
\begin{proof}
Write \(r=\abs\xi\). If \(\sqrt{\nu t}\,r\le1\), then
\[
\begin{aligned}
\Phi(G,\vapp)
&\stackrel{\substack{\eqref{eq:b-multipole}\\\eqref{eq:ABC-def}}}{=}
t\left(A+(\nu t)^{\frac{1}{2}}B+(\nu t) C\right)
-\frac12\sqrt{\frac t\nu}\,r_bG,
\\
\mathcal R_0(\xi,t)
&\stackrel{\eqref{eq:naive-residual-expansion}}{=}
-\frac12\sqrt{\frac t\nu}\,r_b(\xi,t)G(\xi).
\end{aligned}
\]
Hence
\[
\abs{\mathcal R_0(\xi,t)}
\stackrel{\eqref{eq:b-tail}}{\lesssim} t(\nu t)^{3/2}r^5G(\xi)
\lsim{\gamma} (\nu t)^{3/2}e^{-\gamma r^2/4}.
\]

~\\
Suppose now that \(\sqrt{\nu t}\,r>1\), equivalently
\(r>(\sqrt{\nu t})^{-1}\). Proposition~\ref{prop:regular-corrector}
supplies the required uniform coefficient bounds, and
\[
\begin{aligned}
\abs{\mathcal R_0(\xi,t)}
&\stackrel{\substack{\eqref{eq:b-global}\\\eqref{eq:ABC-def}}}{\lesssim} \left(
(\nu t)^{-\frac{1}{2}}r+r^2+(\nu t)^{\frac{1}{2}}r^3+(\nu t)r^4
\right)e^{-\frac{r^2}{4}}
\\
&\lesssim (\nu t)^{-\frac{1}{2}}(1+r)^4e^{-\frac{r^2}{4}}.
\end{aligned}
\]
Since \(r>(\sqrt{\nu t})^{-1}\),
\[
\begin{aligned}
(\nu t)^{-\frac{1}{2}}(1+r)^4e^{-\frac{(1-\gamma)r^2}{4}}
&\lsim{\gamma} (\nu t)^{-\frac{1}{2}}
e^{-\frac{1-\gamma}{8\nu t}}
\lsim{\gamma} (\nu t)^{\frac{3}{2}}.
\end{aligned}
\]
Multiplying by \(e^{-\frac{\gamma r^2}{4}}\) proves
\eqref{eq:R0-target} in the outer region.
\end{proof}
\paragraph{Step 1: the second angular mode.}
Introduce
\begin{equation}\label{eq:first-partial-approximation}
w^{(1)}:=G+(\nu t) F_{\ne0},
\qquad
v^{(1)}:=v^G+(\nu t) v^{F_{\ne0}}.
\end{equation}
By \eqref{eq:core-residual-operator},
\eqref{eq:first-partial-approximation}, and
\eqref{eq:naive-residual-expansion}, we obtain
\begin{equation}\label{eq:first-uncancelled-residual}
\begin{aligned}
\Phi(w^{(1)},\vapp)
={} &
\mathcal R_0+t(\nu t)^{\frac12}B
\\
&+t\left\{
A+\Lambda F_{\ne0}
+\nu(1-\cL)F_{\ne0}
\right\}
\\
&+t(\nu t)\left\{
\partial_tF_{\ne0}
+
C+
\bigl(v^{F_{\ne0}}+\Sigma(t)\xi\bigr)
\cdot\nabla F_{\ne0}
\right\}
\\
&+\left(
\sqrt{\frac t\nu}\,b-t\Sigma(t)\xi
\right)\cdot\nabla\bigl((\nu t)F_{\ne0}\bigr).
\end{aligned}
\end{equation}
We define \(F_{\ne0}\) by
\begin{equation}\label{eq:F-nonzero-equation}
\bigl[\nu(1-\cL)+\Lambda\bigr]F_{\ne0}(\cdot,t)
=-A(\cdot,t).
\end{equation}
Substituting \eqref{eq:F-nonzero-equation} into
\eqref{eq:first-uncancelled-residual} gives
\begin{equation}\label{eq:first-corrected-residual}
\begin{aligned}
\Phi(w^{(1)},\vapp)
={} &
\mathcal R_0+t(\nu t)^{\frac{1}{2}}B
\\
&+(\nu t)\left\{
t\partial_tF_{\ne0}
+t\left[
C+
\bigl(v^{F_{\ne0}}+\Sigma(t)\xi\bigr)
\cdot\nabla F_{\ne0}
\right]
\right\}
\\
&+\left(
\sqrt{\frac t\nu}\,b-t\Sigma(t)\xi
\right)\cdot\nabla((\nu t) F_{\ne0}).
\end{aligned}
\end{equation}
The correction \((\nu t)F_{\ne0}\) also produces the term
\(t(\nu t)(v^{F_{\ne0}}+\Sigma(t)\xi)\cdot\nabla F_{\ne0}\).
Since \(F_{\ne0},\partial_tF_{\ne0}\in L^2_{\pw_G,2}\) and
\(C\in L^2_{\pw_G,4}\), with the angular spaces as in
\eqref{eq:Gaussian-angular-spaces},
and the product of two second angular modes contains only the radial and
fourth modes, the radial part of the order-\(t(\nu t)\) defect in
 \eqref{eq:first-corrected-residual} is exactly
\begin{equation}\label{eq:first-radial-obstruction}
t(\nu t)\Pi_{\mathrm{rad}}\left[
\bigl(v^{F_{\ne0}}+\Sigma(t)\xi\bigr)
\cdot\nabla F_{\ne0}
\right].
\end{equation}
\paragraph{Step 2: the radial profile.}
Introduce the radial profile by
\begin{equation}\label{eq:second-partial-approximation}
w^{(2)}:=w^{(1)}+(\nu t)F_{\mathrm{rad}},
\qquad
v^{(2)}:=v^{(1)}+(\nu t)v^{F_{\mathrm{rad}}}.
\end{equation}
Since \(F_{\mathrm{rad}}\) is radial,
\begin{equation}\label{eq:Lambda-radial-kernel}
\Lambda F_{\mathrm{rad}}
=v^G\cdot\nabla F_{\mathrm{rad}}
+v^{F_{\mathrm{rad}}}\cdot\nabla G
=0.
\end{equation}
Substituting \eqref{eq:second-partial-approximation} and
\eqref{eq:Lambda-radial-kernel} into
\eqref{eq:first-corrected-residual} gives
\begin{equation}\label{eq:second-uncancelled-residual}
\begin{aligned}
\Phi(w^{(2)},\vapp)
={} &
\mathcal R_0+t(\nu t)^{\frac{1}{2}}B
\\
&+(\nu t)\Bigl[
t\partial_tF_{\ne0}+tC
+t\bigl(v^{F_{\ne0}}+\Sigma(t)\xi\bigr)\cdot\nabla F_{\ne0}
\\
&\qquad
+t\partial_tF_{\mathrm{rad}}+F_{\mathrm{rad}}-\cL F_{\mathrm{rad}}
\\
&\qquad
+t\bigl(v^{F_{\ne0}}+\Sigma(t)\xi\bigr)
  \cdot\nabla F_{\mathrm{rad}}
+t v^{F_{\mathrm{rad}}}\cdot\nabla F_{\ne0}
\Bigr]
\\
&+\left(
\sqrt{\frac t\nu}\,b-t\Sigma(t)\xi
\right)\cdot
\left[(\nu t)\nabla F_{\ne0}+(\nu t)\nabla F_{\mathrm{rad}}\right].
\end{aligned}
\end{equation}
Moreover,
\begin{equation}\label{eq:nonlinear-radial-projection}
\begin{aligned}
&\Pi_{\mathrm{rad}}\Bigl[
\bigl(v^{F_{\ne0}}+\Sigma(t)\xi\bigr)\cdot\nabla F_{\ne0}
+\bigl(v^{F_{\ne0}}+\Sigma(t)\xi\bigr)\cdot\nabla F_{\mathrm{rad}}
+v^{F_{\mathrm{rad}}}\cdot\nabla F_{\ne0}
\Bigr]
\\
&\qquad=
\Pi_{\mathrm{rad}}\left[
\bigl(v^{F_{\ne0}}+\Sigma(t)\xi\bigr)\cdot\nabla F_{\ne0}
\right].
\end{aligned}
\end{equation}
We choose
\(F_{\mathrm{rad}}\) as the solution of
\begin{equation}\label{eq:Frad-equation}
\begin{cases}
\displaystyle
t\partial_tF_{\mathrm{rad}}
+(1-\cL)F_{\mathrm{rad}}
=
-t\Pi_{\mathrm{rad}}\left[
\bigl(v^{F_{\ne0}}+\Sigma(t)\xi\bigr)
\cdot\nabla F_{\ne0}
\right],
\\[1mm]
F_{\mathrm{rad}}(\cdot,0)=0.
\end{cases}
\end{equation}
By \eqref{eq:nonlinear-radial-projection} and
\eqref{eq:Frad-equation}, the radial cancellation leaves
\begin{equation}\label{eq:radial-corrected-residual}
\begin{aligned}
\Phi(w^{(2)},\vapp)
=
{}&\mathcal R_0+t(\nu t)^{\frac{1}{2}}B
\\
&+t(\nu t)\Biggl\{
\partial_tF_{\ne0}+C
+\Pi_{\ne0}\Bigl[
\bigl(v^{F_{\ne0}}+\Sigma(t)\xi\bigr)\cdot\nabla F_{\ne0}
\\
&\qquad\qquad
+\bigl(v^{F_{\ne0}}+\Sigma(t)\xi\bigr)\cdot\nabla F_{\mathrm{rad}}
+v^{F_{\mathrm{rad}}}\cdot\nabla F_{\ne0}
\Bigr]
\Biggr\}
\\
&+\left(
\sqrt{\frac t\nu}\,b-t\Sigma(t)\xi
\right)\cdot
\left[(\nu t)\nabla F_{\ne0}+(\nu t)\nabla F_{\mathrm{rad}}\right].
\end{aligned}
\end{equation}
Thus the radial part of the order-\((\nu t)\) defect has been cancelled,
while the term \(t(\nu t)^{\frac{1}{2}}B\) is unchanged.

\paragraph{Step 3: the profile at order \((\nu t)^{\frac{3}{2}}\).}
Set
\begin{equation}\label{eq:third-partial-approximation}
\begin{aligned}
w^{(3)}
&:=w^{(2)}+(\nu t)^{\frac{3}{2}}H,
\\
v^{(3)}
&:=v^{(2)}+(\nu t)^{\frac{3}{2}}v^H.
\end{aligned}
\end{equation}
Substituting \eqref{eq:third-partial-approximation} into
\eqref{eq:radial-corrected-residual} gives
\begin{equation}\label{eq:third-uncancelled-residual}
\begin{aligned}
&\Phi(w^{(3)},\vapp)\\
={} &
\mathcal R_0+t(\nu t)^{\frac{1}{2}}(B+\Lambda H)
\\
&+t(\nu t)\Biggl\{
\partial_tF_{\ne0}+C
+\Pi_{\ne0}\Bigl[
\bigl(v^{F_{\ne0}}+\Sigma(t)\xi\bigr)\cdot\nabla F_{\ne0}
\\
&\qquad\qquad
+\bigl(v^{F_{\ne0}}+\Sigma(t)\xi\bigr)\cdot\nabla F_{\mathrm{rad}}
+v^{F_{\mathrm{rad}}}\cdot\nabla F_{\ne0}
\Bigr]
\Biggr\}
\\
&+(\nu t)^{\frac{3}{2}}
\left(t\partial_tH+\frac{3}{2}H-\cL H\right)
\\
&+t(\nu t)^{\frac{3}{2}}\Bigl[
\bigl(v^{F_{\ne0}}+v^{F_{\mathrm{rad}}}+\Sigma(t)\xi\bigr)
\cdot\nabla H
\\
&\qquad\qquad
+v^H\cdot\nabla F_{\ne0}
+v^H\cdot\nabla F_{\mathrm{rad}}
\Bigr]
\\
&+t(\nu t)^2v^H\cdot\nabla H
\\
&+\left(
\sqrt{\frac t\nu}\,b-t\Sigma(t)\xi
\right)\cdot\Bigl[
(\nu t)\nabla F_{\ne0}
+(\nu t)\nabla F_{\mathrm{rad}}
+(\nu t)^{\frac{3}{2}}\nabla H
\Bigr].
\end{aligned}
\end{equation}
Choose \(H\) to solve
\begin{equation}\label{eq:H-equation}
\Lambda H=-B.
\end{equation}
Then the order-\((\nu t)^{\frac{1}{2}}\) term in
\eqref{eq:third-uncancelled-residual} vanishes, and
\begin{equation}\label{eq:third-corrected-residual}
\begin{aligned}
&\Phi(w^{(3)},\vapp)\\
={} &\mathcal R_0+t(\nu t)\Biggl\{
\partial_tF_{\ne0}+C
+\Pi_{\ne0}\Bigl[
\bigl(v^{F_{\ne0}}+\Sigma(t)\xi\bigr)\cdot\nabla F_{\ne0}
\\
&\qquad\qquad
+\bigl(v^{F_{\ne0}}+\Sigma(t)\xi\bigr)\cdot\nabla F_{\mathrm{rad}}
+v^{F_{\mathrm{rad}}}\cdot\nabla F_{\ne0}
\Bigr]
\Biggr\}
\\
&+(\nu t)^{\frac{3}{2}}
\left(t\partial_tH+\frac{3}{2}H-\cL H\right)
\\
&+t(\nu t)^{\frac{3}{2}}\Bigl[
\bigl(v^{F_{\ne0}}+v^{F_{\mathrm{rad}}}+\Sigma(t)\xi\bigr)
\cdot\nabla H
\\
&\qquad\qquad
+v^H\cdot\nabla F_{\ne0}
+v^H\cdot\nabla F_{\mathrm{rad}}
\Bigr]
\\
&+t(\nu t)^2v^H\cdot\nabla H
\\
&+\left(
\sqrt{\frac t\nu}\,b-t\Sigma(t)\xi
\right)\cdot\Bigl[
(\nu t)\nabla F_{\ne0}
+(\nu t)\nabla F_{\mathrm{rad}}
+(\nu t)^{\frac{3}{2}}\nabla H
\Bigr].
\end{aligned}
\end{equation}

\paragraph{Step 4: the order-\((\nu t)^2\) profile.}
Add one final profile:
\begin{equation}\label{eq:fourth-partial-approximation}
w^{(4)}:=w^{(3)}+(\nu t)^2J,
\qquad
v^{(4)}:=v^{(3)}+(\nu t)^2v^J.
\end{equation}
Substituting \eqref{eq:fourth-partial-approximation} into
\eqref{eq:third-corrected-residual} gives
\begin{equation}\label{eq:fourth-uncancelled-residual}
\begin{aligned}
&\Phi(w^{(4)},\vapp)\\
={} &\mathcal R_0
+t(\nu t)\Biggl\{
\partial_tF_{\ne0}+C+\Lambda J+\Pi_{\ne0}\Bigl[
\bigl(v^{F_{\ne0}}+\Sigma(t)\xi\bigr)\cdot\nabla F_{\ne0}
\\
&\qquad\qquad
+\bigl(v^{F_{\ne0}}+\Sigma(t)\xi\bigr)\cdot\nabla F_{\mathrm{rad}}
+v^{F_{\mathrm{rad}}}\cdot\nabla F_{\ne0}
\Bigr]
\Biggr\}
\\
&+(\nu t)^{\frac{3}{2}}
\left(t\partial_tH+\frac{3}{2}H-\cL H\right)
+(\nu t)^2\left(t\partial_tJ+2J-\cL J\right)
\\
&+t(\nu t)^{\frac{3}{2}}\Bigl[
\bigl(v^{F_{\ne0}}+v^{F_{\mathrm{rad}}}+\Sigma(t)\xi\bigr)
\cdot\nabla H
+v^H\cdot\nabla F_{\ne0}
+v^H\cdot\nabla F_{\mathrm{rad}}
\Bigr]
\\
&+t(\nu t)^2\Bigl[
\bigl(v^{F_{\ne0}}+v^{F_{\mathrm{rad}}}+\Sigma(t)\xi\bigr)
\cdot\nabla J
+v^J\cdot\nabla F_{\ne0}
+v^J\cdot\nabla F_{\mathrm{rad}}
+v^H\cdot\nabla H
\Bigr]
\\
&+t(\nu t)^{\frac{5}{2}}
\left(v^H\cdot\nabla J+v^J\cdot\nabla H\right)
+t(\nu t)^3v^J\cdot\nabla J
\\
&+\left(
\sqrt{\frac t\nu}\,b-t\Sigma(t)\xi
\right)\cdot\Bigl[
(\nu t)\nabla F_{\ne0}
+(\nu t)\nabla F_{\mathrm{rad}}
+(\nu t)^{\frac{3}{2}}\nabla H
+(\nu t)^2\nabla J
\Bigr].
\end{aligned}
\end{equation}
Let \(\mathcal Q_J\) denote the order-\((\nu t)\) source in braces,
apart from \(\Lambda J\); explicitly,
\begin{equation}\label{eq:J-source-def}
\begin{aligned}
\mathcal Q_J
:={}&C+\partial_tF_{\ne0}\\
&+\Pi_{\ne0}\Bigl[
  \bigl(v^{F_{\ne0}}+\Sigma(t)\xi\bigr)
  \cdot\nabla\bigl(F_{\ne0}+F_{\mathrm{rad}}\bigr)
  +v^{F_{\mathrm{rad}}}\cdot\nabla F_{\ne0}
\Bigr].
\end{aligned}
\end{equation}
We choose \(J\) by the unregularized fixed-mode equation
\begin{equation}\label{eq:J-equation}
\Lambda J=-\mathcal Q_J.
\end{equation}
Thus \(\mathcal Q_J+\Lambda J=0\). The terms involving
\(\mathcal Q_J\) and the linear part of \(J\) in
\eqref{eq:fourth-uncancelled-residual} therefore satisfy
\begin{equation}\label{eq:J-linear-residual-cancellation}
\begin{aligned}
&t(\nu t)\bigl(\mathcal Q_J+\Lambda J\bigr)
+(\nu t)^2\bigl(t\partial_tJ+2J-\cL J\bigr)
\\
&\qquad\stackrel{\eqref{eq:J-equation}}{=}(\nu t)^2
\bigl(t\partial_tJ+2J-\cL J\bigr).
\end{aligned}
\end{equation}
We estimate the bracket below and use \((\nu t)^2\le(\nu t)^{3/2}\)
for \(\nu t\le1\). The dependence of \(J\) and \(F_{\ne0}\) on
\(\nu\) is suppressed in the notation.
The final approximation is
\begin{equation}\label{eq:high-app}
\begin{aligned}
w_{2,\mathrm{app}}
&:=w^{(4)}
=G+(\nu t)F_{\ne0}+(\nu t)F_{\mathrm{rad}}
+(\nu t)^{\frac{3}{2}}H+(\nu t)^2J,
\\
v_{2,\mathrm{app}}
&:=v^{(4)}
=v^G+(\nu t)v^{F_{\ne0}}+(\nu t)v^{F_{\mathrm{rad}}}
+(\nu t)^{\frac{3}{2}}v^H+(\nu t)^2v^J.
\end{aligned}
\end{equation}
Combining \eqref{eq:fourth-uncancelled-residual} and
\eqref{eq:J-linear-residual-cancellation}, we obtain the final residual
\begin{equation}\label{eq:profile-residual-reduction}
\begin{aligned}
&\Phi(w_{2,\mathrm{app}},\vapp)\\
={} &
\mathcal R_0
+(\nu t)^{\frac{3}{2}}
\left(t\partial_tH+\frac{3}{2}H-\cL H\right)
+(\nu t)^2\left(t\partial_tJ+2J-\cL J\right)
\\
&+t(\nu t)^{\frac{3}{2}}\Bigl[
\bigl(v^{F_{\ne0}}+v^{F_{\mathrm{rad}}}+\Sigma(t)\xi\bigr)
\cdot\nabla H
+v^H\cdot\nabla F_{\ne0}
+v^H\cdot\nabla F_{\mathrm{rad}}
\Bigr]
\\
&+t(\nu t)^2\Bigl[
\bigl(v^{F_{\ne0}}+v^{F_{\mathrm{rad}}}+\Sigma(t)\xi\bigr)
\cdot\nabla J
+v^J\cdot\nabla F_{\ne0}
+v^J\cdot\nabla F_{\mathrm{rad}}
+v^H\cdot\nabla H
\Bigr]
\\
&+t(\nu t)^{\frac{5}{2}}
\left(v^H\cdot\nabla J+v^J\cdot\nabla H\right)
+t(\nu t)^3v^J\cdot\nabla J
\\
&+\left(
\sqrt{\frac t\nu}\,b-t\Sigma(t)\xi
\right)\cdot\Bigl[
(\nu t)\nabla F_{\ne0}
+(\nu t)\nabla F_{\mathrm{rad}}
+(\nu t)^{\frac{3}{2}}\nabla H
+(\nu t)^2\nabla J
\Bigr].
\end{aligned}
\end{equation}
\subsection{Uniform profile estimates}
\paragraph{Finite-dimensional source spaces.}
For \(n=2,3,4\), write \(\xi=(r\cos\theta,r\sin\theta)\) and set
\begin{equation}\label{eq:finite-harmonic-source-spaces}
\mathscr E_n
:=
\operatorname{span}\left\{
r^nG(r)\cos(n\theta),\
r^nG(r)\sin(n\theta)
\right\}
\stackrel{\substack{\eqref{eq:Gaussian-angular-spaces}\\
\eqref{eq:Gaussian-class-Z}}}{\subset}Y_n\cap\mathcal Z.
\end{equation}
The three source profiles satisfy
\[
(A(\cdot,t),B(\cdot,t),C(\cdot,t))
\stackrel{\substack{\eqref{eq:psi-def}\\\eqref{eq:ABC-def}}}{\in}
\mathscr E_2\times\mathscr E_3\times\mathscr E_4.
\]
Fix any norms on these three spaces.  Differentiation of their
coefficients under the separated Biot--Savart integrals gives
\begin{equation}\label{eq:ABC-finite-dimensional-time-bounds}
\sup_{\substack{0\le t\le T\\0<\nu\le\nu_{0,T}}}
\left[
\max_{0\le j\le2}
\|\partial_t^jA(\cdot,t)\|_{\mathscr E_2}
+
\max_{0\le j\le1}
\left(
\|\partial_t^jB(\cdot,t)\|_{\mathscr E_3}
+
\|\partial_t^jC(\cdot,t)\|_{\mathscr E_4}
\right)
\right]
\stackrel{\substack{\eqref{eq:regular-corrector-bounds}\\
\eqref{eq:approx-separation}}}{\lesssim} 1.
\end{equation}
Since \(A\) depends on \(\nu\) through \(\Zn\) and \(\wapp\),
we use Corollary~\ref{cor:regularized-finite-dimensional-sources}
to obtain uniform bounds.

\paragraph{Step 1. Estimates for \(F_{\ne0}\), \(H\), and their velocities.}
\[
F_{\ne0}(\cdot,t)
\stackrel{\eqref{eq:F-nonzero-equation}}{=}-R_\nu A(\cdot,t).
\]
For every
\(m\in\mathbb N_0\) and \(0<\gamma<1\),
\begin{equation}\label{eq:F-regularized-space-time-bounds}
\max_{0\le j\le2}
\sup_{\substack{0\le t\le T\\0<\nu\le\nu_{0,T}}}
\mathfrak n_{m,\gamma}
\bigl(\partial_t^jF_{\ne0}(\cdot,t)\bigr)
\stackrel{\substack{\eqref{eq:regularized-finite-dimensional-time-bound}\\
\eqref{eq:ABC-finite-dimensional-time-bounds}}}{\lsim{m,\gamma}} 1.
\end{equation}
The uniform inverse bound keeps \((\nu t)F_{\ne0}\) at order
\(\nu t\); the preliminary estimate
\eqref{eq:regularized-preliminary-energy} loses a factor \(\nu^{-1}\).

The unregularized equation for \(H\) gives
\[
H\stackrel{\eqref{eq:H-equation}}{=}-\Lambda^{-1}B,
\qquad
\partial_tH\stackrel{\eqref{eq:H-equation}}{=}-\Lambda^{-1}\partial_tB.
\]
Consequently,
\begin{equation}\label{eq:H-unregularized-space-time-bounds}
\sup_{\substack{0\le t\le T\\0<\nu\le\nu_{0,T}}}
\left(
\mathfrak n_{2,\gamma}(H(\cdot,t))
+
\mathfrak n_{0,\gamma}(\partial_tH(\cdot,t))
\right)
\stackrel{\substack{\eqref{eq:fixed-mode-higher-unregularized}\\
\eqref{eq:ABC-finite-dimensional-time-bounds}}}{\lsim{\gamma}} 1.
\end{equation}
Moreover,
\begin{equation}\label{eq:F-H-mode-membership}
F_{\ne0}(\cdot,t)\in L^2_{\pw_G,2}\cap\cZ,
\qquad
H(\cdot,t)\in L^2_{\pw_G,3}\cap\cZ.
\end{equation}
For every \(0<\gamma<1\), uniformly for
\(0\le t\le T\) and \(0<\nu\le\nu_{0,T}\),
\begin{equation}\label{eq:F-H-Gallay-bounds}
\begin{aligned}
&\norm[L^\infty_\gamma]{F_{\ne0}(\cdot,t)}
+\norm[L^\infty_\gamma]{\nabla F_{\ne0}(\cdot,t)}
\\
&\qquad
+\norm[L^\infty_\gamma]{H(\cdot,t)}
+\norm[L^\infty_\gamma]{\nabla H(\cdot,t)}
\lsim{\gamma} 1.
\end{aligned}
\end{equation}
For the second angular mode \(F_{\ne0}\),
\begin{equation}\label{eq:F-nonzero-velocity-bound}
\abs{v^{F_{\ne0}}(\xi,t)}
\stackrel{\substack{\eqref{eq:fixed-mode-velocity-decay}\\
\eqref{eq:F-H-Gallay-bounds}}}{\lesssim}
\frac{1}{(1+\abs{\xi}^2)^{\frac{3}{2}}},
\qquad
\xi\in\Rt,\quad 0\le t\le T.
\end{equation}
Similarly, the third angular mode of \(H\) gives
\begin{equation}\label{eq:H-velocity-bound}
\abs{v^H(\xi,t)}
\stackrel{\substack{\eqref{eq:fixed-mode-velocity-decay}\\
\eqref{eq:F-H-Gallay-bounds}}}{\lesssim}
\frac{1}{(1+\abs{\xi}^2)^2},
\qquad
\xi\in\Rt,\quad 0\le t\le T.
\end{equation}
\paragraph{Step 2. Estimate for \(F_{\mathrm{rad}}\).}
The regularized fixed-mode inverse does not control this profile:
\(\Lambda F_{\mathrm{rad}}
\stackrel{\eqref{eq:Lambda-radial-kernel}}{=}0\).
The radial evolution has the Duhamel representation
\begin{equation}\label{eq:Frad-Duhamel}
\begin{aligned}
F_{\mathrm{rad}}(\cdot,t)
\stackrel{\eqref{eq:Frad-equation}}{=}{}&-\int_0^t\frac{s}{t}
e^{\left(\log\frac{t}{s}\right)\cL}
\\
&\quad\times
\Pi_{\mathrm{rad}}\left[
\bigl(v^{F_{\ne0}}(\cdot,s)+\Sigma(s)\xi\bigr)
\cdot\nabla F_{\ne0}(\cdot,s)
\right]\,ds.
\end{aligned}
\end{equation}
We claim that, for every
\(\frac{1}{2}<\gamma<1\),
\begin{equation}\label{eq:Frad-Gaussian-bound}
\norm[L^\infty_\gamma]{F_{\mathrm{rad}}(\cdot,t)}
+\norm[L^\infty_\gamma]{\nabla F_{\mathrm{rad}}(\cdot,t)}
\lsim{\gamma} t,
\qquad 0\le t\le T.
\end{equation}
Fix \(1/2<\gamma<\gamma_1<\gamma_2<1\).  Since
\(\Pi_{\mathrm{rad}}\) is an angular average,
\eqref{eq:F-H-Gallay-bounds}, \eqref{eq:F-nonzero-velocity-bound},
and \eqref{eq:Sigma-time-bound} give
\[
\left|
\Pi_{\mathrm{rad}}\left[
\bigl(v^{F_{\ne0}}(\cdot,t)+\Sigma(t)\xi\bigr)
\cdot\nabla F_{\ne0}(\cdot,t)
\right](\xi)
\right|
\lsim{\gamma_2} (1+\abs{\xi})
e^{-\frac{\gamma_2\abs{\xi}^2}{4}}.
\]
Consequently,
\begin{equation}\label{eq:radial-source-Gaussian-bound}
\sup_{0\le t\le T}
\norm[L^\infty_{\gamma_1}]{
\Pi_{\mathrm{rad}}\left[
\bigl(v^{F_{\ne0}}(\cdot,t)+\Sigma(t)\xi\bigr)
\cdot\nabla F_{\ne0}(\cdot,t)
\right]}
\lsim{\gamma_1} 1.
\end{equation}
We now estimate the right-hand side of \eqref{eq:Frad-Duhamel}.
The kernel of \(e^{\tau\cL}\) is
\[
e^{\tau\cL}g(\xi)
\stackrel{\eqref{eq:OU-kernel-explicit}}{=}
\frac{1}{4\pi\left(1-e^{-\tau}\right)}
\int_{\Rt}
\exp\left(
-\frac{\abs{\xi-e^{-\frac{\tau}{2}}\eta}^2}
{4\left(1-e^{-\tau}\right)}
\right)g(\eta)\,d\eta.
\]
Since \(0<\gamma_1<1\),
\(\gamma_1\le e^{-\tau}+\gamma_1(1-e^{-\tau})\le1\).
For every \(\tau>0\), completing the square gives
\[
\begin{aligned}
&\frac{\abs{\xi-e^{-\frac{\tau}{2}}\eta}^2}
{1-e^{-\tau}}
+\gamma_1\abs{\eta}^2
\\
&\quad=
\frac{e^{-\tau}+\gamma_1\left(1-e^{-\tau}\right)}
{1-e^{-\tau}}
\left|
\eta-
\frac{e^{-\frac{\tau}{2}}}
{e^{-\tau}+\gamma_1\left(1-e^{-\tau}\right)}\xi
\right|^2
\\
&\qquad\quad+
\frac{\gamma_1}
{e^{-\tau}+\gamma_1\left(1-e^{-\tau}\right)}
\abs{\xi}^2.
\end{aligned}
\]
The weighted norm \eqref{eq:Gaussian-supremum-norm} gives
$
\abs{g(\eta)}
\le
\norm[L^\infty_{\gamma_1}]{g}
e^{-\frac{\gamma_1\abs{\eta}^2}{4}}
$
and hence
\[
\begin{aligned}
\abs{e^{\tau\cL}g(\xi)}
\le{}&
\frac{\norm[L^\infty_{\gamma_1}]{g}}
{4\pi\left(1-e^{-\tau}\right)}
\exp\left(
-\frac{\gamma_1\abs{\xi}^2}
{4\left(e^{-\tau}+\gamma_1(1-e^{-\tau})\right)}
\right)
\\
&\times
\int_{\Rt}
\exp\left(
-\frac{e^{-\tau}+\gamma_1(1-e^{-\tau})}
{4(1-e^{-\tau})}
\left|
\eta-
\frac{e^{-\frac{\tau}{2}}}
{e^{-\tau}+\gamma_1(1-e^{-\tau})}\xi
\right|^2
\right)d\eta
\\
={}&
\frac{1}{e^{-\tau}+\gamma_1(1-e^{-\tau})}
\exp\left(
-\frac{\gamma_1\abs{\xi}^2}
{4\left(e^{-\tau}+\gamma_1(1-e^{-\tau})\right)}
\right)
\norm[L^\infty_{\gamma_1}]{g}.
\end{aligned}
\]
Consequently,
\[
\begin{aligned}
&e^{\frac{\gamma\abs{\xi}^2}{4}}
\abs{e^{\tau\cL}g(\xi)}
\\
&\quad\le
\frac{1}{e^{-\tau}+\gamma_1(1-e^{-\tau})}
\exp\left[
-\frac{1}{4}
\left(
\frac{\gamma_1}
{e^{-\tau}+\gamma_1(1-e^{-\tau})}
-\gamma
\right)\abs{\xi}^2
\right]
\norm[L^\infty_{\gamma_1}]{g}
\\
&\quad\le
\frac{1}{\gamma_1}
\norm[L^\infty_{\gamma_1}]{g},
\end{aligned}
\]
because
\(\gamma_1/[e^{-\tau}+\gamma_1(1-e^{-\tau})]-\gamma
\ge\gamma_1-\gamma>0\).
Thus
\begin{equation}\label{eq:radial-OU-zero-order-bound}
\norm[L^\infty_\gamma]{e^{\tau\cL}g}
\le
\frac{1}{\gamma_1}
\norm[L^\infty_{\gamma_1}]{g}.
\end{equation}

Differentiating the kernel with respect to \(\xi\) gives
\[
\begin{aligned}
\nabla e^{\tau\cL}g(\xi)
={}&-
\frac{1}{8\pi\left(1-e^{-\tau}\right)^2}
\\
&\quad\times
\int_{\Rt}
\left(\xi-e^{-\frac{\tau}{2}}\eta\right)
\exp\left(
-\frac{\abs{\xi-e^{-\frac{\tau}{2}}\eta}^2}
{4\left(1-e^{-\tau}\right)}
\right)g(\eta)\,d\eta.
\end{aligned}
\]
Moreover,
\[
\begin{aligned}
\xi-e^{-\frac{\tau}{2}}\eta
={}&
\frac{\gamma_1\left(1-e^{-\tau}\right)}
{e^{-\tau}+\gamma_1\left(1-e^{-\tau}\right)}\xi
\\
&-e^{-\frac{\tau}{2}}
\left(
\eta-
\frac{e^{-\frac{\tau}{2}}}
{e^{-\tau}+\gamma_1\left(1-e^{-\tau}\right)}\xi
\right).
\end{aligned}
\]
The two Gaussian integrals needed below are
\[
\begin{aligned}
&\int_{\Rt}
\exp\left(
-\frac{e^{-\tau}+\gamma_1(1-e^{-\tau})}
{4(1-e^{-\tau})}
\left|
\eta-
\frac{e^{-\frac{\tau}{2}}}
{e^{-\tau}+\gamma_1(1-e^{-\tau})}\xi
\right|^2
\right)d\eta
\\
&\qquad=
\frac{4\pi(1-e^{-\tau})}
{e^{-\tau}+\gamma_1(1-e^{-\tau})},
\\[1ex]
&\int_{\Rt}
\left|
\eta-
\frac{e^{-\frac{\tau}{2}}}
{e^{-\tau}+\gamma_1(1-e^{-\tau})}\xi
\right|
\\
&\qquad\times
\exp\left(
-\frac{e^{-\tau}+\gamma_1(1-e^{-\tau})}
{4(1-e^{-\tau})}
\left|
\eta-
\frac{e^{-\frac{\tau}{2}}}
{e^{-\tau}+\gamma_1(1-e^{-\tau})}\xi
\right|^2
\right)d\eta
\\
&\qquad=
4\pi^{\frac{3}{2}}
\frac{(1-e^{-\tau})^{\frac{3}{2}}}
{\left(e^{-\tau}+\gamma_1(1-e^{-\tau})\right)^{\frac{3}{2}}}.
\end{aligned}
\]
The completed-square identity and these two integrals imply
\[
\begin{aligned}
\abs{\nabla e^{\tau\cL}g(\xi)}
\le{}&
\norm[L^\infty_{\gamma_1}]{g}
\exp\left(
-\frac{\gamma_1\abs{\xi}^2}
{4\left(e^{-\tau}+\gamma_1(1-e^{-\tau})\right)}
\right)
\\
&\times
\left(
\frac{\gamma_1\abs{\xi}}
{2\left(e^{-\tau}+\gamma_1(1-e^{-\tau})\right)^2}
\right.
\\
&\hspace{4.2em}\left.
+\frac{\sqrt\pi}{2}
\frac{e^{-\frac{\tau}{2}}}
{\sqrt{1-e^{-\tau}}
\left(e^{-\tau}+\gamma_1(1-e^{-\tau})\right)^{\frac{3}{2}}}
\right).
\end{aligned}
\]
Since
\[
\begin{aligned}
&\abs{\xi}\exp\left[
-\frac{1}{4}
\left(
\frac{\gamma_1}
{e^{-\tau}+\gamma_1(1-e^{-\tau})}
-\gamma
\right)\abs{\xi}^2
\right]
\\
&\qquad\le
\sqrt{\frac{2}{e(\gamma_1-\gamma)}},
\end{aligned}
\]
and \(e^{-\tau}+\gamma_1(1-e^{-\tau})\ge\gamma_1\),
\(e^{-\tau/2}\le1\), and \(0<1-e^{-\tau}\le1\),
we obtain
\begin{equation}\label{eq:radial-OU-gradient-bound}
\norm[L^\infty_\gamma]{\nabla e^{\tau\cL}g}
\lsim{\gamma,\gamma_1}
\frac{1}
{\sqrt{1-e^{-\tau}}}
\norm[L^\infty_{\gamma_1}]{g}.
\end{equation}
The time integrals satisfy
\(\int_0^t(s/t)\,ds=t/2\) and
\(\int_0^t(s/t)(1-s/t)^{-1/2}\,ds=4t/3\).
Using \eqref{eq:Frad-Duhamel}, \eqref{eq:radial-source-Gaussian-bound},
\eqref{eq:radial-OU-zero-order-bound}, and
\eqref{eq:radial-OU-gradient-bound}, we obtain
\[
\begin{aligned}
&\|F_{\mathrm{rad}}(t)\|_{L^\infty_\gamma}
+\|\nabla F_{\mathrm{rad}}(t)\|_{L^\infty_\gamma}
\\
&\qquad\lsim{\gamma} \int_0^t\frac{s}{t}
\left(1+\frac1{\sqrt{1-s/t}}\right)ds
\lsim{\gamma} t,
\end{aligned}
\]
which is \eqref{eq:Frad-Gaussian-bound}. The radial profile also has
zero mass:
\begin{equation}\label{eq:Frad-zero-mass}
\int_{\Rt}F_{\mathrm{rad}}(\xi,t)\,d\xi=0.
\end{equation}
Indeed, the vector field
\(v^{F_{\ne0}}+\Sigma(t)\xi\) is divergence free, and therefore
\(\int_{\Rt}(v^{F_{\ne0}}+\Sigma(t)\xi)
\cdot\nabla F_{\ne0}\,d\xi=0\).
The radial projection and the semigroup \(e^{\tau\cL}\) preserve the
integral, so \eqref{eq:Frad-Duhamel} proves
\eqref{eq:Frad-zero-mass}.

We next claim that
\begin{equation}\label{eq:Frad-velocity-bound}
\abs{v^{F_{\mathrm{rad}}}(\xi,t)}
\lesssim
\frac{t}{(1+\abs{\xi}^2)^{\frac32}},
\qquad
\xi\in\Rt,\quad 0\le t\le T.
\end{equation}
For a radial vorticity, the Biot--Savart law gives
\[
\begin{aligned}
v^{F_{\mathrm{rad}}}(\xi,t)
&=
\frac{\xi^\perp}{2\pi\abs{\xi}^2}
\int_{\{\abs\eta\le\abs\xi\}}F_{\mathrm{rad}}(\eta,t)\,d\eta
\\
&\stackrel{\eqref{eq:Frad-zero-mass}}{=}
-\frac{\xi^\perp}{2\pi\abs{\xi}^2}
\int_{\{\abs\eta>\abs\xi\}}F_{\mathrm{rad}}(\eta,t)\,d\eta.
\end{aligned}
\]
The first formula controls the velocity on bounded sets; for large
\(\abs\xi\), the second formula gives
\[
\abs{v^{F_{\mathrm{rad}}}(\xi,t)}
\stackrel{\eqref{eq:Frad-Gaussian-bound}}{\lsim{\gamma}}
\frac{t}{\abs\xi}
\int_{\{\abs\eta>\abs\xi\}}
e^{-\frac{\gamma\abs{\eta}^2}{4}}\,d\eta
\lsim{\gamma}
\frac{t\,e^{-\frac{\gamma\abs\xi^2}{8}}}
{1+\abs\xi}.
\]
This proves \eqref{eq:Frad-velocity-bound}.
\paragraph{Step 3. Higher space--time bounds and estimates for \(J\).}
We first establish the mixed derivative estimates needed for the
source of \(J\), and show that the radial profile belongs to the class
\(\cZ\) defined in \eqref{eq:Gaussian-class-Z}.
\begin{lemma}\label{lem:profile-mixed-regularity}
After possibly decreasing \(\nu_{0,T}\), for every \(0<\gamma<1\),
uniformly for \(0\le t\le T\) and
\(0<\nu\le\nu_{0,T}\),
\begin{equation}\label{eq:F-mixed-regularity}
\mathfrak n_{3,\gamma}(F_{\ne0})
+\mathfrak n_{2,\gamma}(\partial_tF_{\ne0})
+\mathfrak n_{0,\gamma}(\partial_t^2F_{\ne0})
\lsim{\gamma} 1,
\end{equation}
and
\begin{equation}\label{eq:Frad-mixed-regularity}
\mathfrak n_{3,\gamma}(F_{\mathrm{rad}}(\cdot,t))
\lsim{\gamma} t,
\qquad
\mathfrak n_{1,\gamma}(\partial_tF_{\mathrm{rad}}(\cdot,t))
\lsim{\gamma} 1.
\end{equation}
Moreover,
\(F_{\mathrm{rad}}(\cdot,t),\partial_tF_{\mathrm{rad}}(\cdot,t)\in\cZ\).
The map \(t\mapsto F_{\mathrm{rad}}(\cdot,t)\) is continuous with
respect to \(\mathfrak n_{3,\gamma}\) and \(C^1\) with respect to
\(\mathfrak n_{1,\gamma}\).
\end{lemma}
\begin{proof}
The estimate \eqref{eq:F-mixed-regularity} is the special case
\((m,j)=(3,0),(2,1),(0,2)\) of
\eqref{eq:F-regularized-space-time-bounds}.  It remains to establish
the assertions for the radial profile.

Set
\begin{equation}\label{eq:Frad-source-N}
\mathcal N(t)
:=
\Pi_{\mathrm{rad}}\left[
\bigl(v^{F_{\ne0}}+\Sigma(t)\xi\bigr)\cdot\nabla F_{\ne0}
\right].
\end{equation}
For any \(0<\gamma_1<\gamma_2<1\),
\eqref{eq:F-mixed-regularity}, \eqref{eq:Biot-Savart-finite-derivatives},
and \eqref{eq:Sigma-time-bound} imply
\begin{equation}\label{eq:Frad-source-space-time-bounds}
\mathfrak n_{2,\gamma_1}(\mathcal N(t))
+\mathfrak n_{0,\gamma_1}(\partial_t\mathcal N(t))
\lsim{\gamma_1} 1,
\end{equation}
where
\begin{equation}\label{eq:Frad-source-time-formula}
\begin{aligned}
\partial_t\mathcal N
=\Pi_{\mathrm{rad}}\Bigl[
&\bigl(v^{\partial_tF_{\ne0}}+\partial_t\Sigma(t)\xi\bigr)
\cdot\nabla F_{\ne0}
\\
&+\bigl(v^{F_{\ne0}}+\Sigma(t)\xi\bigr)
\cdot\nabla\partial_tF_{\ne0}
\Bigr].
\end{aligned}
\end{equation}
The loss from \(\gamma_2\) to \(\gamma_1\) absorbs \(\abs\xi\).

For \(0<u<1\), write
\[
S_u:=e^{(\log(1/u))\cL},
\]
and set \(S_1=I\) by continuity.  Changing variables \(s=tu\) gives
\begin{equation}\label{eq:Frad-rescaled-Duhamel}
F_{\mathrm{rad}}(t)
\stackrel{\eqref{eq:Frad-Duhamel}}{=}
-t\int_0^1uS_u\mathcal N(tu)\,du.
\end{equation}
Differentiating with respect to time gives
\begin{equation}\label{eq:Frad-time-Duhamel}
\partial_tF_{\mathrm{rad}}(t)
\stackrel{\eqref{eq:Frad-rescaled-Duhamel}}{=}
-\int_0^1uS_u\mathcal N(tu)\,du
-t\int_0^1u^2S_u\partial_t\mathcal N(tu)\,du.
\end{equation}
Dominated convergence also gives
\begin{equation}\label{eq:Frad-time-trace}
\partial_tF_{\mathrm{rad}}(0)
\stackrel{\eqref{eq:Frad-time-Duhamel}}{=}
-\int_0^1uS_u\mathcal N(0)\,du.
\end{equation}

The commutation identity
\(\nabla e^{\tau\cL}=e^{\tau/2}e^{\tau\cL}\nabla\)
and \eqref{eq:OU-one-derivative-Gaussian} give, for \(k=1,2,3\),
\begin{equation}\label{eq:OU-higher-derivative-bound}
\norm[L^\infty_\gamma]{\nabla^kS_ug}
\lsim{\gamma,\gamma_1}
\frac{u^{-\frac{k-1}{2}}}{\sqrt{1-u}}\,
\mathfrak n_{k-1,\gamma_1}(g),
\quad 0<u<1,\quad 0<\gamma<\gamma_1<1.
\end{equation}
The endpoint factors are integrable; in particular,
\[
\int_0^1\frac{u}{\sqrt{1-u}}\,du=\frac43,\qquad
\int_0^1\frac{u^{1/2}}{\sqrt{1-u}}\,du=\frac\pi2,\qquad
\int_0^1\frac{1}{\sqrt{1-u}}\,du=2,
\]
and
\[
\int_0^1u^2\,du=\frac13,\qquad
\int_0^1\frac{u^2}{\sqrt{1-u}}\,du=\frac{16}{15}.
\]
Combining these integrals with
\eqref{eq:Frad-source-space-time-bounds} proves
\eqref{eq:Frad-mixed-regularity}.  The corresponding integrable
majorants, together with the continuity of \(\mathcal N\) in
\(\mathfrak n_{2,\gamma_1}\) and of \(\partial_t\mathcal N\) in
\(\mathfrak n_{0,\gamma_1}\), prove continuity of
\(F_{\mathrm{rad}}\) in \(\mathfrak n_{3,\gamma}\) and \(C^1\)
continuity in \(\mathfrak n_{1,\gamma}\), including the trace at
\(t=0\).

Finally, if \(g=Gh\), then
\[
\cL(Gh)=G\left(\Delta h-\frac12\xi\cdot\nabla h\right).
\]
Thus \(e^{\tau\cL}(Gh)=G\mathcal P_\tau h\), where \(\mathcal P_\tau\) is the
Ornstein--Uhlenbeck semigroup with drift \(-\xi/2\), and
\[
\partial^\alpha\mathcal P_\tau h
=e^{-\frac{\abs\alpha\tau}{2}}\mathcal P_\tau(\partial^\alpha h).
\]
If
\(\abs{\partial^\alpha h(\xi)}\le M_\alpha(1+\abs\xi)^N\),
the Mehler formula gives a bound in terms of the Gaussian moment
\[
\mathfrak m_N:=\frac1{4\pi}\int_{\mathbb R^2}
e^{-|y|^2/4}(1+|y|)^N\,dy<\infty:
\]
\begin{equation}\label{eq:OU-polynomial-seminorm-bound}
\abs{\partial^\alpha\mathcal P_{\log(1/u)}h(\xi)}
\le M_\alpha\mathfrak m_N u^{\abs\alpha/2}
\bigl(1+\sqrt u\,\abs\xi\bigr)^N
\le M_\alpha\mathfrak m_N(1+\abs\xi)^N.
\end{equation}
The source \(\mathcal N(t)\) and its time derivative belong to the
class \(\cZ\) in \eqref{eq:Gaussian-class-Z}.
For each fixed \(\nu\), both are linear combinations of finitely many
fixed \(\cZ\)-profiles with bounded continuous time coefficients.
This follows from the finite-dimensional representation of
\(F_{\ne0}\), \eqref{eq:Frad-source-N}, and
\eqref{eq:Frad-source-time-formula}.
For each fixed \(\nu\), each Gaussian quotient in this finite family
has polynomially growing derivatives by \eqref{eq:Gaussian-class-Z}.
The estimate \eqref{eq:OU-polynomial-seminorm-bound} preserves these
polynomial bounds.
The integrable factors \(u,u^2\) in
\eqref{eq:Frad-rescaled-Duhamel} and \eqref{eq:Frad-time-Duhamel}
then permit differentiation and dominated convergence. Consequently,
\(F_{\mathrm{rad}},\partial_tF_{\mathrm{rad}}\in\cZ\), including the
trace at \(t=0\), for each fixed \(\nu\).
The constants in this argument may depend on \(\nu\); the bounds in
\eqref{eq:Frad-mixed-regularity} are uniform.
\end{proof}

We next estimate \(\mathcal Q_J\).  Since \(F_{\ne0}\) is a second
mode, \(F_{\mathrm{rad}}\) is radial, and \(C\) is a fourth mode,
\(\mathcal Q_J\) contains only the second and fourth angular modes.
The product rule applied to \eqref{eq:J-source-def}, together with
Lemma~\ref{lem:profile-mixed-regularity},
\eqref{eq:Biot-Savart-finite-derivatives}, \eqref{eq:Sigma-time-bound},
and \eqref{eq:ABC-finite-dimensional-time-bounds}, gives
\begin{equation}\label{eq:J-source-bound}
\mathfrak n_{2,\gamma}(\mathcal Q_J(\cdot,t))
\lsim{\gamma} 1.
\end{equation}
Write \(F=F_{\ne0}\), \(R=F_{\mathrm{rad}}\), and
\(V=v^F+\Sigma(t)\xi\). Direct differentiation gives
\begin{equation}\label{eq:J-source-time-formula}
\begin{aligned}
\partial_t\mathcal Q_J
\stackrel{\eqref{eq:J-source-def}}{=}{}&\partial_tC+\partial_t^2F
\\
&+\Pi_{\ne0}\Bigl[
\bigl(v^{\partial_tF}+\partial_t\Sigma(t)\xi\bigr)
\cdot\nabla(F+R)
\\
&\qquad
+V\cdot\nabla\bigl(\partial_tF+\partial_tR\bigr)
+v^{\partial_tR}\cdot\nabla F
+v^R\cdot\nabla\partial_tF
\Bigr].
\end{aligned}
\end{equation}
Only first spatial derivatives occur in this formula, so
Lemma~\ref{lem:profile-mixed-regularity},
\eqref{eq:Biot-Savart-finite-derivatives}, \eqref{eq:Sigma-time-bound},
and \eqref{eq:ABC-finite-dimensional-time-bounds} yield
\begin{equation}\label{eq:J-source-time-bound}
\norm[L^\infty_\gamma]{\partial_t\mathcal Q_J(\cdot,t)}
\lsim{\gamma} 1.
\end{equation}
The Biot--Savart velocities in the displayed formulas have smooth
derivatives of polynomial growth; multiplying such functions by a
\(\cZ\)-profile preserves the class \eqref{eq:Gaussian-class-Z}.
It follows that
\begin{equation}\label{eq:J-source-Z-membership}
\mathcal Q_J,\ \partial_t\mathcal Q_J
\in
\left(L^2_{\pw_G,2}\oplus L^2_{\pw_G,4}\right)\cap\cZ,
\end{equation}
and \(t\mapsto\mathcal Q_J(\cdot,t)\) is continuous in
\(\mathfrak n_{2,\gamma}\) and \(C^1\) in
\(\mathfrak n_{0,\gamma}\), for every \(0<\gamma<1\).
These continuity statements follow from the displayed product
formulas and Lemma~\ref{lem:profile-mixed-regularity}.

Choose \(\gamma<\widetilde\gamma<\widehat\gamma<1\).
For each angular mode of \(J\), apply
\eqref{eq:fixed-mode-higher-unregularized} with \(m=2\) and
exponents \(\widetilde\gamma,\widehat\gamma\).
Differentiation in time gives
\[
\Lambda\partial_tJ
\stackrel{\eqref{eq:J-equation}}{=}
-\partial_t\mathcal Q_J.
\]
Apply the same estimate with \(m=0\) and exponents
\(\gamma,\widetilde\gamma\). Since
\[
\norm[L^\infty_\gamma]{\xi\cdot\nabla J}
\lsim{\gamma,\widetilde\gamma} 
\mathfrak n_{1,\widetilde\gamma}(J),
\]
\eqref{eq:fixed-mode-higher-unregularized}, \eqref{eq:J-source-bound},
and \eqref{eq:J-source-time-bound} give
\begin{equation}\label{eq:J-higher-space-time-bound}
\mathfrak n_{2,\gamma}(J(\cdot,t))
+\norm[L^\infty_\gamma]{\partial_tJ(\cdot,t)}
+\norm[L^\infty_\gamma]{\cL J(\cdot,t)}
\lsim{\gamma} 1.
\end{equation}
In particular,
\begin{equation}\label{eq:J-mode-and-Gaussian-bound}
\begin{aligned}
J(\cdot,t)
&\in
\left(L^2_{\pw_G,2}\oplus L^2_{\pw_G,4}\right)\cap\cZ,
\\
\norm[L^\infty_\gamma]{J(\cdot,t)}
+\norm[L^\infty_\gamma]{\nabla J(\cdot,t)}
&\lsim{\gamma} 1,
\qquad 0\le t\le T,\quad 0<\nu\le\nu_{0,T}.
\end{aligned}
\end{equation}

The second and fourth modes have the velocity bound
\begin{equation}\label{eq:J-velocity-bound}
\abs{v^J(\xi,t)}
\stackrel{\substack{\eqref{eq:fixed-mode-velocity-decay}\\
\eqref{eq:J-mode-and-Gaussian-bound}}}{\lesssim}
\frac{1}{(1+\abs{\xi}^2)^{\frac{3}{2}}},
\qquad
\xi\in\Rt,\quad 0\le t\le T,\quad 0<\nu\le\nu_{0,T}.
\end{equation}
The second mode gives the slower decay
\((1+\abs\xi^2)^{-\frac32}\).

\begin{theorem}\label{thm:concentrated-approximation-summary}
Let \((w_{2,\mathrm{app}},v_{2,\mathrm{app}})\) be the concentrated
approximation defined in \eqref{eq:high-app}.
For every \(t\in[0,T]\), the angular modes of the profiles satisfy
\begin{equation}\label{eq:profile-mode-membership}
\begin{aligned}
F_{\ne0}(\cdot,t)&\in L^2_{\pw_G,2}\cap\cZ,
&F_{\mathrm{rad}}(\cdot,t)&\in L^2_{\pw_G,\mathrm{rad}}\cap\cZ,
\\
H(\cdot,t)&\in L^2_{\pw_G,3}\cap\cZ,
&J(\cdot,t)&\in
\left(L^2_{\pw_G,2}\oplus L^2_{\pw_G,4}\right)\cap\cZ.
\end{aligned}
\end{equation}
For every \(\frac12<\gamma<1\), uniformly for
\(0\le t\le T\) and \(0<\nu\le\nu_{0,T}\),
\begin{equation}\label{eq:profile-Gaussian-bounds}
\begin{aligned}
&\norm[L^\infty_\gamma]{F_{\ne0}}
+\norm[L^\infty_\gamma]{\nabla F_{\ne0}}
+\norm[L^\infty_\gamma]{F_{\mathrm{rad}}}
+\norm[L^\infty_\gamma]{\nabla F_{\mathrm{rad}}}
\\
&\qquad
+\norm[L^\infty_\gamma]{H}
+\norm[L^\infty_\gamma]{\nabla H}
+\norm[L^\infty_\gamma]{J}
+\norm[L^\infty_\gamma]{\nabla J}
\lsim{\gamma} 1,
\end{aligned}
\end{equation}
and
\begin{equation}\label{eq:high-profile-velocity}
\begin{aligned}
\abs{v^{F_{\ne0}}(\xi,t)}
+\abs{v^{F_{\mathrm{rad}}}(\xi,t)}
+\abs{v^J(\xi,t)}
&\lesssim
\frac{1}{(1+\abs\xi^2)^{\frac32}},
\\
\abs{v^H(\xi,t)}
&\lesssim
\frac{1}{(1+\abs\xi^2)^2}.
\end{aligned}
\end{equation}
The approximation itself obeys
\begin{equation}\label{eq:concentrated-approximation-error}
\begin{aligned}
\mathfrak n_{1,\gamma}
\bigl(w_{2,\mathrm{app}}(\cdot,t)-G\bigr)
&\lsim{\gamma} \nu t,
\\
\sup_{\xi\in\mathbb R^2}
(1+|\xi|^2)^{3/2}
\bigl|v_{2,\mathrm{app}}(\xi,t)-v^G(\xi)\bigr|
&\lesssim \nu t,
\\
\|w_{2,\mathrm{app}}(\cdot,t)-G\|_{L^1}
&\lesssim \nu t.
\end{aligned}
\end{equation}
\end{theorem}
\begin{proof}
The assertions on angular modes, Gaussian decay, and velocity follow from
\eqref{eq:F-H-mode-membership},
\eqref{eq:F-H-Gallay-bounds},
\eqref{eq:F-nonzero-velocity-bound},
\eqref{eq:H-velocity-bound},
\eqref{eq:Frad-Gaussian-bound},
\eqref{eq:Frad-velocity-bound},
\eqref{eq:Frad-mixed-regularity},
\eqref{eq:J-mode-and-Gaussian-bound}, and
\eqref{eq:J-velocity-bound}.

For \eqref{eq:concentrated-approximation-error}, set \(s=\nu t\).
The standing choice of \(\nu_{0,T}\) gives \(0\le s\le1\). Hence
\[
\begin{aligned}
\mathfrak n_{1,\gamma}(w_{2,\mathrm{app}}-G)
&\stackrel{\substack{\eqref{eq:high-app}\\
\eqref{eq:profile-Gaussian-bounds}}}{\lsim{\gamma}} (s+s^{3/2}+s^2)
\lsim{\gamma} s,
\\
\sup_\xi(1+|\xi|^2)^{3/2}|v_{2,\mathrm{app}}-v^G|
&\stackrel{\substack{\eqref{eq:high-app}\\
\eqref{eq:high-profile-velocity}}}{\lesssim} (s+s^{3/2}+s^2)
\lesssim s.
\end{aligned}
\]
Finally, the integrability of \(e^{-3|\xi|^2/16}\) gives
\[
\|w_{2,\mathrm{app}}-G\|_{L^1}
\stackrel{\eqref{eq:finite-Gaussian-seminorm}}{\le}
\mathfrak n_{1,3/4}(w_{2,\mathrm{app}}-G)
\int_{\mathbb R^2}e^{-3|\xi|^2/16}\,d\xi
\lesssim s.
\]
\end{proof}

\subsection{The residual estimate}
\begin{proposition}\label{prop:high-residual}
Define
\(\Phi_{\mathrm{app}}(\xi,t)
:=\Phi(w_{2,\mathrm{app}},\vapp)(\xi,t)\).
For every \(\gamma\in(\frac12,1)\), uniformly for
\(0<\nu\le\nu_{0,T}\),
\begin{equation}\label{eq:residual-bound}
\abs{\Phi_{\mathrm{app}}(\xi,t)}
\lsim{\gamma} (\nu t)^{\frac32}
e^{-\frac{\gamma\abs\xi^2}{4}},
\qquad
\xi\in\Rt,\quad 0<t\le T.
\end{equation}
\end{proposition}
\begin{proof}
Fix \(1/2<\gamma<\gamma_1<\gamma_2<1\).
The profile \(H\) satisfies
\[
\mathfrak n_{2,\gamma_2}(H)
+\norm[L^\infty_{\gamma_2}]{\partial_tH}
\stackrel{\eqref{eq:H-unregularized-space-time-bounds}}{\lsim{\gamma_2}} 1.
\]
Since
\(\norm[L^\infty_{\gamma_1}]{\xi\cdot\nabla H}
\lsim{\gamma_1,\gamma_2} \mathfrak n_{1,\gamma_2}(H)\),
we have
\(\norm[L^\infty_{\gamma_1}]{\cL H}\lsim{\gamma_1} 1\).
Hence the evolution terms satisfy
\begin{equation}\label{eq:residual-profile-regularity}
\begin{aligned}
&\norm[L^\infty_{\gamma_1}]{t\partial_tH}
+\norm[L^\infty_{\gamma_1}]{\cL H}
+\norm[L^\infty_{\gamma_1}]{t\partial_tJ}
+\norm[L^\infty_{\gamma_1}]{J}
+\norm[L^\infty_{\gamma_1}]{\cL J}
\stackrel{\substack{\eqref{eq:H-unregularized-space-time-bounds}\\
\eqref{eq:J-higher-space-time-bound}}}{\lsim{\gamma_1}} 1.
\end{aligned}
\end{equation}

Set \(s:=\nu t\le1\), and denote the last line of
\eqref{eq:profile-residual-reduction} by \(\mathcal D_{\mathrm{drift}}\).
By \eqref{eq:profile-Gaussian-bounds}, \eqref{eq:high-profile-velocity},
and \eqref{eq:residual-profile-regularity},
\[
|\Phi_{\mathrm{app}}-\mathcal R_0-\mathcal D_{\mathrm{drift}}|
\lsim{\gamma} s^{3/2}e^{-\gamma|\xi|^2/4}.
\]
Here \(s^2\le s^{\frac32}\), \(t\le T\), and the loss from
\(\gamma_1\) to \(\gamma\) absorbs all polynomial factors in \(\xi\).
The same estimate for \(\mathcal R_0\) is
\eqref{eq:R0-target}.

It remains to estimate the drift term in the last line of
\eqref{eq:profile-residual-reduction}.  If
\(\sqrt{s}\abs\xi\le1\), then
\[
\begin{aligned}
&\left|
\left(
\sqrt{\frac t\nu}\,b-t\Sigma(t)\xi
\right)\cdot
\left[
s\nabla F_{\ne0}+s\nabla F_{\mathrm{rad}}
+s^{\frac32}\nabla H+s^2\nabla J
\right]
\right|
\\
&\qquad\stackrel{\substack{\eqref{eq:b-vector-Taylor}\\
\eqref{eq:profile-Gaussian-bounds}}}{\lsim{\gamma_1}} t\sqrt{s}\,\abs\xi^2\,
s(1+\abs\xi)^N
e^{-\frac{\gamma_1\abs\xi^2}{4}}
\lsim{\gamma} s^{\frac32}
e^{-\frac{\gamma\abs\xi^2}{4}},
\end{aligned}
\]
where \(N\) is a fixed integer and \(t\le T\).
If \(\sqrt{s}\abs\xi>1\), then
\[
\left|\sqrt{t/\nu}\,b-t\Sigma(t)\xi\right|
\stackrel{\substack{\eqref{eq:b-global}\\\eqref{eq:Sigma-time-bound}}}{\lesssim} \left(\frac{t}{\sqrt{s}}+t\abs\xi\right).
\]
Therefore,
\[
\begin{aligned}
&\left|
\left(
\sqrt{\frac t\nu}\,b-t\Sigma(t)\xi
\right)\cdot
\left[
s\nabla F_{\ne0}+s\nabla F_{\mathrm{rad}}
+s^{\frac32}\nabla H+s^2\nabla J
\right]
\right|
\\
&\qquad
\stackrel{\eqref{eq:profile-Gaussian-bounds}}{\lsim{\gamma_1}} 
\left(t\sqrt{s}+ts\abs\xi\right)
(1+\abs\xi)^N
e^{-\frac{\gamma_1\abs\xi^2}{4}}
\lsim{\gamma} s^{\frac32}e^{-\frac{\gamma\abs\xi^2}{4}}.
\end{aligned}
\]
The last inequality follows from the elementary bound
\begin{equation}\label{eq:outer-Gaussian-absorption}
\sup_{\substack{0<s\le1\\r\ge s^{-1/2}}}
s^{-M}(1+r)^N
e^{-\delta r^2/4}
<\infty,
\qquad M,N\ge0,\quad\delta>0,
\end{equation}
which follows from \(s^{-M}\le r^{2M}\).  Apply it with
\(\delta=\gamma_1-\gamma\).
The estimates on the two regions give
\(|\mathcal D_{\mathrm{drift}}|
\lsim{\gamma} s^{3/2}e^{-\gamma|\xi|^2/4}\); consequently,
\[
|\Phi_{\mathrm{app}}|
\le |\Phi_{\mathrm{app}}-\mathcal R_0-\mathcal D_{\mathrm{drift}}|
+|\mathcal D_{\mathrm{drift}}|+|\mathcal R_0|
\stackrel{\eqref{eq:R0-target}}{\lsim{\gamma}} s^{3/2}e^{-\gamma|\xi|^2/4}.
\]
This is \eqref{eq:residual-bound}.
\end{proof}

\section{Equations for the remainder variables}
\label{sec:remainder-equations}

We compare the exact vorticity components solving \eqref{eq:comp-eq},
which are smooth for \(t>0\), with the regular approximation
\eqref{eq:regular-approx-def} and the concentrated approximation
\eqref{eq:high-app}. The regular component
and its approximation have mass \(\int\w_0^E\), while the
concentrated component and its approximation have mass one. The estimates obtained for positive times will be extended to \(t=0\)
in Section~\ref{sec:closure}.
We write the correction to the Gaussian profile as
\begin{equation}\label{eq:remainder-core-composite}
\begin{aligned}
w_{2,a}
&:=F_{\ne0}+F_{\mathrm{rad}}
+(\nu t)^{\frac{1}{2}}H+(\nu t)J,
\\
v_{2,a}
&:=v^{F_{\ne0}}+v^{F_{\mathrm{rad}}}
+(\nu t)^{\frac{1}{2}}v^H+(\nu t)v^J
=K\star_\xi w_{2,a}.
\end{aligned}
\end{equation}
Thus \(w_{2,\mathrm{app}}=G+\nu t\,w_{2,a}\) and
\(v_{2,\mathrm{app}}=v^G+\nu t\,v_{2,a}\).
For \(t>0\), define the wave and concentrated remainders by
\begin{equation}\label{eq:remainder-ansatz}
\w^{E,\nu}=\wapp+\nu^{\frac{3}{2}}\wb_1,
\qquad
w_2=w_{2,\mathrm{app}}+\nu t\,\wb_2,
\end{equation}
and set
\begin{equation}\label{eq:remainder-velocity-ansatz}
\vb_1:=K\star_x\wb_1,
\qquad
\vb_2:=K\star_\xi\wb_2.
\end{equation}
It follows that
\(v^{E,\nu}=\vapp+\nu^{3/2}\vb_1\) and
\(v_2=v^G+\nu t(v_{2,a}+\vb_2)\).

\paragraph{The concentrated remainder in self-similar variables.}
Throughout this paragraph, \(x=\Zn(t)+\sqrt{\nu t}\,\xi\).
The velocity \(\vb_1(x,t)\) is evaluated in the original variables
corresponding to \(\xi\).
Since \(\partial_t\Zn=\vapp(\Zn,t)\),
\[
v^{E,\nu}(x,t)-\partial_t\Zn
\stackrel{\substack{\eqref{eq:strain-field}\\
\eqref{eq:remainder-ansatz}}}{=}
b(\xi,t)+\nu^{3/2}\vb_1(x,t).
\]
Subtracting the approximate residual and dividing by \(\nu t\) gives
\begin{equation}\label{eq:concentrated-remainder-equation}
\begin{aligned}
0\stackrel{\substack{\eqref{eq:exact-inner-equation}\\
\eqref{eq:remainder-ansatz}}}{=}&
(t\partial_t-\cL+1)\wb_2
+\frac{1}{\nu}\Lambda\wb_2
+\sqrt{\frac{t}{\nu}}\,b\cdot\nabla\wb_2
\\
&+t\Bigl(
\vb_2\cdot\nabla w_{2,a}
+v_{2,a}\cdot\nabla\wb_2
+\vb_2\cdot\nabla\wb_2
\Bigr)
\\
&+\frac{1}{\sqrt t}\,
\vb_1(x,t)\cdot\nabla G
+\nu\sqrt t\,
\vb_1(x,t)\cdot\nabla\bigl(w_{2,a}+\wb_2\bigr)
\\
&+\frac{1}{\nu t}\Phi(w_{2,\mathrm{app}},\vapp),
\qquad \xi\in\Rt,\quad t>0.
\end{aligned}
\end{equation}
Here the residual in the last line is given explicitly by
\eqref{eq:profile-residual-reduction}.  The coefficients in
\eqref{eq:concentrated-remainder-equation} follow from
\((t\partial_t-\cL)(\nu t\,\wb_2)
=\nu t(t\partial_t-\cL+1)\wb_2\),
\(\Lambda\wb_2=v^G\cdot\nabla\wb_2+\vb_2\cdot\nabla G\),
\(\sqrt{t/\nu}\,\nu^{3/2}/(\nu t)=t^{-1/2}\), and
\(\sqrt{t/\nu}\,\nu^{3/2}=\nu\sqrt t\).

\paragraph{The wave remainder in the original variables.}
The Biot--Savart kernel is homogeneous of degree \(-1\).  For
\(\xi=(x-\Zn(t))/\sqrt{\nu t}\),
\eqref{eq:remainder-ansatz}, \eqref{eq:remainder-core-composite}, and
\eqref{eq:remainder-velocity-ansatz} give
\begin{equation}\label{eq:physical-total-velocity}
\begin{aligned}
v^{B,\nu}(x,t)
&=V_G^{\Zn,\nu}(x,t)
+\sqrt{\nu t}\,
\bigl(v_{2,a}+\vb_2\bigr)(\xi,t),
\\
u^\nu(x,t)
&\stackrel{\substack{\eqref{eq:remainder-ansatz}\\
\eqref{eq:regular-approx-def}}}{=}\vapp(x,t)+V_G^{\Zn,\nu}(x,t)
+\nu^{\frac{3}{2}}\vb_1(x,t)
+\sqrt{\nu t}\,
\bigl(v_{2,a}+\vb_2\bigr)(\xi,t).
\end{aligned}
\end{equation}
Using \eqref{eq:comp-eq}, \eqref{eq:regular-approx-residual},
\eqref{eq:remainder-ansatz}, and \eqref{eq:physical-total-velocity},
we obtain the wave remainder equation
\begin{equation}\label{eq:wave-remainder-equation}
\begin{aligned}
\partial_t\wb_1
+u^\nu\cdot\nabla\wb_1
-\nu\Delta\wb_1
={}&-\vb_1\cdot\nabla\wapp
\\
&-
\frac{\sqrt t}{\nu}
\bigl(v_{2,a}+\vb_2\bigr)
\left(
\frac{x-\Zn(t)}{\sqrt{\nu t}},t
\right)\cdot\nabla\wapp
\\
&-
\sqrt\nu\left(
v_{1,a}\cdot\nabla w_{1,a}-\Delta w_{1,a}
\right),
\qquad x\in\Rt,\quad t>0.
\end{aligned}
\end{equation}
The transport velocity \(u^\nu\) includes both the exact wave velocity
and the exact concentrated velocity.

\paragraph{Mass constraints.}
Conservation of the component masses, the initial value
\(w_{1,a}(0)=0\), and the zero masses of the core corrections give
\begin{equation}\label{eq:remainder-mass-constraints}
\int_{\Rt}\wb_1(x,t)\,dx
\stackrel{\substack{\eqref{eq:remainder-ansatz}\\
\eqref{eq:w1a-zero-mass}}}{=}0,
\qquad
\int_{\Rt}\wb_2(\xi,t)\,d\xi
\stackrel{\eqref{eq:remainder-ansatz}}{=}0,
\qquad t>0.
\end{equation}

\section{Estimate for the linearized concentrated equation}
\label{sec:linearized-estimate}
We adapt the strain-cancelling energy estimates of
\cite{Ga11,DonatiGallay2026} to the forced linear equation used for
the concentrated remainder.
The linear equation is
\begin{equation}\label{eq:linearized}
(t\partial_t-\cL+1)w
+\frac{1}{\nu}\Lambda w
+\sqrt{\frac{t}{\nu}}\,b(\xi,t)\cdot\nabla w
=f,
\end{equation}
with
\(v=K\star_\xi w\) and
\(\Lambda w=v^G\cdot\nabla w+v\cdot\nabla G\),
and the zero-mass condition
\begin{equation}\label{eq:zero-mass}
\int_{\Rt}w(\xi,t)\,d\xi=0.
\end{equation}
The drift in \eqref{eq:linearized} is the field defined in
\eqref{eq:strain-field}.  We also recall that
\(\Sigma(t)=\operatorname{sym}\nabla\vapp(\Zn(t),t)\) and
\(\operatorname{tr}\Sigma(t)=0\).
In this section, we assume that \(b\) is divergence-free and satisfies
\begin{equation}\label{eq:linear-b-bounds}
\begin{aligned}
\nabla_\xi\cdot b(\xi,t)&=0,
\qquad
\abs{b(\xi,t)}\lsim{T} 1,
&&\xi\in\Rt,
\\
\abs{b(\xi,t)}
&\lsim{T} \sqrt{\nu t}\abs\xi,
&&\sqrt{\nu t}\abs\xi\le1,
\\
\abs{b(\xi,t)\cdot\xi
-\sqrt{\nu t}\,\xi\cdot\Sigma(t)\xi}
&\lsim{T} \nu t\abs\xi^3,
&&\sqrt{\nu t}\abs\xi\le1.
\end{aligned}
\end{equation}

\begin{theorem}
\label{thm:linear-weighted-estimate}
Fix \(T>0\), and assume \eqref{eq:linear-b-bounds} and
\eqref{eq:Sigma-time-bound}, with
\(\Sigma(t)^T=\Sigma(t)\) and \(\operatorname{tr}\Sigma(t)=0\).
There exist \(a_0>0\), \(B_0>1\),
\(\nu_T>0\), and \(\kappa_T>0\) such that the following
holds.  If \(0<\nu\le\nu_T\) and \(w\) is a smooth real-valued solution of
\eqref{eq:linearized} satisfying \eqref{eq:zero-mass}, with finite
weighted energy and locally integrable dissipation, then,
for almost every \(0<t\le T\),
\begin{equation}\label{eq:linear-weighted-estimate}
\begin{aligned}
t\frac{d}{dt}E_{\pw}(t)
+\kappa_T\mathcal Q(t)
-\abs{\inner{f(t)}{w(t)}_{\pw(t)}}
&\lsim{T}tE_{\pw}(t).
\end{aligned}
\end{equation}
\end{theorem}

Here \(\pw\) is defined by \eqref{eq:p-def}, \(E_{\pw}\) by
\eqref{eq:energy}, and \(\mathcal Q\) by
\eqref{eq:IJK-def}--\eqref{eq:Q-def}; the pairing is defined in
\eqref{eq:p-weighted-pairing}.  No smallness condition is imposed on
\(\sup_{0<t\le T}t\abs{\Sigma(t)}\).
Throughout the construction, we take
\(0<a_0\ll1\), \(B_0\gg1\), and
\(\beta=a_0^2/B_0\), in that order.

In Section~\ref{sec:w2-estimate}, we apply this estimate to
\(w=\wb_2\) and bound each source term.

\subsection{The radial weight}
The radial part of the weight contains a Gaussian
core, a constant plateau, and an outer exponential tail.  The angular
part is supported in the core and solves the cell equation that cancels
the quadratic strain.  Fix \(0<a_0\ll1\) and \(B_0\gg1\), chosen in
that order, and set
\begin{equation}\label{eq:weight-scales}
a_\nu(t)
:=
a_0(\nu t)^{-\frac{1}{4}},
\qquad
B_\nu(t)
:=
B_0(\nu t)^{-\frac{1}{2}},
\qquad
\beta:=\frac{a_0^2}{B_0}.
\end{equation}
Then
\begin{equation}\label{eq:weight-scale-identity}
\beta B_\nu(t)
\stackrel{\eqref{eq:weight-scales}}{=}a_\nu(t)^2.
\end{equation}
We impose
\begin{equation}\label{eq:standing-smallness}
\nu T\le\frac{a_0^4}{16},
\qquad
\beta\le1.
\end{equation}
Define \(\vartheta:(-\infty,1]\to[0,1]\) by
\begin{equation}\label{eq:weight-cutoff}
\vartheta(y)
:=
\begin{cases}
0,
&y\le-1,
\\[1mm]
\zeta_2(y+1)^2-\zeta_3(y+1)^3,
&-1\le y\le1,
\end{cases}
\end{equation}
where
\begin{equation}\label{eq:weight-cutoff-coefficients}
\zeta_2
:=
\frac34-\frac78e^{-\frac{1}{4}},
\qquad
\zeta_3
:=
\frac14-\frac5{16}e^{-\frac{1}{4}}.
\end{equation}
The radial background weight is
\begin{equation}\label{eq:p0-def}
\pw_0(t,\xi)
:=
\begin{cases}
\displaystyle
e^{\frac{\abs{\xi}^2}{4}}
\left(
1-\vartheta\bigl(\abs{\xi}^2-a_\nu(t)^2\bigr)
\right),
&
0\le\abs{\xi}^2\le a_\nu(t)^2+1,
\\[2mm]
\displaystyle
e^{\frac{a_\nu(t)^2}{4}},
&
a_\nu(t)^2+1\le\abs{\xi}^2\le B_\nu(t)^2,
\\[2mm]
\displaystyle
e^{\frac{\beta\abs{\xi}}{4}},
&
\abs{\xi}\ge B_\nu(t).
\end{cases}
\end{equation}
The weight \(\pw_0\) is \(C^1\) across the transition annulus
and continuous at \(|\xi|=B_\nu(t)\). These matching conditions imply the continuity of
\(q\) in \eqref{eq:q-def}. The weighted energy identity uses only
first derivatives of \(\pw_0\).
\begin{lemma}\label{lem:weight-matching}
The function \(\vartheta\) defined in
\eqref{eq:weight-cutoff}--\eqref{eq:weight-cutoff-coefficients}
satisfies
\begin{equation}\label{eq:weight-cutoff-matching}
\begin{gathered}
\vartheta(-1)=\vartheta'(-1)=0,
\\
\vartheta(1)=1-e^{-\frac14},
\qquad
\vartheta'(1)=\frac14e^{-\frac14}.
\end{gathered}
\end{equation}
Consequently, for every \(t>0\), the radial weight
\(\pw_0(t,\cdot)\) defined in \eqref{eq:p0-def} is continuous on
\(\mathbb{R}^2\) and belongs to
\(W^{1,\infty}_{\mathrm{loc}}(\mathbb{R}^2)\).
In particular, the weighted integration by parts identities used below
are valid for \(\pw_0\).
\end{lemma}

\begin{proof}
Fix \(t>0\), and abbreviate \(a=a_\nu(t)\), \(B=B_\nu(t)\),
\(r=|\xi|\), and \(y=r^2-a^2\).  On \([-1,1]\),
\(\vartheta(y)=\zeta_2(y+1)^2-\zeta_3(y+1)^3\) and
\(\vartheta'(y)=2\zeta_2(y+1)-3\zeta_3(y+1)^2\).  Hence
\(\vartheta(-1)=\vartheta'(-1)=0\).  At \(y=1\),
\(\vartheta(1)=4\zeta_2-8\zeta_3\) and
\(\vartheta'(1)=4\zeta_2-12\zeta_3\).  Moreover,
\[
\begin{aligned}
4\zeta_2-8\zeta_3
&\stackrel{\eqref{eq:weight-cutoff-coefficients}}{=}
\left(3-\frac72e^{-\frac14}\right)
-
\left(2-\frac52e^{-\frac14}\right)=
1-e^{-\frac14},
\\
4\zeta_2-12\zeta_3
&\stackrel{\eqref{eq:weight-cutoff-coefficients}}{=}
\left(3-\frac72e^{-\frac14}\right)
-
\left(3-\frac{15}{4}e^{-\frac14}\right)=
\frac14e^{-\frac14}.
\end{aligned}
\]
This proves \eqref{eq:weight-cutoff-matching}.

We next check the matching conditions.  In the Gaussian and transition
regions, \(\pw_0(t,r)=e^{r^2/4}(1-\vartheta(y))\), and therefore
\begin{equation}\label{eq:p0-radial-derivative}
\partial_r\pw_0(t,r)
=
r e^{\frac{r^2}{4}}
\left(
\frac12(1-\vartheta(y))-2\vartheta'(y)
\right).
\end{equation}

Since \(a\stackrel{\eqref{eq:standing-smallness}}{\ge}2\), the interfaces are
well defined.  At
\(\Gamma_-(t):=\{|\xi|^2=a^2-1\}\), the identities
\(\vartheta(-1)=\vartheta'(-1)=0\) show that both one-sided traces are
\(\pw_0=e^{r^2/4}\) and
\(\partial_r\pw_0=(r/2)e^{r^2/4}\).

At \(\Gamma_+(t):=\{|\xi|^2=a^2+1\}\), the transition trace is
\(e^{(a^2+1)/4}e^{-1/4}=e^{a^2/4}\), which is the plateau value.
Moreover,
\[
\partial_r\pw_0
\stackrel{\substack{\eqref{eq:p0-radial-derivative}\\
\eqref{eq:weight-cutoff-matching}}}{=}
r e^{r^2/4}\left(\frac12e^{-1/4}
-2\cdot\frac14e^{-1/4}\right)=0.
\]
Thus both the value and the radial derivative match at the two
interfaces surrounding the transition annulus.

At the outer interface
\(\Gamma_B(t):=\{|\xi|=B\}\), the plateau and exponential traces are
\(e^{a^2/4}\) and \(e^{\beta B/4}\), respectively.  They coincide
because \(\beta B\stackrel{\eqref{eq:weight-scale-identity}}{=}a^2\).
The radial derivatives need not match:
\(\partial_r^-\pw_0(t,B)=0\), whereas
\(\partial_r^+\pw_0(t,B)=(\beta/4)e^{\beta B/4}\).  Since \(\pw_0\) is continuous, its first distributional derivatives
contain no interface measure. The energy identity uses only these
derivatives; the mixed diffusion term is not integrated by parts again.
Each piece has locally bounded first derivatives, so
\(\pw_0(t,\cdot)\in W^{1,\infty}_{\mathrm{loc}}(\mathbb R^2)\).
\end{proof}

\begin{lemma}\label{lem:weight-cutoff-properties}
For \(-1<y\le1\), the transition cutoff satisfies
\begin{equation}\label{eq:weight-cutoff-properties}
0<\vartheta'(y)
\le
\frac14\bigl(1-\vartheta(y)\bigr).
\end{equation}
\end{lemma}
\begin{proof}
Direct differentiation gives
\[
\vartheta'(y)
=(y+1)\bigl(2\zeta_2-3\zeta_3(y+1)\bigr),
\qquad
\vartheta''(y)
=2\zeta_2-6\zeta_3(y+1).
\]
By \eqref{eq:weight-cutoff-coefficients}, both derivatives are
positive on \((-1,1]\).  Hence
\(h:=1-\vartheta-4\vartheta'\) is strictly decreasing there.
Moreover,
\(h(1)\stackrel{\eqref{eq:weight-cutoff-matching}}{=}0\), and therefore
\(h(y)\ge0\) for \(-1<y\le1\), which is exactly the upper bound in
\eqref{eq:weight-cutoff-properties}.
\end{proof}
Set
\begin{equation}\label{eq:vartheta-tilde-def}
\widetilde\vartheta
:=
\vartheta+4\vartheta'.
\end{equation}
\begin{lemma}\label{lem:p0-inner-comparison}
Assume \eqref{eq:standing-smallness}.  For every \(t>0\) and every
\(\xi\in\Rt\) with \(\abs{\xi}^2\le a_\nu(t)^2+1\),
\begin{equation}\label{eq:p0-vs-gaussian}
e^{-\frac{1}{4}}e^{\frac{\abs{\xi}^2}{4}}
\le
\pw_0(t,\xi)
\le
e^{\frac{\abs{\xi}^2}{4}}.
\end{equation}
\end{lemma}

\begin{proof}
Set \(y:=\abs{\xi}^2-a_\nu(t)^2\). In the region under consideration,
\(y\le1\), and
\[
\pw_0(t,\xi)
\stackrel{\eqref{eq:p0-def}}{=}
e^{\abs{\xi}^2/4}\bigl(1-\vartheta(y)\bigr),
\]
and
\eqref{eq:p0-vs-gaussian} is equivalent to
\[
e^{-\frac14}\le1-\vartheta(y)\le1,
\qquad y\le1.
\]

If \(y\le-1\), then
\(\vartheta(y)\stackrel{\eqref{eq:weight-cutoff}}{=}0\), and both
inequalities are immediate.  If \(-1\le y\le1\), then
\(\vartheta'\stackrel{\eqref{eq:weight-cutoff-properties}}{>}0\)
on \((-1,1]\), so \(\vartheta\) is increasing there.  Hence
\[
0
\stackrel{\eqref{eq:weight-cutoff-matching}}{=}
\vartheta(-1)
\le\vartheta(y)\le
\vartheta(1)
\stackrel{\eqref{eq:weight-cutoff-matching}}{=}
1-e^{-\frac14},
\]
which is the required two-sided bound.
\end{proof}
If \(y=\abs{\xi}^2-a_\nu(t)^2\), direct differentiation on each
smooth region gives
\begin{equation}\label{eq:p0-gradient}
\nabla\pw_0(t,\xi)
\stackrel{\substack{\eqref{eq:p0-def}\\
\eqref{eq:vartheta-tilde-def}}}{=}
\begin{cases}
\displaystyle
\frac{\xi}{2}e^{\frac{\abs{\xi}^2}{4}}
\left(1-\widetilde\vartheta(y)\right),
&
\abs{\xi}^2\le a_\nu(t)^2+1,
\\[2mm]
0,
&
a_\nu(t)^2+1<\abs{\xi}^2<B_\nu(t)^2,
\\[2mm]
\displaystyle
\frac{\beta}{4}
\frac{\xi}{\abs{\xi}}
e^{\frac{\beta\abs{\xi}}{4}},
&
\abs{\xi}>B_\nu(t),
\end{cases}
\end{equation}
and
\begin{equation}\label{eq:p0-time}
t\partial_t\pw_0(t,\xi)
\stackrel{\substack{\eqref{eq:p0-def}\\
\eqref{eq:weight-scales}}}{=}
\begin{cases}
\displaystyle
-\frac{a_\nu(t)^2}{2}
e^{\frac{\abs{\xi}^2}{4}}\vartheta'(y),
&
\abs{\xi}^2\le a_\nu(t)^2+1,
\\[2mm]
\displaystyle
-\frac{a_\nu(t)^2}{8}\pw_0(t,\xi),
&
a_\nu(t)^2+1<\abs{\xi}^2<B_\nu(t)^2,
\\[2mm]
0,
&
\abs{\xi}>B_\nu(t).
\end{cases}
\end{equation}

\begin{lemma}\label{lem:p0-pointwise-coercivity}
Let \(\pw_0\) be the radial weight defined in
\eqref{eq:p0-def}, and let
\begin{equation}\label{eq:chi-def}
\chi_\nu(t,\xi)
:=
\begin{cases}
\abs{\xi}^2,
&0\le\abs{\xi}\le a_\nu(t),
\\
a_\nu(t)^2,
&a_\nu(t)\le\abs{\xi}\le B_\nu(t),
\\
\beta\abs{\xi},
&\abs{\xi}\ge B_\nu(t).
\end{cases}
\end{equation}
Under the smallness assumptions \eqref{eq:standing-smallness}, for
every \(t\in(0,T]\) and almost every \(\xi\in\Rt\), one has
the pointwise bound
\begin{equation}\label{eq:p0-pointwise-coercivity}
\frac{\xi\cdot\nabla\pw_0}{4\pw_0}
-\frac{\abs{\nabla\pw_0}^2}{3\pw_0^2}
-\frac{t\partial_t\pw_0}{2\pw_0}
\ge
\frac1{48}\chi_\nu(t,\xi).
\end{equation}
\end{lemma}
\begin{proof}
Denote by \(\mathfrak c_0(t,\xi)\) the left-hand side of
\eqref{eq:p0-pointwise-coercivity}, and write \(r=\abs\xi\) and
\(y=r^2-a_\nu(t)^2\).  Let 
\[
\widetilde\vartheta=\vartheta+4\vartheta',
\qquad
\mathcal A(y)=\frac{1-\widetilde\vartheta(y)}{1-\vartheta(y)},
\qquad
\mathcal B(y)=\frac{\vartheta'(y)}{1-\vartheta(y)},
\]
so that \(0\le\mathcal A\le1\), \(\mathcal B\ge0\), and
\(\frac1{48}\mathcal A+\frac14\mathcal B\ge\frac1{48}\).
The standing smallness assumption \(\nu T\le a_0^4/16\) gives
\(a_\nu(t)\ge2\) for \(0<t\le T\), and \(\beta\le1\).
We evaluate \(\mathfrak c_0\) in the four radial regions.

\smallskip
\noindent\emph{1. Gaussian region \(r^2\le a_\nu(t)^2-1\).}
Here \(y\le-1\), so \(\vartheta(y)=\vartheta'(y)=\widetilde\vartheta(y)=0\)
and \(\pw_0=e^{r^2/4}\).  Thus
\[
\frac{\nabla\pw_0}{\pw_0}=\frac\xi2,
\qquad
\frac{t\partial_t\pw_0}{\pw_0}=0,
\qquad
\mathfrak c_0=\frac{r^2}{8}-\frac{r^2}{12}
=\frac{r^2}{24}=\frac{\chi_\nu}{24}.
\]

\smallskip
\noindent\emph{2. Transition region \(a_\nu(t)^2-1\le r^2\le a_\nu(t)^2+1\).}
Here \(\pw_0=e^{r^2/4}(1-\vartheta(y))\), and therefore
\[
\frac{\nabla\pw_0}{\pw_0}=\frac\xi2\mathcal A(y),
\qquad
\frac{t\partial_t\pw_0}{\pw_0}=-\frac{a_\nu(t)^2}{2}\mathcal B(y),
\]
so that
\[
\mathfrak c_0
=r^2\left(\frac18\mathcal A(y)-\frac1{12}\mathcal A(y)^2\right)
+\frac{a_\nu(t)^2}{4}\mathcal B(y).
\]
Since \(0\le\mathcal A\le1\), one has
\(\frac18\mathcal A-\frac1{12}\mathcal A^2\ge\frac1{24}\mathcal A\).
Moreover \(a_\nu(t)\ge2\) gives \(r^2\ge a_\nu(t)^2-1\ge a_\nu(t)^2/2\),
while \(\chi_\nu\le a_\nu(t)^2\) throughout this region.  Hence
\[
\mathfrak c_0
\ge
a_\nu(t)^2
\left(\frac1{48}\mathcal A(y)+\frac14\mathcal B(y)\right)
\ge
\frac{a_\nu(t)^2}{48}
\ge
\frac{\chi_\nu}{48}.
\]

\smallskip
\noindent\emph{3. Plateau region \(a_\nu(t)^2+1\le r^2\le B_\nu(t)^2\).}
Here \(\pw_0=e^{a_\nu(t)^2/4}\), so \(\nabla\pw_0=0\); and since
\(t\partial_ta_\nu(t)^2=-\frac12a_\nu(t)^2\), one has
\(t\partial_t\pw_0=-\frac{a_\nu(t)^2}{8}\pw_0\).  Therefore
\[
\mathfrak c_0=\frac{a_\nu(t)^2}{16}=\frac{\chi_\nu}{16}.
\]

\smallskip
\noindent\emph{4. Outer exponential region \(r\ge B_\nu(t)\).}
Here \(\pw_0=e^{\beta r/4}\) is independent of \(t\), and
\(\nabla\pw_0=\frac\beta4\frac\xi r\pw_0\).  Hence
\[
\mathfrak c_0=\frac{\beta r}{16}-\frac{\beta^2}{48}.
\]
Since \(\beta B_\nu(t)=a_\nu(t)^2\), we have
\(\beta r\ge a_\nu(t)^2\ge4\), and \(\beta\le1\) gives
\(\beta^2\le\beta\le\frac14\beta r\).  Consequently
\[
\mathfrak c_0
\ge\frac{\beta r}{16}-\frac{\beta r}{192}
=\frac{11}{192}\beta r
\ge\frac{\chi_\nu}{48}.
\]

\smallskip
Combining the four cases proves the claim.
\end{proof}
\subsection{The strain-adapted angular correction}
We use an additive angular correction to implement the strain
cancellation of \cite{DonatiGallay2026}. Define
\begin{equation}\label{eq:q-def}
q(t,\xi)
:=
\begin{cases}
0,
&
\xi=0,
\\[2mm]
\displaystyle
\pi
\frac{\xi\cdot\nabla\pw_0(t,\xi)}
     {1-e^{-\frac{\abs{\xi}^2}{4}}}
\bigl(\xi\cdot\Sigma(t)\xi^\perp\bigr),
&
0<\abs{\xi}^2\le a_\nu(t)^2+1,
\\[3mm]
0,
&
\abs{\xi}^2>a_\nu(t)^2+1.
\end{cases}
\end{equation}
We extend \(q\) to \(t=0\) by setting
\begin{equation}\label{eq:q-initial-extension}
q(0,\xi)
:=
\begin{cases}
0,
&
\xi=0,
\\[2mm]
\displaystyle
\pi
\frac{
\frac{\abs{\xi}^2}{2}
e^{\frac{\abs{\xi}^2}{4}}
}{
1-e^{-\frac{\abs{\xi}^2}{4}}
}
\bigl(\xi\cdot\Sigma(0)\xi^\perp\bigr),
&
\xi\neq0.
\end{cases}
\end{equation}
For \(s\ge0\), set
\begin{equation}\label{eq:h-def}
\mathfrak h(s)
:=
\begin{cases}
\displaystyle
\frac{s e^{s/4}}{2(1-e^{-s/4})},
&s>0,
\\[2mm]
2,
&s=0.
\end{cases}
\end{equation}
\begin{lemma}\label{lem:q-zero-extension}
Assume \eqref{eq:standing-smallness}.  The function \(q\) defined by
\eqref{eq:q-def} and \eqref{eq:q-initial-extension} is continuous and
locally Lipschitz on \([0,T]\times\Rt\).  It is \(C^1\) on each smooth
piece determined by the moving interfaces
\(\abs{\xi}^2=a_\nu(t)^2\pm1\), and its first derivatives are locally
bounded on those pieces.  More precisely, for every \(L>0\),
\begin{equation}\label{eq:q-local-Lipschitz}
\sup_{\substack{0<t\le T,\ \abs{\xi}\le L\\
\abs{\xi}^2\ne a_\nu(t)^2\pm1}}
\left(
\abs{q(t,\xi)}
+\abs{\nabla q(t,\xi)}
+\abs{\partial_tq(t,\xi)}
\right)
<\infty.
\end{equation}
Moreover,
\begin{equation}\label{eq:q-boundary-trace}
q(t,\xi)=0
\qquad
\text{whenever }
0<t\le T
\quad\text{and}\quad
\abs{\xi}^2=a_\nu(t)^2+1.
\end{equation}
\end{lemma}

\begin{proof}
The function \(\mathfrak h\) is smooth on \([0,\infty)\).  Writing
\(s=\abs{\xi}^2\) and \(y=s-a_\nu(t)^2\),
\eqref{eq:p0-gradient}, \eqref{eq:q-def}, and \eqref{eq:h-def} give
\begin{equation}\label{eq:q-compact-form}
q(t,\xi)
=
\pi\mathfrak h(s)
\bigl(1-\widetilde\vartheta(y)\bigr)
\bigl(\xi\cdot\Sigma(t)\xi^\perp\bigr),
\qquad y\le1,
\end{equation}
while \(q(t,\xi)=0\) for \(y>1\).  Since the last factor is quadratic
in \(\xi\), and
\(a_\nu(t)\stackrel{\eqref{eq:standing-smallness}}{\ge}2\),
\eqref{eq:q-compact-form} also proves that the value
\(q(t,0)=0\) removes the apparent singularity at the origin.

The endpoint values are
\[
\widetilde\vartheta(-1)
\stackrel{\substack{\eqref{eq:vartheta-tilde-def}\\
\eqref{eq:weight-cutoff-matching}}}{=}0,
\qquad
\widetilde\vartheta(1)
\stackrel{\substack{\eqref{eq:vartheta-tilde-def}\\
\eqref{eq:weight-cutoff-matching}}}{=}1.
\]
Hence the traces in
\eqref{eq:q-compact-form} match at \(y=-1\), and the inner trace at
\(y=1\) equals the exterior value zero.  Thus \(q\) is continuous
across both moving interfaces, and \eqref{eq:q-boundary-trace} follows.

Fix \(L>0\).  Since \(a_\nu(t)\to\infty\) as \(t\to0\), there
exists \(t_L\in(0,T]\) such that \(y\le-1\) whenever
\(0<t\le t_L\) and \(\abs{\xi}\le L\).  On this set,
\[
q(t,\xi)
\stackrel{\eqref{eq:q-compact-form}}{=}
\pi\mathfrak h(\abs{\xi}^2)
\bigl(\xi\cdot\Sigma(t)\xi^\perp\bigr),
\]
whose value at \(t=0\) is precisely \eqref{eq:q-initial-extension}.
It therefore gives a \(C^1\) extension with bounded first derivatives
on \([0,t_L]\times\{\abs{\xi}\le L\}\).

On \([t_L,T]\times\{\abs{\xi}\le L\}\), formula
\eqref{eq:q-compact-form} is \(C^1\) on each smooth piece and has
bounded one-sided first derivatives.  Together with the matching
traces proved above, this yields local Lipschitz continuity across the
two smooth interfaces and proves \eqref{eq:q-local-Lipschitz}.
\end{proof}

We define the full weight by
\begin{equation}\label{eq:p-def}
\pw(t,\xi)
:=
\pw_0(t,\xi)+(\nu t)q(t,\xi).
\end{equation}

\begin{lemma}\label{lem:q-cell}
For every \(t\in(0,T]\), one has, for almost every
\(\xi\in\{\abs{\xi}^2\le a_\nu(t)^2+1\}\),
\begin{equation}\label{eq:q-cell}
v^G(\xi)\cdot\nabla q(t,\xi)
=
-\Sigma(t)\xi\cdot\nabla\pw_0(t,\xi).
\end{equation}
\end{lemma}
\begin{proof}
Because \(\pw_0\) is radial,
\begin{equation}\label{eq:Sigma-radial-reduction}
\Sigma\xi\cdot\nabla\pw_0
=
\frac{\xi\cdot\nabla\pw_0}{\abs{\xi}^2}
(\xi\cdot\Sigma\xi).
\end{equation}
Since \(\Sigma\) is symmetric and trace free,
\begin{equation}\label{eq:angular-strain-identity}
\partial_\theta
\bigl(\xi\cdot\Sigma\xi^\perp\bigr)
=
-2\xi\cdot\Sigma\xi,
\qquad
\partial_\theta:=\xi^\perp\cdot\nabla.
\end{equation}
Also,
\(v^G\cdot\nabla
=(1-e^{-\abs{\xi}^2/4})(2\pi\abs{\xi}^2)^{-1}\partial_\theta\).
Therefore, for almost every \(\xi\) in this region,
\[ v^G\cdot\nabla q \stackrel{\substack{\eqref{eq:q-def}\\ \eqref{eq:angular-strain-identity}}}{=} -\frac{\xi\cdot\nabla\pw_0}{\abs\xi^2}(\xi\cdot\Sigma\xi) \stackrel{\eqref{eq:Sigma-radial-reduction}}{=} -\Sigma\xi\cdot\nabla\pw_0, \]which is \eqref{eq:q-cell}.
\end{proof}
\begin{lemma}\label{lem:adapted-weight}
For all \(0<t\le T\) and
\(\xi\in\Rt\),
\begin{equation}\label{eq:q-pointwise}
\abs{q}
\lsim{T}
\abs{\xi}^2(1+\abs{\xi}^2)\pw_0,
\end{equation}
and, for every \(0<t\le T\) and almost every \(\xi\in\Rt\),
\begin{equation}\label{eq:q-derivative-pointwise}
\abs{\xi}\abs{\nabla q}
+\abs{t\partial_tq}
\lsim{T}
\abs{\xi}^2(1+\abs{\xi}^2)^2\pw_0.
\end{equation}
Consequently,
\begin{equation}\label{eq:normalized-q-bounds}
\begin{aligned}
(\nu t)\frac{\abs{q}}{\pw_0}
&\lsim{T}
a_0^4,
\\
(\nu t)
\frac{
  \abs{\xi}\abs{\nabla q}
  +\abs{q+t\partial_tq}
}{\pw_0}
&\lsim{T}
a_0^4\chi_\nu(t,\xi).
\end{aligned}
\end{equation}
For every \(0<t\le T\), the first line of
\eqref{eq:normalized-q-bounds} holds everywhere in \(\xi\), and the
second almost everywhere in \(\xi\).
After decreasing \(a_0\), if necessary, one also has
\begin{equation}\label{eq:p-comparable}
\frac34\pw_0
\le
\pw
\le
\frac54\pw_0.
\end{equation}
Moreover,
\begin{equation}\label{eq:p-global-bounds}
\frac12e^{\frac{\beta\abs{\xi}}{4}}
\le
\pw(t,\xi)
\le
2e^{\frac{\abs{\xi}^2}{4}}.
\end{equation}
\end{lemma}
The identity \((\nu t)a_\nu(t)^4=a_0^4\) and
\eqref{eq:normalized-q-bounds} make the angular correction small
relative to \(\pw_0\). They ensure the positivity of \(\pw\) and
control the error terms after the cancellation \eqref{eq:q-cell}.
\begin{proof}
Write
\[
r=\abs{\xi},
\qquad
s=r^2,
\qquad
y=s-a_\nu(t)^2.\]
We first prove the pointwise estimates.  The function \(\mathfrak h\)
is smooth at \(s=0\), and
\begin{equation}\label{eq:h-growth-bounds}
\abs{\mathfrak h(s)}
\stackrel{\eqref{eq:h-def}}{\lesssim}
(1+s)e^{s/4},
\qquad
s\abs{\mathfrak h'(s)}
\stackrel{\eqref{eq:h-def}}{\lesssim}
s(1+s)e^{s/4}.
\end{equation}
Indeed, these bounds follow from smoothness on \(0\le s\le1\), while
for \(s\ge1\) the denominator \(1-e^{-s/4}\) is bounded away from zero
and direct differentiation applies.

On \(y\le1\),
\[
0\le1-\widetilde\vartheta(y)
\stackrel{\eqref{eq:weight-cutoff-properties}}{\le}
1-\vartheta(y).
\]
Moreover, on each smooth piece,
\[
\abs{\partial_y(1-\widetilde\vartheta(y))}
\lesssim \bigl(1-\vartheta(y)\bigr).
\]
The derivative vanishes for \(y<-1\); on \(-1<y<1\), this follows
from the boundedness of \(\vartheta'\), \(\vartheta''\), and the lower
bound \(1-\vartheta\ge e^{-1/4}\).  Since
\[
\nabla_\xi y=2\xi,
\qquad
t\partial_ty=\frac12a_\nu(t)^2,
\]
and a cutoff derivative can occur only for \(\abs y<1\), where
\(a_\nu(t)^2\le1+s\), we obtain, almost everywhere on \(y\le1\),
\begin{equation}\label{eq:cutoff-factor-derivative-bounds}
\begin{aligned}
r\abs{\nabla_\xi(1-\widetilde\vartheta(y))}
&\lesssim
s\bigl(1-\vartheta(y)\bigr),
\\
\abs{t\partial_t(1-\widetilde\vartheta(y))}
&\lesssim
(1+s)\bigl(1-\vartheta(y)\bigr).
\end{aligned}
\end{equation}

The quadratic strain factor satisfies
\begin{equation}\label{eq:quadratic-strain-factor-bounds}
\begin{aligned}
\abs{\xi\cdot\Sigma(t)\xi^\perp}
+r\abs{\nabla(\xi\cdot\Sigma(t)\xi^\perp)}
+\abs{t\partial_t(\xi\cdot\Sigma(t)\xi^\perp)}
\stackrel{\eqref{eq:Sigma-time-bound}}{\lsim{T}} s
\end{aligned}
\end{equation}
Since \(\pw_0=e^{s/4}(1-\vartheta(y))\) on \(y\le1\),
\eqref{eq:q-compact-form}, \eqref{eq:h-growth-bounds}, and
\eqref{eq:quadratic-strain-factor-bounds} imply
\[
\abs q
\lsim{T}
s(1+s)\pw_0.
\]
Differentiating \eqref{eq:q-compact-form} on each smooth piece and
applying \eqref{eq:h-growth-bounds},
\eqref{eq:cutoff-factor-derivative-bounds}, and
\eqref{eq:quadratic-strain-factor-bounds} gives
\[
r\abs{\nabla q}
+
\abs{t\partial_tq}
\lsim{T}
s(1+s)^2\pw_0.
\]
On \(y>1\), \(q\) and its weak first derivatives vanish almost
everywhere, proving \eqref{eq:q-pointwise} and~\eqref{eq:q-derivative-pointwise}.

\smallskip
\noindent\emph{Estimates in the inner region.}
In the inner region \(s\le a_\nu(t)^2+1\), the bound
\(a_\nu(t)\ge2\) gives
\[
1+s
\le
a_\nu(t)^2+2
\le
\frac32a_\nu(t)^2.
\]
Since
\((\nu t)a_\nu(t)^4
\stackrel{\eqref{eq:weight-scales}}{=}a_0^4\), it follows that
\begin{equation}\label{eq:inner-weight-scale}
(\nu t)(1+s)^2
\le
\frac94a_0^4
\qquad\text{if }s\le a_\nu(t)^2+1.
\end{equation}

\[
(\nu t)\frac{\abs q}{\pw_0}
\stackrel{\eqref{eq:q-pointwise}}{\lsim{T}}
(\nu t)s(1+s)
\stackrel{\eqref{eq:inner-weight-scale}}{\lsim{T}}
a_0^4.
\]
This proves the first estimate in
\eqref{eq:normalized-q-bounds}; outside this region, \(q=0\).

For the second estimate, note first that
\[
B_\nu(t)
\stackrel{\eqref{eq:weight-scale-identity}}{=}
\frac{a_\nu(t)^2}{\beta}
\stackrel{\eqref{eq:standing-smallness}}{\ge}a_\nu(t)^2
>
\sqrt{a_\nu(t)^2+1},
\]
where \(a_\nu(t)\ge2\).  Thus this region is contained in the first two regions in the
definition of \(\chi_\nu\).  If \(r\le a_\nu(t)\), then
\[
\chi_\nu(t,\xi)
\stackrel{\eqref{eq:chi-def}}{=}
r^2.
\]
In the remaining part of the inner region,
\[
a_\nu(t)\le r\le\sqrt{a_\nu(t)^2+1},
\]
we have
$
\chi_\nu(t,\xi)
\stackrel{\eqref{eq:chi-def}}{=}
a_\nu(t)^2.
$
Since \(a_\nu(t)\ge2\),
\[
r^2
\le
a_\nu(t)^2+1
\le
\frac54a_\nu(t)^2
=
\frac54\chi_\nu(t,\xi).
\]
Thus \(s\le\frac54\chi_\nu(t,\xi)\) throughout the inner region.

By \eqref{eq:q-derivative-pointwise}, \eqref{eq:q-pointwise}, and
\eqref{eq:inner-weight-scale}, almost everywhere,
\[
(\nu t)
\frac{
r\abs{\nabla q}
+\abs{q+t\partial_tq}
}{\pw_0}
\lsim{T}
a_0^4\chi_\nu(t,\xi).
\]
This proves the second estimate in
\eqref{eq:normalized-q-bounds}.

$
\pw
\stackrel{\eqref{eq:p-def}}{=}
\pw_0+(\nu t)q,
$
and the first normalized estimate gives
\[
\left|
\frac{\pw-\pw_0}{\pw_0}
\right|
=
(\nu t)\frac{\abs q}{\pw_0}
\stackrel{\eqref{eq:normalized-q-bounds}}{\lsim{T}}
a_0^4.
\]
After decreasing \(a_0\), if necessary, so that the last bound is
at most \(1/4\), we conclude that
\[
\frac34\pw_0
\le
\pw
\le
\frac54\pw_0.
\]
This proves \eqref{eq:p-comparable}.

\smallskip
\noindent\emph{Global upper and lower bounds.}
We first establish a lower bound for \(\pw_0\). In the inner and
transition regions,
\[
\pw_0
\stackrel{\eqref{eq:p0-vs-gaussian}}{\ge}
e^{-1/4}e^{r^2/4}.
\]
Moreover, since \(\beta\le1\),
\[
r^2-\beta r
=
\left(r-\frac{\beta}{2}\right)^2-\frac{\beta^2}{4}
\ge
-\frac14.
\]
Therefore,
\[
e^{-1/4}e^{r^2/4}
=
e^{-1/4}e^{\beta r/4}
e^{(r^2-\beta r)/4}
\ge
e^{-5/16}e^{\beta r/4}.
\]
Thus
\[
\pw_0
\ge
e^{-5/16}e^{\beta r/4}
\]
in the inner and transition regions.

On the plateau, \(r\le B_\nu(t)\), and hence
\[
\beta r
\le
\beta B_\nu(t)
\stackrel{\eqref{eq:weight-scale-identity}}{=}
a_\nu(t)^2.
\]
Consequently,
\[
\pw_0
=
e^{a_\nu(t)^2/4}
\ge
e^{\beta r/4}.
\]
In the outer region,
\[
\pw_0
\stackrel{\eqref{eq:p0-def}}{=}
e^{\beta r/4}
\]
and therefore we have proved the global lower bound
\begin{equation}\label{eq:p0-global-lower}
\pw_0(t,\xi)
\ge
e^{-5/16}e^{\beta r/4}.
\end{equation}

We next prove the global upper bound
\begin{equation}\label{eq:p0-below-gaussian}
\pw_0(t,\xi)
\le
\pw_G(\xi):=e^{r^2/4}.
\end{equation}
In the inner and transition regions,
\(\pw_0\stackrel{\eqref{eq:p0-def}}{\le}e^{r^2/4}\).
On the plateau,
\[
\pw_0
\stackrel{\eqref{eq:p0-def}}{=}
e^{a_\nu(t)^2/4}\le e^{r^2/4},
\]
because \(r^2\ge a_\nu(t)^2+1\). In the outer region,
\[
\pw_0
\stackrel{\eqref{eq:p0-def}}{=}
e^{\beta r/4}\le e^{r^2/4},
\]
because \(r\ge B_\nu(t)>1\) and \(\beta\le1\), so that
\(\beta r\le r^2\). This proves
\eqref{eq:p0-below-gaussian}.

Finally,
\[
\pw
\stackrel{\eqref{eq:p-comparable}}{\ge}
\frac34\pw_0
\stackrel{\eqref{eq:p0-global-lower}}{\ge}
\frac34e^{-5/16}e^{\beta r/4}
\ge
\frac12e^{\beta r/4},
\]
because \((3/4)e^{-5/16}>1/2\), and
\[
\pw
\stackrel{\eqref{eq:p-comparable}}{\le}
\frac54\pw_0
\stackrel{\eqref{eq:p0-below-gaussian}}{\le}
\frac54e^{r^2/4}
\le
2e^{r^2/4}.
\]
Therefore,
\[
\frac12e^{\frac{\beta\abs{\xi}}{4}}
\le
\pw(t,\xi)
\le
2e^{\frac{\abs{\xi}^2}{4}}.
\]
This proves \eqref{eq:p-global-bounds} and completes the proof.
\end{proof}

\subsection{Weighted energy identity}
For a positive weight \(\pw=\pw(t,\xi)\), define
\begin{equation}\label{eq:energy}
E_{\pw}(t)
:=
\frac12\int_{\Rt}\pw(t,\xi)\abs{w(\xi,t)}^2\,d\xi.
\end{equation}
For real-valued functions \(g,h\), define
\begin{equation}\label{eq:p-weighted-pairing}
\inner{g}{h}_{\pw(t)}
:=
\int_{\Rt}\pw(t,\xi)g(\xi)h(\xi)\,d\xi.
\end{equation}
\begin{lemma}\label{lem:energy-id}
Let \(w\) be a smooth solution of \eqref{eq:linearized}, let
\(v=K\star_\xi w\), and let \(\pw=\pw(t,\xi)\) be a positive,
continuous, locally Lipschitz weight with the integrability required
below. The first weak derivatives of \(\pw\) agree almost everywhere
with its derivatives on the smooth pieces, and continuity at the
transition circles cancels the interface terms. Then, for almost every
time,
\begin{equation}\label{eq:energy-id}
\begin{aligned}
t\frac{d}{dt}E_{\pw}(t)
\stackrel{\substack{\eqref{eq:linearized}\\
\eqref{eq:energy}}}{=}&
\frac{t}{2}\int_{\Rt}(\partial_t\pw)\abs{w}^2\,d\xi
-
\int_{\Rt}\pw\abs{\nabla w}^2\,d\xi
-
\int_{\Rt}w\nabla\pw\cdot\nabla w\,d\xi
\\
&-
\frac14\int_{\Rt}(\xi\cdot\nabla\pw)\abs{w}^2\,d\xi
-
\frac12\int_{\Rt}\pw\abs{w}^2\,d\xi
\\
&+
\frac{1}{2\nu}
\int_{\Rt}(v^G\cdot\nabla\pw)\abs{w}^2\,d\xi
+
\frac12\sqrt{\frac{t}{\nu}}
\int_{\Rt}(b\cdot\nabla\pw)\abs{w}^2\,d\xi
\\
&-
\frac1\nu
\int_{\Rt}\pw(v\cdot\nabla G)w\,d\xi
+
\inner{f(t)}{w(t)}_{\pw(t)}.
\end{aligned}
\end{equation}
\end{lemma}
\begin{proof}
Since \(\cL w-w=\Delta w+\frac12\xi\cdot\nabla w\),
integration by parts gives
\[
\int_{\Rt}\pw w\Delta w\,d\xi
=
-\int_{\Rt}\pw\abs{\nabla w}^2\,d\xi
-\int_{\Rt}w\nabla\pw\cdot\nabla w\,d\xi
\]
and
\[
\begin{aligned}
\frac12\int_{\Rt}\pw w\,\xi\cdot\nabla w\,d\xi
&=
\frac14\int_{\Rt}\pw\,\xi\cdot\nabla\abs w^2\,d\xi
\\
&=
-\frac12\int_{\Rt}\pw\abs w^2\,d\xi
-\frac14\int_{\Rt}(\xi\cdot\nabla\pw)\abs w^2\,d\xi.
\end{aligned}
\]
Moreover,
\(\int_{\Rt}\pw w\,t\partial_tw\,d\xi
=t\,dE_{\pw}/dt-(t/2)\int_{\Rt}(\partial_t\pw)\abs w^2\,d\xi\).
Finally, \(v^G\) and \(b\) are divergence free, and therefore
\[
\begin{aligned}
-\int_{\Rt}\pw w\,(v^G\cdot\nabla w)\,d\xi
&=\frac12\int_{\Rt}(v^G\cdot\nabla\pw)\abs w^2\,d\xi,
\\
-\int_{\Rt}\pw w\,(b\cdot\nabla w)\,d\xi
&=\frac12\int_{\Rt}(b\cdot\nabla\pw)\abs w^2\,d\xi.
\end{aligned}
\]
These identities establish \eqref{eq:energy-id}.
For a piecewise \(C^1\) weight, one may perform the calculation in each
smooth region; the interface contributions cancel because \(\pw\) is
continuous.  Equivalently, approximate \(\pw\) by smooth positive weights
and pass to the limit.
\end{proof}
We write
\begin{equation}\label{eq:energy-parts}
t\frac{d}{dt}E_{\pw}(t)
\stackrel{\eqref{eq:energy-id}}{=}
\mathcal D_{\pw}(t)
+\mathcal T_{\pw}(t)
+\mathcal R_{\pw}(t)
+\mathcal F_{\pw}(t),
\end{equation}
where
\begin{equation}\label{eq:energy-parts-def}
\begin{aligned}
\mathcal D_{\pw}(t)
&:=
\frac{t}{2}\int(\partial_t\pw)\abs{w}^2
-\int\pw\abs{\nabla w}^2
-\int w\nabla\pw\cdot\nabla w
\\
&\quad
-\frac14\int(\xi\cdot\nabla\pw)\abs{w}^2
-\frac12\int\pw\abs{w}^2,
\\
\mathcal T_{\pw}(t)
&:=
\frac{1}{2\nu}\int(v^G\cdot\nabla\pw)\abs{w}^2
+\frac12\sqrt{\frac{t}{\nu}}
\int(b\cdot\nabla\pw)\abs{w}^2,
\\
\mathcal R_{\pw}(t)
&:=
-\frac1\nu\int\pw(v\cdot\nabla G)w,
\\
\mathcal F_{\pw}(t)
&:=
\inner{f(t)}{w(t)}_{\pw(t)}.
\end{aligned}
\end{equation}
All integrals without an indicated domain are over \(\Rt\).

\subsection{Diffusive coercivity}
Define
\begin{equation}\label{eq:IJK-def}
\begin{aligned}
\mathcal I(t)
&:=
\int_{\Rt}\pw_0\abs{\nabla w}^2\,d\xi,
\\
\mathcal J(t)
&:=
\int_{\Rt}\pw_0\chi_\nu(t,\xi)\abs{w}^2\,d\xi,
\\
\mathcal K(t)
&:=
\int_{\Rt}\pw_0\abs{w}^2\,d\xi,
\end{aligned}
\end{equation}
and
\begin{equation}\label{eq:Q-def}
\mathcal Q(t)
:=
\mathcal I(t)+\mathcal J(t)+\mathcal K(t).
\end{equation}

\begin{equation}\label{eq:weighted-energy-comparable}
\frac38\mathcal K(t)
\stackrel{\eqref{eq:p-comparable}}{\le}
E_{\pw}(t)
\stackrel{\eqref{eq:p-comparable}}{\le}
\frac58\mathcal K(t).
\end{equation}

We estimate the diffusion term \(\mathcal D_{\pw}\), the transport
term \(\mathcal T_{\pw}\), and the nonlocal Oseen term
\(\mathcal R_{\pw}\) separately. The forcing contributes
\(\mathcal F_{\pw}=\inner{f(t)}{w(t)}_{\pw(t)}\).
\begin{theorem}\label{thm:weight-coercivity}
Recall \(\mathcal D_{\pw}(t)\) from
\eqref{eq:energy-parts-def}, \(\mathcal I(t),\mathcal J(t),\mathcal K(t)\)
from \eqref{eq:IJK-def}, and \(\mathcal Q(t)\) from
\eqref{eq:Q-def}.  For \(a_0\) sufficiently small, depending only on
\(T\), and under \eqref{eq:standing-smallness},
\begin{equation}\label{eq:weight-coercivity}
\mathcal D_{\pw}(t)
\le
-\frac3{16}\mathcal I(t)
-\frac1{128}\mathcal J(t)
-\frac38\mathcal K(t)
\le
-\frac1{128}\mathcal Q(t).
\end{equation}
\end{theorem}
\begin{proof}
Expanding the angular part of the weight gives
\begin{equation}\label{eq:diffusive-full-splitting}
\begin{aligned}
\mathcal D_{\pw}(t)
&\stackrel{\substack{\eqref{eq:energy-parts-def}\\
\eqref{eq:p-def}}}{=}
\frac{t}{2}\int_{\Rt}(\partial_t\pw_0)|w|^2\,d\xi
-\int_{\Rt}\pw_0|\nabla w|^2\,d\xi
-\int_{\Rt}w\nabla\pw_0\cdot\nabla w\,d\xi
\\
&-
\frac14\int_{\Rt}(\xi\cdot\nabla\pw_0)|w|^2\,d\xi
-\frac12\int_{\Rt}\pw_0|w|^2\,d\xi
\\
&+
\frac{\nu t}{2}
\int_{\Rt}(q+t\partial_tq)|w|^2\,d\xi
-(\nu t)\int_{\Rt}q|\nabla w|^2\,d\xi
\\
&-
(\nu t)\int_{\Rt}w\nabla q\cdot\nabla w\,d\xi
-\frac{\nu t}{4}
\int_{\Rt}(\xi\cdot\nabla q)|w|^2\,d\xi
\\
&-
\frac{\nu t}{2}\int_{\Rt}q|w|^2\,d\xi .
\end{aligned}
\end{equation}

We first estimate the contribution of the background weight
\(\pw_0\). Let
\[
\begin{aligned}
\mathcal D_{\pw_0}
:={}&
\frac{t}{2}\int_{\Rt}(\partial_t\pw_0)|w|^2\,d\xi
-\int_{\Rt}\pw_0|\nabla w|^2\,d\xi
-\int_{\Rt}w\nabla\pw_0\cdot\nabla w\,d\xi
\\
&-
\frac14\int_{\Rt}(\xi\cdot\nabla\pw_0)|w|^2\,d\xi
-\frac12\int_{\Rt}\pw_0|w|^2\,d\xi .
\end{aligned}
\]

The mixed term satisfies, by Young's inequality,
\[
\begin{aligned}
\left|
\int_{\Rt}w\nabla\pw_0\cdot\nabla w\,d\xi
\right|
&\le
\frac34
\int_{\Rt}\pw_0|\nabla w|^2\,d\xi
\\
&\quad+
\frac13
\int_{\Rt}
\frac{|\nabla\pw_0|^2}{\pw_0}|w|^2\,d\xi .
\end{aligned}
\]
Therefore,
\[
\begin{aligned}
\mathcal D_{\pw_0}
\le{}&
-\frac14
\int_{\Rt}\pw_0|\nabla w|^2\,d\xi
\\
&+
\int_{\Rt}
\left(
\frac t2\partial_t\pw_0
-\frac14\xi\cdot\nabla\pw_0
-\frac12\pw_0
+\frac13\frac{|\nabla\pw_0|^2}{\pw_0}
\right)
|w|^2\,d\xi .
\end{aligned}
\]

The coefficient of \(|w|^2\) can be written as
\[
\begin{aligned}
&
\frac t2\partial_t\pw_0
-\frac14\xi\cdot\nabla\pw_0
-\frac12\pw_0
+\frac13\frac{|\nabla\pw_0|^2}{\pw_0}
\\
&=
-\pw_0
\left(
\frac{\xi\cdot\nabla\pw_0}{4\pw_0}
-\frac{|\nabla\pw_0|^2}{3\pw_0^2}
-\frac{t\partial_t\pw_0}{2\pw_0}
+\frac12
\right).
\end{aligned}
\]

The pointwise coercivity estimate gives
\[
\frac{\xi\cdot\nabla\pw_0}{4\pw_0}
-\frac{|\nabla\pw_0|^2}{3\pw_0^2}
-\frac{t\partial_t\pw_0}{2\pw_0}
\stackrel{\eqref{eq:p0-pointwise-coercivity}}{\ge}
\frac1{48}\chi_\nu(t,\xi).
\]
Consequently,
\[
\begin{aligned}
\mathcal D_{\pw_0}
\le{}&
-\frac14
\int_{\Rt}\pw_0|\nabla w|^2\,d\xi
\\
&-
\int_{\Rt}
\pw_0
\left(
\frac1{48}\chi_\nu(t,\xi)+\frac12
\right)
|w|^2\,d\xi .
\end{aligned}
\]
\begin{equation}\label{eq:p0-part-coercivity}
\mathcal D_{\pw_0}
\stackrel{\eqref{eq:IJK-def}}{\le}
-\frac14\mathcal I(t)
-\frac1{48}\mathcal J(t)
-\frac12\mathcal K(t).
\end{equation}

It remains to control the five terms in
\eqref{eq:diffusive-full-splitting} that contain \(q\).  Denote their
sum by \(\mathcal D_q(t)\), so that
\(\mathcal D_{\pw}(t)=\mathcal D_{\pw_0}(t)+\mathcal D_q(t)\).
Choose \(c_{q,T}>0\), independent of \(a_0,B_0,\nu\),
large enough that the three estimates below hold, and set
\(\delta_T:=c_{q,T}a_0^4\).
Then, for every \(0<t\le T\) and almost every \(\xi\in\Rt\),
\begin{equation}\label{eq:q-diffusive-normalized-bounds}
\begin{aligned}
(\nu t)\abs q
&\stackrel{\eqref{eq:normalized-q-bounds}}{\le}
\delta_T\pw_0,
\\
(\nu t)
\bigl(
  \abs{q+t\partial_tq}
  +\abs\xi\abs{\nabla q}
\bigr)
&\stackrel{\eqref{eq:normalized-q-bounds}}{\le}
\delta_T\pw_0\chi_\nu,
\\
(\nu t)\abs{\nabla q}
&\stackrel{\substack{\eqref{eq:normalized-q-bounds}\\
\eqref{eq:chi-def}}}{\le}
\delta_T\pw_0\chi_\nu^{1/2}.
\end{aligned}
\end{equation}
For the third line, in the inner region \(\abs\xi^2\le a_\nu(t)^2+1\),
\(a_\nu(t)\ge2\) implies
\(\chi_\nu(t,\xi)\stackrel{\eqref{eq:chi-def}}{\simeq}\abs\xi^2\).
Thus \(\chi_\nu/\abs\xi\lesssim \chi_\nu^{1/2}\) for \(\xi\ne0\);
outside this region, the weak derivatives of \(q\) vanish almost everywhere.

\begin{equation}\label{eq:q-gradient-coefficient-bound}
\left|
(\nu t)\int_{\Rt}q\abs{\nabla w}^2\,d\xi
\right|
\stackrel{\eqref{eq:q-diffusive-normalized-bounds}}{\le}
\delta_T\mathcal I(t).
\end{equation}
For the mixed term, Young's inequality gives
\begin{align}
\left|
(\nu t)\int_{\Rt}w\nabla q\cdot\nabla w\,d\xi
\right|
&\stackrel{\eqref{eq:q-diffusive-normalized-bounds}}{\le}
\delta_T
\int_{\Rt}
\pw_0\chi_\nu^{1/2}\abs w\abs{\nabla w}\,d\xi
\le
\frac1{32}\mathcal I(t)
+8\delta_T^2\mathcal J(t).
\label{eq:q-mixed-diffusive-bound}
\end{align}
Finally, the zeroth-order terms satisfy
\begin{align}
&\left|
\frac{\nu t}{2}
\int_{\Rt}(q+t\partial_tq)\abs w^2\,d\xi
\right|
+
\left|
\frac{\nu t}{4}
\int_{\Rt}(\xi\cdot\nabla q)\abs w^2\,d\xi
\right|
\stackrel{\eqref{eq:q-diffusive-normalized-bounds}}{\le}
\frac34\delta_T\mathcal J(t),
\label{eq:q-confinement-terms-bound}
\\
&\left|
\frac{\nu t}{2}
\int_{\Rt}q\abs w^2\,d\xi
\right|
\stackrel{\eqref{eq:q-diffusive-normalized-bounds}}{\le}
\frac12\delta_T\mathcal K(t).
\label{eq:q-L2-term-bound}
\end{align}
Combining \eqref{eq:q-gradient-coefficient-bound},
\eqref{eq:q-mixed-diffusive-bound},
\eqref{eq:q-confinement-terms-bound}, and \eqref{eq:q-L2-term-bound},
we bound the sum of the five terms by
\begin{equation}\label{eq:q-part-coercivity-error}
\mathcal D_q(t)
\le
\left(\delta_T+\frac1{32}\right)\mathcal I(t)
+\left(\frac34\delta_T+8\delta_T^2\right)\mathcal J(t)
+\frac12\delta_T\mathcal K(t).
\end{equation}
Since \(\delta_T=c_{q,T}a_0^4\), taking \(a_0\) small enough, depending
only on \(T\), gives
\begin{equation}\label{eq:q-part-absorbed}
\mathcal D_q(t)
\stackrel{\eqref{eq:q-part-coercivity-error}}{\le}
\frac1{16}\mathcal I(t)
+\frac5{384}\mathcal J(t)
+\frac18\mathcal K(t).
\end{equation}
Since \(1/48-5/384=1/128\) and
\(\mathcal I,\mathcal J,\mathcal K\ge0\),
\[
\begin{aligned}
\mathcal D_{\pw}
&\stackrel{\substack{\eqref{eq:p0-part-coercivity}\\
\eqref{eq:q-part-absorbed}}}{\le}
-\frac3{16}\mathcal I-\frac1{128}\mathcal J-\frac38\mathcal K
\\
&\stackrel{\eqref{eq:Q-def}}{\le}-\frac1{128}\mathcal Q.
\end{aligned}
\]
This proves \eqref{eq:weight-coercivity}.
\end{proof}
\subsection{Cancellation and estimate of the ambient transport}
Let
\begin{equation}\label{eq:inner-region}
\Omega_{\mathrm{in}}(t)
:=
\left\{
\xi:
\abs{\xi}^2\le a_\nu(t)^2+1
\right\},
\end{equation}

\begin{equation}\label{eq:plateau-region}
\Omega_{\mathrm{mid}}(t)
:=
\left\{
\xi:
a_\nu(t)^2+1
\le
\abs{\xi}^2
\le
B_\nu(t)^2
\right\},
\end{equation}

and
\begin{equation}\label{eq:outer-region}
\Omega_{\mathrm{out}}(t)
:=
\left\{
\xi:
\abs{\xi}\ge B_\nu(t)
\right\}.
\end{equation}

\smallskip
If \(\xi\in\Omega_{\mathrm{in}}(t)\), then
\[
\sqrt{\nu t}\abs{\xi}
\stackrel{\eqref{eq:inner-region}}{\le}
2\sqrt{\nu t}\,a_\nu(t)
\stackrel{\eqref{eq:weight-scales}}{=}
2a_0(\nu t)^{\frac{1}{4}}
\stackrel{\eqref{eq:standing-smallness}}{\le}
a_0^2
\le
\frac12.
\]
Thus the local estimates in \eqref{eq:linear-b-bounds} apply on
the support of \(q\).
Define
\begin{equation}\label{eq:rho-b-def}
\rho_b(\xi,t)
:=
b(\xi,t)\cdot\xi
-
\sqrt{\nu t}\,\xi\cdot\Sigma(t)\xi.
\end{equation}
\begin{equation}\label{eq:rho-b-bound}
\abs{\rho_b(\xi,t)}
\stackrel{\eqref{eq:linear-b-bounds}}{\lsim{T}}
\nu t\abs{\xi}^3,
\qquad
\xi\in\Omega_{\mathrm{in}}(t).
\end{equation}
\begin{lemma}\label{lem:transport-estimates}
For the weight \eqref{eq:p-def}, there exists
\(c_{\mathrm{tr},T}\ge1\), independent of \(a_0,B_0,\nu\)
in the ranges considered above, such that
\begin{equation}\label{eq:transport-total-bound}
\abs{\mathcal T_{\pw}(t)}
\le
c_{\mathrm{tr},T}
\left(
a_0(\nu T)^{\frac14}
+a_0^4
+\frac1{B_0}
\right)\mathcal J(t),
\qquad 0<t\le T.
\end{equation}
\end{lemma}

\begin{proof}
Fix \(t\in(0,T]\), and suppress the time dependence of the
integration domains. Since \(\pw_0\) is radial,
\(v^G\cdot\nabla\pw_0=0\). Moreover,
\(q=\nabla q=0\) almost everywhere outside
\(\Omega_{\mathrm{in}}\), and
\(\nabla\pw_0=0\) on \(\Omega_{\mathrm{mid}}\).
Thus \eqref{eq:energy-parts-def}, \eqref{eq:p-def}, and
\eqref{eq:q-cell} give
\[
\begin{aligned}
\mathcal T_{\pw}(t)
={}&
\frac12\sqrt{\frac t\nu}
\int_{\Omega_{\mathrm{in}}}
\bigl(b-\sqrt{\nu t}\,\Sigma(t)\xi\bigr)
\cdot\nabla\pw_0\,\abs w^2\,d\xi
\\
&+
\frac{\nu t}{2}\sqrt{\frac t\nu}
\int_{\Omega_{\mathrm{in}}}
(b\cdot\nabla q)\abs w^2\,d\xi
\\
&+
\frac12\sqrt{\frac t\nu}
\int_{\Omega_{\mathrm{out}}}
(b\cdot\nabla\pw_0)\abs w^2\,d\xi
\\
=:{}&
\mathcal T_1+\mathcal T_2+\mathcal T_3.
\end{aligned}
\]
In particular, the leading strain contribution cancels exactly.

~\\
By \eqref{eq:p0-gradient} and \eqref{eq:rho-b-def},
\[
\mathcal T_1
=
\frac14\sqrt{\frac t\nu}
\int_{\Omega_{\mathrm{in}}}
\rho_b e^{\frac{\abs\xi^2}{4}}
\left(
1-\widetilde\vartheta
\bigl(\abs\xi^2-a_\nu(t)^2\bigr)
\right)\abs w^2\,d\xi.
\]
Since
\(0\le1-\widetilde\vartheta\le1-\vartheta\),
\eqref{eq:p0-def} and \eqref{eq:rho-b-bound} imply
\[
\abs{\mathcal T_1}
\lsim{T}
t\sqrt{\nu t}
\int_{\Omega_{\mathrm{in}}}
\abs\xi^3\pw_0\abs w^2\,d\xi.
\]
The standing assumption \eqref{eq:standing-smallness} gives
\(a_\nu(t)\ge2\). Hence, by \eqref{eq:chi-def},
\[
\abs\xi\le2a_\nu(t),
\qquad
\abs\xi^2\le\frac54\chi_\nu(t,\xi),
\qquad \xi\in\Omega_{\mathrm{in}},
\]
and therefore
\[
\begin{aligned}
\abs{\mathcal T_1}
&\lsim{T}
t\sqrt{\nu t}\,a_\nu(t)
\int_{\Omega_{\mathrm{in}}}
\pw_0\chi_\nu\abs w^2\,d\xi
\le
C_T t a_0(\nu t)^{\frac14}\mathcal J(t),
\end{aligned}
\]
where we used \eqref{eq:weight-scales} and
\eqref{eq:IJK-def}.

~\\
For the second term, the local velocity bound in
\eqref{eq:linear-b-bounds} gives
\[
\sqrt{\frac t\nu}\abs{b(\xi,t)}
\lsim{T}t\abs\xi,
\qquad \xi\in\Omega_{\mathrm{in}}.
\]
Using \eqref{eq:normalized-q-bounds}, we obtain
\[
\begin{aligned}
\abs{\mathcal T_2}
&\lsim{T}
t(\nu t)
\int_{\Omega_{\mathrm{in}}}
\abs\xi\abs{\nabla q}\abs w^2\,d\xi
\\
&\lsim{T}
t a_0^4
\int_{\Omega_{\mathrm{in}}}
\pw_0\chi_\nu\abs w^2\,d\xi
\\
&\le C_T t a_0^4\mathcal J(t).
\end{aligned}
\]

~\\
Finally, on \(\Omega_{\mathrm{out}}\),
\eqref{eq:linear-b-bounds}, \eqref{eq:p0-gradient}, and
\eqref{eq:chi-def} give
\[
\abs b\lsim{T}1,
\qquad
\abs{\nabla\pw_0}=\frac{\beta}{4}\pw_0,
\qquad
\chi_\nu=\beta\abs\xi.
\]
Since
\(\abs\xi\ge B_\nu(t)=B_0(\nu t)^{-1/2}\), we also have
\[
\sqrt{\frac t\nu}\le\frac{t\abs\xi}{B_0}.
\]
Consequently,
\[
\begin{aligned}
\abs{\mathcal T_3}
&\lsim{T}
\frac{t}{B_0}
\int_{\Omega_{\mathrm{out}}}
\beta\abs\xi\,\pw_0\abs w^2\,d\xi
\le
\frac{C_Tt}{B_0}\mathcal J(t).
\end{aligned}
\]
Summing the three estimates and using \(t\le T\) proves
\eqref{eq:transport-total-bound}, with
\(c_{\mathrm{tr},T}\) independent of \(a_0,B_0,\nu\).
\end{proof}

\subsection{Biot--Savart bounds and the nonlocal Oseen term}
\begin{lemma}
\label{lem:basic-biotsavart}
For the fixed parameters \(a_0,B_0\), uniformly in \(\nu\),
\begin{equation}\label{eq:w-L43}
\norm[L^{\frac{4}{3}}]{w}
\lsim{T,\beta}
\mathcal K(t)^{\frac{1}{2}},
\end{equation}
\begin{equation}\label{eq:v-L4}
\norm[L^4]{v}
\lsim{T,\beta}
\mathcal K(t)^{\frac{1}{2}},
\end{equation}
and
\begin{equation}\label{eq:w-L4}
\norm[L^4]{w}
\lesssim
\mathcal K(t)^{\frac{1}{4}}\mathcal I(t)^{\frac{1}{4}}.
\end{equation}
Consequently,
\begin{equation}\label{eq:v-Linfty-basic}
\norm[L^\infty]{v}
\lsim{T,\beta}
\mathcal K(t)^{\frac{1}{4}}
\bigl(\mathcal I(t)+\mathcal K(t)\bigr)^{\frac{1}{4}}.
\end{equation}
\end{lemma}
\begin{proof}
The weight satisfies
\(\pw_0^{-1/2}\stackrel{\eqref{eq:p0-global-lower}}{\lesssim}
e^{-\beta\abs{\xi}/8}\in L^4(\Rt)\), where the
norm depends only on the fixed \(\beta>0\).  H\"older's inequality and
the Hardy--Littlewood--Sobolev inequality therefore give
\[
\norm[L^{4/3}]{w}
\stackrel{\substack{\eqref{eq:p-comparable}\\
\eqref{eq:p-global-bounds}}}{\lsim{\beta}}
\mathcal K^{1/2},
\qquad
\norm[L^4]{v}
\lesssim\norm[L^{4/3}]{w}.
\]
Since \(\pw_0\ge1\), the two-dimensional Gagliardo--Nirenberg inequality
also gives
\[
\norm[L^4]{w}\lesssim
\norm[L^2]{w}^{1/2}\norm[L^2]{\nabla w}^{1/2}
\le\mathcal K^{1/4}\mathcal I^{1/4}.
\]
Finally, splitting the Biot--Savart integral into
\(\{\abs{\xi-\eta}<R\}\) and its complement yields
\[
\norm[L^\infty]{v}
\lesssim
\left(
R^{1/2}\norm[L^4]{w}
+
R^{-1/2}\norm[L^{4/3}]{w}
\right).
\]
Choosing \(R=\norm[L^{4/3}]{w}/\norm[L^4]{w}\), with the zero cases
understood by approximation, gives
\[
\begin{aligned}
\norm[L^\infty]{v}
&\lesssim \norm[L^4]{w}^{1/2}\norm[L^{4/3}]{w}^{1/2}
\\
&\stackrel{\substack{\eqref{eq:w-L4}\\
\eqref{eq:w-L43}}}{\lsim{T,\beta}}
\mathcal K^{3/8}\mathcal I^{1/8}
\lsim{T,\beta} \mathcal K^{1/4}(\mathcal I+\mathcal K)^{1/4}.
\end{aligned}
\]
This proves \eqref{eq:v-Linfty-basic}.
\end{proof}
\begin{lemma}\label{lem:v-decay}
Assume \(w\in H^1(\Rt;\pw\,d\xi)\) and \(\int_{\Rt}w(\xi)\,d\xi=0\).
Then
\begin{equation}\label{eq:v-decay}
\abs{v(\xi)}
\lsim{\beta}
\frac{1}{1+\abs{\xi}^2}
\left(
\norm[L^2_{\pw}]{w}
+
\norm[L^2_{\pw}]{\nabla w}
\right).
\end{equation}
The implicit constant is uniform in \(t\in(0,T]\) and
\(0<\nu\le\nu_T\), with the fixed \(\beta=\frac{a_0^2}{B_0}>0\).
\end{lemma}

\begin{proof}
Set \(g(\xi)=(1+\abs\xi^2)w(\xi)\). We first prove
\begin{equation}\label{eq:v-weighted-unweighted}
\norm[L^\infty]{
(1+\abs{\xi}^2)v
}
\lesssim
\norm[L^4\cap L^{\frac{4}{3}}]{
(1+\abs{\xi}^2)w
}.
\end{equation}
Write \(M=\norm[L^4]{g}+\norm[L^{4/3}]{g}\).
Splitting the kernel at unit distance and applying H\"older's
inequality gives
\[
\sup_{\xi\in\Rt}\int_{\Rt}\abs{K(\xi-\eta)}\abs{g(\eta)}\,d\eta
\lesssim M,
\qquad
\int_{\Rt}\abs\eta\abs{w(\eta)}\,d\eta\lesssim M,
\]
where the second estimate uses
\(\abs\eta/(1+\abs\eta^2)\in L^4(\Rt)\).
The first estimate controls \(v(\xi)\) when \(\abs\xi\le2\).
For \(R=\abs\xi\ge2\), the zero-mass condition yields
\[
\begin{aligned}
v(\xi)
={}&\int_{\abs\eta\le R/2}
\bigl(K(\xi-\eta)-K(\xi)\bigr)w(\eta)\,d\eta
\\
&+\int_{\abs\eta>R/2}K(\xi-\eta)w(\eta)\,d\eta
-K(\xi)\int_{\abs\eta>R/2}w(\eta)\,d\eta.
\end{aligned}
\]
On the first region,
\(\abs{K(\xi-\eta)-K(\xi)}\lesssim R^{-2}\abs\eta\).
On the second region, \(\abs{w(\eta)}\le4R^{-2}\abs{g(\eta)}\).
The last term is bounded by
\(CR^{-2}\int\abs\eta\abs{w(\eta)}\,d\eta\).
Thus all three terms are bounded by \(CR^{-2}M\), proving
\eqref{eq:v-weighted-unweighted}.

~\\
Set \(W=(1+\abs{\xi}^2)\pw^{-1/2}\). The weight satisfies
\(W\stackrel{\eqref{eq:p-global-bounds}}{\le}
\sqrt2(1+\abs{\xi}^2)e^{-\beta\abs{\xi}/8}
\in L^4\cap L^\infty\).  Hence
\[
\begin{aligned}
\norm[L^{4/3}]{g}
&\le\norm[L^4]{W}\norm[L^2_{\pw}]{w},\\
\norm[L^2]{g}
&\le\norm[L^\infty]{W}\norm[L^2_{\pw}]{w}.
\end{aligned}
\]
Moreover, since
\(\nabla g=2\xi w+(1+\abs{\xi}^2)\nabla w\),
\[
\norm[L^2]{\nabla g}
\lesssim
\norm[L^\infty]{W}
\left(
\norm[L^2_{\pw}]{w}
+
\norm[L^2_{\pw}]{\nabla w}
\right).
\]
The two-dimensional Gagliardo--Nirenberg inequality controls
\(\norm[L^4]{g}\).  Consequently,
\[
\begin{aligned}
\norm[L^\infty]{(1+\abs\xi^2)v}
&\stackrel{\eqref{eq:v-weighted-unweighted}}{\lesssim}
\norm[L^4\cap L^{4/3}]{g}
\\
&\lsim{\beta} 
\left(\norm[L^2_{\pw}]{w}
+\norm[L^2_{\pw}]{\nabla w}\right),
\end{aligned}
\]
which proves \eqref{eq:v-decay}.
\end{proof}
\begin{lemma}\label{lem:reaction-cancellation}
For every real-valued \(w\) such that
\((1+\abs\xi)w\in L^1\cap L^2\), with \(v=K\star w\), \(\int_{\Rt}(v\cdot\xi)w\,d\xi=0\).
Consequently,
\begin{equation}\label{eq:gaussian-reaction-zero}
\int_{\Rt}
e^{\frac{\abs{\xi}^2}{4}}
(v\cdot\nabla G)w\,d\xi
=0.
\end{equation}
\end{lemma}
\begin{proof}
Symmetrizing the Biot--Savart integral gives
\[
\begin{aligned}
2\int_{\Rt}(v\cdot\xi)w\,d\xi
&=
\iint_{\Rt\times\Rt}
K(\xi-\eta)\cdot(\xi-\eta)
w(\xi)w(\eta)\,d\eta\,d\xi
\\
&=0,
\end{aligned}
\]
because \(K(z)\cdot z=0\). Since
\(\nabla G(\xi)=-(8\pi)^{-1}\xi e^{-\abs{\xi}^2/4}\),
\eqref{eq:gaussian-reaction-zero} follows.
\end{proof}
\begin{lemma}\label{lem:reaction}
For the fixed parameters \(a_0,B_0\) and every \(\eta>0\),
there exists \(\nu_{T,\eta}>0\) such that
\begin{equation}\label{eq:reaction-bound}
\bigl|\mathcal R_{\mathsf p}(t)\bigr|
\le
\eta\mathcal K(t)
+C_{T,\beta}\,t\mathcal K(t),
\qquad
0<\nu\le\nu_{T,\eta},
\quad 0<t\le T,
\end{equation}
where \(C_{T,\beta}\) is independent of \(\nu,t,\eta\).
\end{lemma}

\begin{proof}
Let \(\mathsf p_G(\xi)=e^{|\xi|^2/4}\).
By \eqref{eq:gaussian-reaction-zero} and \eqref{eq:p-def},
\[
\begin{aligned}
\mathcal R_{\mathsf p}(t)
={}&
-\frac1\nu
\int_{\mathbb R^2}
(\mathsf p_0-\mathsf p_G)(v\cdot\nabla G)w\,d\xi
\\
&-
t\int_{\mathbb R^2}q(v\cdot\nabla G)w\,d\xi.
\end{aligned}
\]
Thus H\"older's inequality and \eqref{eq:v-L4} give
\[
\begin{aligned}
\bigl|\mathcal R_{\mathsf p}(t)\bigr|
&\le
\|v\|_{L^4}
\|\mathsf p_0^{1/2}w\|_{L^2}
\left(
\frac1\nu
\left\|
\frac{(\mathsf p_0-\mathsf p_G)\nabla G}
{\mathsf p_0^{1/2}}
\right\|_{L^4}
+
t\left\|
\frac{q\nabla G}{\mathsf p_0^{1/2}}
\right\|_{L^4}
\right)
\\
&\le
C_{T,\beta}\mathcal K(t)
\left(
\frac1\nu
\left\|
\frac{(\mathsf p_0-\mathsf p_G)\nabla G}
{\mathsf p_0^{1/2}}
\right\|_{L^4}
+
t\left\|
\frac{q\nabla G}{\mathsf p_0^{1/2}}
\right\|_{L^4}
\right).
\end{aligned}
\]

~\\
We first estimate the coefficient containing
\(\mathsf p_0-\mathsf p_G\).
Set
\[
r_\nu:=\frac{a_0}{2}(\nu T)^{-1/4}.
\]
By \eqref{eq:standing-smallness} and \eqref{eq:weight-scales},
\(a_\nu(t)\ge2\) and
\[
\sqrt{a_\nu(t)^2-1}
\ge \frac12a_\nu(t)
\ge r_\nu,
\qquad 0<t\le T.
\]
Since \(\mathsf p_0=\mathsf p_G\) when
\(|\xi|^2\le a_\nu(t)^2-1\), the difference
\(\mathsf p_0-\mathsf p_G\) vanishes on \(\{|\xi|<r_\nu\}\).
Moreover, \eqref{eq:p0-below-gaussian} gives
\[
|\mathsf p_0-\mathsf p_G|\,|\nabla G|
\le \mathsf p_G|\nabla G|
\lesssim |\xi|.
\]
Using \eqref{eq:p0-global-lower}, we obtain
\[
\left|
\frac{(\mathsf p_0-\mathsf p_G)\nabla G}
{\mathsf p_0^{1/2}}
\right|
\lesssim
|\xi|e^{-\beta|\xi|/8}
\mathbf 1_{\{|\xi|\ge r_\nu\}}.
\]
For sufficiently small \(\nu\), one has \(\beta r_\nu\ge2\), and
polar coordinates yield
\[
\begin{aligned}
\left\|
|\xi|e^{-\beta|\xi|/8}
\mathbf 1_{\{|\xi|\ge r_\nu\}}
\right\|_{L^4}^4
&=
2\pi\int_{r_\nu}^{\infty}r^5e^{-\beta r/2}\,dr
\lesssim
\beta^{-1}r_\nu^5e^{-\beta r_\nu/2}.
\end{aligned}
\]
Consequently,
\[
\sup_{0<t\le T}
\frac1\nu
\left\|
\frac{(\mathsf p_0-\mathsf p_G)\nabla G}
{\mathsf p_0^{1/2}}
\right\|_{L^4}
\lesssim
\frac{\beta^{-1/4}}{\nu}
r_\nu^{5/4}e^{-\beta r_\nu/8}
\to0
\qquad\text{as }\nu\to0,
\]
because \(r_\nu=\frac{a_0}{2}(\nu T)^{-1/4}\).

~\\
For the remaining coefficient, \eqref{eq:q-pointwise} and
\eqref{eq:p0-below-gaussian} imply
\[
\begin{aligned}
\frac{|q\nabla G|}{\mathsf p_0^{1/2}}
&\le
C_T|\xi|^3(1+|\xi|^2)
\mathsf p_0^{1/2}e^{-|\xi|^2/4}
\\
&\le
C_T|\xi|^3(1+|\xi|^2)e^{-|\xi|^2/8}.
\end{aligned}
\]
The right-hand side belongs to \(L^4(\mathbb R^2)\), so
\[
\left\|
\frac{q\nabla G}{\mathsf p_0^{1/2}}
\right\|_{L^4}
\le C_T,
\]
uniformly in \(t\) and \(\nu\).
Returning to the initial estimate and choosing
\(\nu_{T,\eta}\) sufficiently small proves
\eqref{eq:reaction-bound}.
\end{proof}

\begin{proof}[Proof of Theorem~\ref{thm:linear-weighted-estimate}]
First choose \(0<a_0\ll1\) so that
Lemma~\ref{lem:adapted-weight} and
Theorem~\ref{thm:weight-coercivity} apply and
\(c_{\mathrm{tr},T}a_0^4\le1/1536\).
Next choose \(B_0\gg1\) so that
\(c_{\mathrm{tr},T}/B_0\le1/1536\), and set
\(\beta=a_0^2/B_0\).

Finally, decrease \(\nu_T\) so that
\begin{equation}\label{eq:linear-parameter-choice}
\nu_TT\le\frac{a_0^4}{16},
\qquad
c_{\mathrm{tr},T}a_0(\nu_TT)^{\frac14}\le\frac1{1536},
\qquad
\nu_T\le\nu_{T,\frac1{1024}},
\end{equation}
where the last constant is supplied by Lemma~\ref{lem:reaction}.

For these parameters,
\[
\mathcal D_{\pw}(t)
\stackrel{\eqref{eq:weight-coercivity}}{\le}
-\frac1{128}\mathcal Q(t),
\qquad
\abs{\mathcal T_{\pw}(t)}
\stackrel{\substack{\eqref{eq:transport-total-bound}\\
\eqref{eq:linear-parameter-choice}}}{\le}
\frac1{512}\mathcal Q(t).
\]
Moreover, with \(\eta=1/1024\),
\[
\abs{\mathcal R_{\pw}(t)}-\frac1{1024}\mathcal K(t)
\stackrel{\eqref{eq:reaction-bound}}{\lsim{T}}t\mathcal K(t).
\]
Since \(\mathcal K\stackrel{\eqref{eq:Q-def}}{\le}\mathcal Q\),
combining \eqref{eq:energy-parts}, \eqref{eq:weight-coercivity},
\eqref{eq:transport-total-bound}, \eqref{eq:reaction-bound}, and
\eqref{eq:linear-parameter-choice} yields
\[
\begin{aligned}
t\frac{d}{dt}E_{\pw}(t)
+
\frac5{1024}\mathcal Q
-\abs{\inner{f(t)}{w(t)}_{\pw(t)}}
&\lsim{T}
t\mathcal K
\\
&\stackrel{\eqref{eq:weighted-energy-comparable}}{\lesssim}
tE_{\pw}(t).
\end{aligned}
\]
This proves
\eqref{eq:linear-weighted-estimate} with
\(\kappa_T=\frac5{1024}\).
\end{proof}
\section{Estimate for the concentrated remainder}
\label{sec:w2-estimate}
After decreasing \(\nu_T\), we assume throughout
Sections~\ref{sec:w2-estimate}--\ref{sec:closure} that
\(\nu_T\le\nu_{0,T}\) and \(\nu_TT\le1\).
In this section, \(E_{\pw},\mathcal I,\mathcal J,\mathcal K,\mathcal Q\)
denote the quantities of Section~\ref{sec:linearized-estimate} evaluated at
\(w=\wb_2\) and \(v=\vb_2\).
For the regular remainder, we use the norm
\begin{equation}\label{eq:wave-remainder-norm}
\norm{\wb_1(t)}
:=
\norm[L^4(\Rt)]{\wb_1(\cdot,t)}
+
\norm[L^{\frac{4}{3}}(\Rt)]{\wb_1(\cdot,t)}.
\end{equation}
Both Lebesgue norms are taken in the original spatial variable \(x\).

\begin{proposition}
\label{prop:w2-apriori}
After decreasing \(\nu_T\), there is
\(\kappa_1>0\) such that, for almost every
\(t\in(0,T]\),
\begin{equation}\label{eq:w2-apriori}
\begin{aligned}
t\frac{d}{dt}E_{\pw}(t)
+\kappa_1\mathcal Q(t)
\lsim{T}{}&
t\bigl(E_{\pw}(t)+E_{\pw}(t)^2\bigr)
+t^{-1}\norm{\wb_1(t)}^2
\\
&+
\nu t
+\sqrt\nu\,\norm{\wb_1(t)}\mathcal K(t).
\end{aligned}
\end{equation}
\end{proposition}

To prove Proposition~\ref{prop:w2-apriori}, we estimate the forcing
in \eqref{eq:linearized}.
For each \((\xi,t)\), we evaluate the regular velocity at
\begin{equation}\label{eq:w2-physical-point}
x=\Zn(t)+\sqrt{\nu t}\,\xi.
\end{equation}
Equation \eqref{eq:concentrated-remainder-equation} has the form
\eqref{eq:linearized}, with
\begin{equation}\label{eq:w2-linear-source}
\begin{aligned}
f(\xi,t)
={}&-t\,\vb_2\cdot\nabla w_{2,a}
-t\,v_{2,a}\cdot\nabla\wb_2
-t\,\vb_2\cdot\nabla\wb_2
\\
&-\frac{1}{\sqrt t}\,\vb_1(x,t)\cdot\nabla G
-\nu\sqrt t\,\vb_1(x,t)\cdot\nabla w_{2,a}
\\
&-\nu\sqrt t\,\vb_1(x,t)\cdot\nabla\wb_2
-\frac{1}{\nu t}\Phi_{\mathrm{app}}(\xi,t).
\end{aligned}
\end{equation}
The zero-mass condition in Theorem~\ref{thm:linear-weighted-estimate}
follows from \eqref{eq:remainder-mass-constraints}. We estimate the
pairing of each term in \eqref{eq:w2-linear-source} with \(\wb_2\).

By \cite{NN}, the regular velocity obeys
\begin{equation}\label{eq:wave-remainder-velocity-bound}
\begin{aligned}
\norm[L^\infty]{\vb_1(t)}
&\lesssim
\norm[L^4]{\wb_1(t)}^{\frac12}
\norm[L^{\frac{4}{3}}]{\wb_1(t)}^{\frac12}
\stackrel{\eqref{eq:wave-remainder-norm}}{\lesssim}
\norm{\wb_1(t)}.
\end{aligned}
\end{equation}
Thus the three terms in \eqref{eq:w2-linear-source} involving the regular
wave are controlled by the \(L^4\cap L^{\frac{4}{3}}\) norm of \(\wb_1\).
\begin{lemma}\label{lem:weight-pairing}
For the weight constructed above, one has, for every \(0<t\le T\) and
almost every \(\xi\in\Rt\),
\begin{equation}\label{eq:grad-p-global}
\abs{\nabla\pw(t,\xi)}
\lsim{T}
(1+\abs{\xi})\pw(t,\xi),
\end{equation}
and, more sharply,
\begin{equation}\label{eq:nu-grad-p-small}
\nu\sqrt t\,\abs{\nabla\pw(t,\xi)}
\lsim{T}
\sqrt\nu\,\pw(t,\xi).
\end{equation}
Moreover, for each \(\gamma\in(\tfrac12,1)\),
\begin{equation}\label{eq:gaussian-pairing}
\norm[L^2\cap L^4]{
\pw^{\frac{1}{2}}e^{-\frac{\gamma\abs{\xi}^2}{4}}
}
+
\norm[L^2]{\pw^{\frac{1}{2}}\nabla G}
\lsim{\gamma,T}
1.
\end{equation}
The composite corrector satisfies
\begin{equation}\label{eq:composite-profile-bounds}
\abs{w_{2,a}}+\abs{\nabla w_{2,a}}
\lsim{\gamma,T}
e^{-\frac{\gamma\abs{\xi}^2}{4}},
\qquad
\abs{v_{2,a}(\xi,t)}
\lsim{T}
\frac{1}{(1+\abs{\xi}^2)^{\frac{3}{2}}}.
\end{equation}
\end{lemma}
\begin{proof}
In the inner region \(\abs\xi^2\le a_\nu(t)^2+1\),
\(\chi_\nu(t,\xi)\lesssim\abs\xi(1+\abs\xi)\); outside this region,
\(\nabla q=0\) almost everywhere.  Hence
\[
\begin{aligned}
\abs{\nabla\pw}
&\stackrel{\eqref{eq:p-def}}{\le}
\abs{\nabla\pw_0}+(\nu t)\abs{\nabla q}
\stackrel{\substack{\eqref{eq:p0-gradient}\\
\eqref{eq:normalized-q-bounds}}}{\lsim{T}}
(1+\abs\xi)\pw_0
\stackrel{\eqref{eq:p-comparable}}{\lesssim}
(1+\abs\xi)\pw,
\end{aligned}
\]
which proves \eqref{eq:grad-p-global}.  Also,
\[
\pw^{\frac{1}{2}}e^{-\frac{\gamma\abs{\xi}^2}{4}}
\stackrel{\eqref{eq:p-global-bounds}}{\le}
\sqrt2e^{-\frac{(2\gamma-1)\abs{\xi}^2}{8}},
\qquad
\pw^{\frac{1}{2}}\abs{\nabla G}
\stackrel{\eqref{eq:p-global-bounds}}{\lesssim}
\abs{\xi}e^{-\frac{\abs{\xi}^2}{8}}.
\]
These bounds imply \eqref{eq:gaussian-pairing}.  On
\(\{\abs{\xi}\le B_\nu(t)\}\),
\[
\nu\sqrt t\,\abs{\nabla\pw}
\stackrel{\eqref{eq:grad-p-global}}{\lsim{T}}
\nu\sqrt t\,(1+\abs\xi)\pw
\stackrel{\eqref{eq:weight-scales}}{\lesssim}
\sqrt\nu\,\pw.
\]
On \(\{\abs{\xi}\ge B_\nu(t)\}\), one has \(q=0\) and
\[
\nu\sqrt t\,\abs{\nabla\pw}
=\frac{\beta}{4}\nu\sqrt t\,\pw
\lsim{T}\sqrt\nu\,\pw.
\]
This proves \eqref{eq:nu-grad-p-small}.  For the composite profile,
\[
\begin{aligned}
\abs{w_{2,a}}+\abs{\nabla w_{2,a}}
&\stackrel{\substack{\eqref{eq:remainder-core-composite}\\
\eqref{eq:profile-Gaussian-bounds}}}{\lsim{\gamma,T}}
e^{-\gamma\abs\xi^2/4},
\\
\abs{v_{2,a}}
&\stackrel{\substack{\eqref{eq:remainder-core-composite}\\
\eqref{eq:high-profile-velocity}}}{\lsim{T}}
\frac{1}{(1+\abs\xi^2)^{3/2}},
\end{aligned}
\]
where we used \(\nu t\le1\).
\end{proof}
\begin{lemma}\label{lem:w2-corrector-vorticity}
For \(0<t\le T\),
\begin{equation}\label{eq:corrector-vorticity-estimate}
t\left|
\inner{\vb_2\cdot\nabla w_{2,a}}{\wb_2}_{\pw(t)}
\right|
\lsim{T}t\mathcal K(t).
\end{equation}
\end{lemma}
\begin{proof}
H\"older's inequality with exponents \(2,4,4\), followed by
\eqref{eq:v-L4}, \eqref{eq:p-comparable},
\eqref{eq:gaussian-pairing}, and
\eqref{eq:composite-profile-bounds}, gives
\begin{align*}
t\left|
\inner{\vb_2\cdot\nabla w_{2,a}}{\wb_2}_{\pw(t)}
\right|
&\le
t\norm[L^2]{\pw^{\frac12}\wb_2}
\norm[L^4]{\vb_2}
\norm[L^4]{\pw^{\frac12}\nabla w_{2,a}}
\lsim{T}
t\mathcal K(t).
\end{align*}
\end{proof}

\begin{lemma}\label{lem:w2-corrector-transport}
For \(0<t\le T\),
\begin{equation}\label{eq:corrector-transport-estimate}
t\left|
\inner{v_{2,a}\cdot\nabla\wb_2}{\wb_2}_{\pw(t)}
\right|
\lsim{T}t\mathcal K(t).
\end{equation}
\end{lemma}
\begin{proof}
Since \(v_{2,a}=K\star_\xi w_{2,a}\) is divergence free,
integration by parts gives
\[
\inner{v_{2,a}\cdot\nabla\wb_2}{\wb_2}_{\pw(t)}
=
-\frac12\int_{\Rt}
(v_{2,a}\cdot\nabla\pw)\abs{\wb_2}^2\,d\xi.
\]
\[
\abs{v_{2,a}\cdot\nabla\pw}
\stackrel{\substack{\eqref{eq:grad-p-global}\\
\eqref{eq:composite-profile-bounds}}}{\lsim{T}}
\frac{1+\abs\xi}{(1+\abs\xi^2)^{\frac32}}\pw
\lesssim\pw.
\]
Consequently,
\[
\begin{aligned}
t\left|
\inner{v_{2,a}\cdot\nabla\wb_2}{\wb_2}_{\pw(t)}
\right|
&\lsim{T}
t\int_{\Rt}\pw\abs{\wb_2}^2\,d\xi
\stackrel{\eqref{eq:p-comparable}}{\lesssim}
t\mathcal K(t).
\end{aligned}
\]
\end{proof}

\begin{lemma}\label{lem:w2-self-advection}
For every \(\eta>0\),
\begin{equation}\label{eq:nonlinear-core-estimate}
\begin{aligned}
t\left|
\inner{\vb_2\cdot\nabla\wb_2}{\wb_2}_{\pw(t)}
\right|
-\eta\bigl(\mathcal I(t)+\mathcal K(t)\bigr)
\lsim{\eta,T}{}&
t\mathcal K(t)^2.
\end{aligned}
\end{equation}
\end{lemma}
\begin{proof}
The velocity \(\vb_2\) is divergence free, so
\[
\inner{\vb_2\cdot\nabla\wb_2}{\wb_2}_{\pw(t)}
=
-\frac12\int_{\Rt}
(\vb_2\cdot\nabla\pw)\abs{\wb_2}^2\,d\xi.
\]
By \eqref{eq:remainder-mass-constraints} and
Lemma~\ref{lem:v-decay},
\[
\abs{\vb_2(\xi,t)}
\lsim{T}
\frac{1}{1+\abs\xi^2}
\bigl(\mathcal I(t)+\mathcal K(t)\bigr)^{\frac12}.
\]
\[
\begin{aligned}
t\left|
\inner{\vb_2\cdot\nabla\wb_2}{\wb_2}_{\pw(t)}
\right|
\stackrel{\substack{\eqref{eq:grad-p-global}\\
\eqref{eq:p-comparable}}}{\lsim{T}}
t\bigl(\mathcal I+\mathcal K\bigr)^{\frac12}\mathcal K.
\end{aligned}
\]
Young's inequality, followed by \(t^2\le Tt\), gives
\[
\begin{aligned}
t\left|
\inner{\vb_2\cdot\nabla\wb_2}{\wb_2}_{\pw(t)}
\right|
&\le\eta(\mathcal I+\mathcal K)
+C_{\eta,T}t^2\mathcal K^2
\\
&\le\eta(\mathcal I+\mathcal K)
+C_{\eta,T}Tt\mathcal K^2,
\end{aligned}
\]
which proves \eqref{eq:nonlinear-core-estimate}.
\end{proof}

\begin{lemma}\label{lem:w2-wave-G}
For every \(\eta>0\),
\begin{equation}\label{eq:regular-leading-coupling}
\frac{1}{\sqrt t}\left|
\inner{\vb_1(x,t)\cdot\nabla G}{\wb_2}_{\pw(t)}
\right|
-\eta\mathcal K(t)
\lsim{\eta,T}
t^{-1}\norm{\wb_1(t)}^2.
\end{equation}
\end{lemma}
\begin{proof}
H\"older's inequality, \eqref{eq:wave-remainder-velocity-bound},
\eqref{eq:gaussian-pairing}, and \eqref{eq:p-comparable} give
\begin{align*}
\frac{1}{\sqrt t}\left|
\inner{\vb_1(x,t)\cdot\nabla G}{\wb_2}_{\pw(t)}
\right|
&\le
\frac{1}{\sqrt t}
\norm[L^\infty]{\vb_1(t)}
\norm[L^2]{\pw^{\frac12}\nabla G}
\norm[L^2]{\pw^{\frac12}\wb_2}
\\
&\lsim{T}
t^{-\frac12}\norm{\wb_1(t)}\mathcal K(t)^{\frac12}.
\end{align*}
Young's inequality proves \eqref{eq:regular-leading-coupling}.
\end{proof}

\begin{lemma}\label{lem:w2-wave-profile}
For every \(\eta>0\),
\begin{equation}\label{eq:regular-profile-coupling}
\begin{aligned}
\nu\sqrt t\left|
\inner{\vb_1(x,t)\cdot\nabla w_{2,a}}{\wb_2}_{\pw(t)}
\right|
-\eta\mathcal K(t)
&\lsim{\eta,T}
\nu^2t\norm{\wb_1(t)}^2
\\
&\le
t^{-1}\norm{\wb_1(t)}^2.
\end{aligned}
\end{equation}
\end{lemma}
\begin{proof}
H\"older's inequality, \eqref{eq:wave-remainder-velocity-bound},
\eqref{eq:gaussian-pairing}, \eqref{eq:composite-profile-bounds},
and \eqref{eq:p-comparable} give
\begin{align*}
\nu\sqrt t\left|
\inner{\vb_1(x,t)\cdot\nabla w_{2,a}}{\wb_2}_{\pw(t)}
\right|
&\le
\nu\sqrt t\,\norm[L^\infty]{\vb_1(t)}
\norm[L^2]{\pw^{\frac12}\nabla w_{2,a}}
\norm[L^2]{\pw^{\frac12}\wb_2}
\\
&\lsim{T}
\nu\sqrt t\,\norm{\wb_1(t)}\mathcal K(t)^{\frac12}.
\end{align*}
Young's inequality gives the first estimate in
\eqref{eq:regular-profile-coupling}. Since \((\nu t)^2\le1\),
we have \(\nu^2t\norm{\wb_1(t)}^2\le t^{-1}\norm{\wb_1(t)}^2\), which gives the second.
\end{proof}

\begin{lemma}\label{lem:w2-wave-transport}
For \(0<t\le T\),
\begin{equation}\label{eq:regular-core-transport}
\nu\sqrt t\left|
\inner{\vb_1(x,t)\cdot\nabla\wb_2}{\wb_2}_{\pw(t)}
\right|
\lsim{T}
\sqrt\nu\,\norm{\wb_1(t)}\mathcal K(t).
\end{equation}
\end{lemma}
\begin{proof}
Since \(\nabla_x\cdot\vb_1=0\),
\[
\nabla_\xi\cdot\vb_1(x,t)
\stackrel{\eqref{eq:w2-physical-point}}{=}
\sqrt{\nu t}\,
(\nabla_x\cdot\vb_1)(x,t)
=0.
\]
Therefore, integration by parts, together with \eqref{eq:wave-remainder-velocity-bound},
\eqref{eq:nu-grad-p-small}, and \eqref{eq:p-comparable}, gives
\begin{align*}
\nu\sqrt t\left|
\inner{\vb_1(x,t)\cdot\nabla\wb_2}{\wb_2}_{\pw(t)}
\right|
&=
\frac{\nu\sqrt t}{2}
\left|
\int_{\Rt}
(\vb_1(x,t)\cdot\nabla\pw)\abs{\wb_2}^2\,d\xi
\right|
\\
&\lsim{T}
\sqrt\nu\,\norm{\wb_1(t)}\mathcal K(t).
\end{align*}
\end{proof}

\begin{lemma}\label{lem:w2-residual}
For every \(\eta>0\),
\begin{equation}\label{eq:high-residual-pairing}
\frac{1}{\nu t}\left|
\inner{\Phi_{\mathrm{app}}(t)}{\wb_2(t)}_{\pw(t)}
\right|
-\eta\mathcal K(t)
\lsim{\eta,T}
\nu t.
\end{equation}
\end{lemma}
\begin{proof}
Fix \(\gamma\in(\frac12,1)\).
The Cauchy--Schwarz inequality, \eqref{eq:residual-bound},
\eqref{eq:gaussian-pairing}, and \eqref{eq:p-comparable} give
\begin{align*}
\frac{1}{\nu t}\left|
\inner{\Phi_{\mathrm{app}}(t)}{\wb_2(t)}_{\pw(t)}
\right|
&\le
\frac{1}{\nu t}
\norm[L^2]{\pw^{\frac12}\Phi_{\mathrm{app}}(t)}
\norm[L^2]{\pw^{\frac12}\wb_2(t)}
\\
&\lsim{T}
(\nu t)^{\frac12}\mathcal K(t)^{\frac12}.
\end{align*}
Young's inequality proves \eqref{eq:high-residual-pairing}.
\end{proof}

\begin{proof}[Proof of Proposition~\ref{prop:w2-apriori}]
Theorem~\ref{thm:linear-weighted-estimate}, the source identity
\eqref{eq:w2-linear-source}, and
Lemmas~\ref{lem:w2-corrector-vorticity}--\ref{lem:w2-residual} give
\begin{align*}
t\frac{d}{dt}E_{\pw}
+\kappa_T\mathcal Q
-\eta(\mathcal I+\mathcal K)
-3\eta\mathcal K
\lsim{\eta,T}{}&
tE_{\pw}+t\mathcal K+t\mathcal K^2+t^{-1}\norm{\wb_1(t)}^2
+\sqrt\nu\,\norm{\wb_1(t)}\mathcal K+\nu t.
\end{align*}
Choose \(\eta=\kappa_T/8\).  Since
\(\mathcal I+\mathcal K\le\mathcal Q\), the four terms proportional to
\(\eta\) are bounded by \(\frac{\kappa_T}{2}\mathcal Q\) and can be
absorbed into the left-hand side.  Moreover,
\[
\mathcal K
\stackrel{\eqref{eq:weighted-energy-comparable}}{\le}
\frac83E_{\pw},
\]
and hence
\(t\mathcal K\lesssim tE_{\pw}\) and
\(t\mathcal K^2\lesssim tE_{\pw}^2\).
Taking \(\kappa_1=\frac{\kappa_T}{2}\) proves
\eqref{eq:w2-apriori}.
\end{proof}

\section{Estimate for the regular remainder}
\label{sec:w1-estimate}
Following \cite{NN}, we use
\eqref{eq:wave-remainder-equation} to estimate the regular remainder in
\(L^4\cap L^{4/3}\), using the norm
\eqref{eq:wave-remainder-norm}. Denote its right-hand side by
\begin{equation}\label{eq:regular-remainder-forcing}
\begin{aligned}
f_E
={}&
-\vb_1\cdot\nabla\wapp
-\frac{\sqrt t}{\nu}
\bigl(v_{2,a}+\vb_2\bigr)
\left(
\frac{x-\Zn(t)}{\sqrt{\nu t}},t
\right)\cdot\nabla\wapp
\\
&-
\sqrt\nu
\left(
v_{1,a}\cdot\nabla w_{1,a}-\Delta w_{1,a}
\right).
\end{aligned}
\end{equation}
Thus
\(\partial_t\wb_1+u^\nu\cdot\nabla\wb_1-\nu\Delta\wb_1=f_E\).
Its zero trace at \(t=0\) is established in
Theorem~\ref{lem:short-time-startup}.
\begin{lemma}
\label{lem:far-field-core}
If \(t\in(0,T]\) and \(x\in\supp\wapp(t)\), then
\begin{equation}\label{eq:far-field-xi}
\abs{\xi}
=
\frac{\abs{x-\Zn(t)}}{\sqrt{\nu t}}
\ge
\frac8{\sqrt{\nu t}}.
\end{equation}
Consequently,
\begin{equation}\label{eq:far-field-v2a}
\abs{v_{2,a}(\xi,t)}
\lsim{T}
(\nu t)^{\frac{3}{2}},
\end{equation}
and
\begin{equation}\label{eq:far-field-vb2}
\abs{\vb_2(\xi,t)}
\lsim{T}
\nu t\,\mathcal Q(t)^{\frac{1}{2}}.
\end{equation}
\end{lemma}
\begin{proof}
For \(x\in\supp\wapp(t)\),
\[
\abs\xi
\stackrel{\eqref{eq:w2-physical-point}}{=}
\frac{\abs{x-\Zn(t)}}{\sqrt{\nu t}}
\stackrel{\eqref{eq:approx-separation}}{\ge}
\frac8{\sqrt{\nu t}},
\]
which proves \eqref{eq:far-field-xi}.  Moreover,
\[
\abs{v_{2,a}(\xi,t)}
\stackrel{\substack{\eqref{eq:high-profile-velocity}\\
\eqref{eq:remainder-core-composite}}}{\lsim{T}}
\frac{1}{(1+\abs\xi^2)^{3/2}}
\stackrel{\eqref{eq:far-field-xi}}{\lesssim}
(\nu t)^{3/2},
\]
which proves \eqref{eq:far-field-v2a}.
By \eqref{eq:remainder-mass-constraints} and
Lemma~\ref{lem:v-decay},
\[
\abs{\vb_2(\xi,t)}
\lsim{T}
\frac{1}{1+\abs{\xi}^2}
\left(
\mathcal I(t)+\mathcal K(t)
\right)^{\frac{1}{2}}
\stackrel{\eqref{eq:far-field-xi}}{\lesssim}
\nu t\,\mathcal Q(t)^{1/2}.
\]
\end{proof}
\begin{proposition}\label{prop:w1-estimate}
For every \(0<t_0\le t\le T\), uniformly in \(t_0\),
\begin{equation}\label{eq:w1-estimate}
\norm{\wb_1(t)}
\lsim{T}
\norm{\wb_1(t_0)}
+\int_{t_0}^t
s^{\frac32}\mathcal Q(s)^{\frac12}\,ds
+\sqrt\nu\,(t-t_0).
\end{equation}
\end{proposition}

\begin{proof}
The regular remainder \(\wb_1\) solves
\eqref{eq:wave-remainder-equation}. Its transport velocity
\(u^\nu\) is divergence free, and its forcing \(f_E\) is defined in
\eqref{eq:regular-remainder-forcing}.
The standard \(L^p\) estimate, applied with \(p=4\) and \(p=4/3\),
gives
\[
\norm{\wb_1(t)}
\le
\norm{\wb_1(t_0)}
+\int_{t_0}^t\norm[\lbb]{f_E(s)}\,ds,
\]
where we use the norm \eqref{eq:wave-remainder-norm}.

By \eqref{eq:wave-remainder-velocity-bound} and the regularity
bounds in Proposition~\ref{prop:regular-corrector},
\[
\norm[\lbb]{\vb_1\cdot\nabla\wapp}
\lsim{T}\norm{\wb_1(t)}.
\]
On the support of \(\nabla\wapp\),
\eqref{eq:far-field-v2a} and \eqref{eq:far-field-vb2} yield
\[
\frac{\sqrt t}{\nu}\abs{v_{2,a}(\xi,t)}
\lsim{T}\sqrt\nu\,t^2
\lsim{T}\sqrt\nu,
\qquad
\frac{\sqrt t}{\nu}\abs{\vb_2(\xi,t)}
\lsim{T}t^{\frac32}\mathcal Q(t)^{\frac12}.
\]
Moreover, Proposition~\ref{prop:regular-corrector} gives
\[
\sqrt\nu\,
\norm[\lbb]{v_{1,a}\cdot\nabla w_{1,a}-\Delta w_{1,a}}
\lsim{T}\sqrt\nu.
\]
Combining these bounds in
\eqref{eq:regular-remainder-forcing}, we obtain
\[
\norm[\lbb]{f_E(t)}
\lsim{T}
\norm{\wb_1(t)}
+t^{\frac32}\mathcal Q(t)^{\frac12}
+\sqrt\nu.
\]
Gronwall's inequality on \([t_0,t]\) now proves
\eqref{eq:w1-estimate}, with a constant independent of \(t_0\).
\end{proof}
\Needspace{12\baselineskip}
\section{Coupled bootstrap and proof of the main theorem}
\label{sec:closure}

We first establish the initial traces of the remainders with
\(\nu>0\) fixed, and then close the estimates uniformly in \(\nu\)
on \([0,T]\). Throughout this section, \(\pw_G\) and \(Y\) are
defined in \eqref{eq:Gaussian-space-Y}, and
\(\norm{\wb_1(t)}\) denotes the norm in
\eqref{eq:wave-remainder-norm}.

\subsection{Initial traces of the remainders}

For \(t>0\), set
\[
h(t):=w_2(t)-w_{2,\mathrm{app}}(t),
\qquad
e(t):=\w^{E,\nu}(t)-\wapp(t).
\]
For each fixed \(\nu>0\), the atomic-profile convergence and
Gaussian bounds in \cite[Section~2.1 and Proposition~4.5]{GG},
after rescaling the viscosity and using
\(\abs{\Zn(t)-z_0}/\sqrt{\nu t}\to0\), give
\begin{equation}\label{eq:short-time-G-trace}
w_2(t)\to G\quad\hbox{in }Y,
\qquad
\abs{w_2(\xi,t)}\le
C_{\nu,\gamma}e^{-\gamma\abs\xi^2/4},
\qquad 0<\gamma<1,
\end{equation}
where the Gaussian bound holds uniformly for sufficiently small
positive times. Here the Gaussian bound upgrades the convergence
in polynomially weighted spaces to convergence in \(Y\).
The strong initial trace of the wave component follows from its
measure trace and the \(L^p\) contraction estimate. Together with
the initial traces of the approximations, this yields
\begin{equation}\label{eq:short-time-wave-trace}
\norm[Y]{h(t)}\to0,
\qquad
\norm[\lbb]{e(t)}\to0,
\qquad t\to0^+.
\end{equation}

These qualitative limits do not yet determine the initial trace of
\(\wb_2=h/(\nu t)\). We need a rate for \(h\) that remains
vanishing after division by \(t\), and a rate for \(e\) that makes
\(\norm{\wb_1(t)}^2/t^2\) integrable. The following estimate
provides both. Its proof first controls the wave error using
separation and mass cancellation, and then uses a Gaussian energy
estimate for the core error.

\begin{theorem}[Short-time estimates]
\label{lem:short-time-startup}
For each fixed \(0<\nu\le\nu_T\), there are \(C_\nu>0\) and a
time \(t_\nu\), with \(0<t_\nu\le\min\{T,1\}\), such that
\begin{equation}\label{eq:short-time-startup}
\norm[Y]{h(t)}\le C_\nu t^{3/2},
\qquad
\norm[\lbb]{e(t)}\le C_\nu t,
\qquad 0<t\le t_\nu.
\end{equation}
Consequently,
\begin{equation}\label{eq:short-time-remainder-traces}
E_{\pw}(t)\le C_\nu t,
\qquad
\norm{\wb_1(t)}\le C_\nu t,
\qquad 0<t\le t_\nu.
\end{equation}
The constants and \(t_\nu\) may depend on \(\nu\).
\end{theorem}

\begin{proof}
Fix \(\nu>0\), and write
\[
v^e:=K\star_x e,
\qquad
v^h:=K\star_\xi h.
\]
The Gaussian bounds and standard localized energy estimates justify
the weighted integrations by parts below on every compact interval
of positive times.

We first estimate the wave error. Since \(w_2-G\) has zero mass,
the Gaussian bound in \eqref{eq:short-time-G-trace} and the
Biot--Savart formula give
\begin{equation}\label{eq:short-time-far-field}
\abs{v_2(\xi,t)-v^G(\xi)}
\le\frac{C_\nu}{1+\abs\xi^2}.
\end{equation}
Indeed, subtracting \(K(\xi)\) inside the convolution removes the
leading far-field term. By \eqref{eq:approx-separation}, it follows
that
\[
\abs{v^{B,\nu}(x,t)-V_G^{\Zn,\nu}(x,t)}
\le C_\nu\sqrt t,
\qquad x\in\supp\nabla\wapp(t).
\]
Subtracting \eqref{eq:regular-approx-residual} from the exact wave
equation, and using the transport--diffusion estimate together with
\(\norm[L^\infty]{v^e}\lesssim\norm[\lbb]{e}\), gives
\[
\norm[\lbb]{e(t)}
\le C_\nu\int_0^t
\left(\norm[\lbb]{e(s)}+\sqrt s+\nu^2\right)\,ds.
\]
Here the initial term vanishes by
\eqref{eq:short-time-wave-trace}. Gronwall's inequality yields
\[
\norm[\lbb]{e(t)}+\norm[L^\infty]{v^e(t)}
\le C_\nu t,
\qquad 0<t\le t_\nu\le1.
\]

We next estimate the core error. A direct calculation from
\eqref{eq:concentrated-remainder-equation}, using
\(h=\nu t\,\wb_2\) and \(v^e=\nu^{3/2}\vb_1\), gives
\begin{equation}\label{eq:unscaled-core-error}
\begin{aligned}
0={}&(t\partial_t-\cL)h+\nu^{-1}\Lambda h
+\sqrt{\frac t\nu}\,b\cdot\nabla h
\\
&+t\left(v_{2,a}\cdot\nabla h
+v^h\cdot\nabla w_{2,a}\right)
+\nu^{-1}v^h\cdot\nabla h
\\
&+\sqrt{\frac t\nu}\,v^e(x,t)\cdot
\nabla(w_{2,\mathrm{app}}+h)
+\Phi_{\mathrm{app}},
\end{aligned}
\end{equation}
where \(x=\Zn(t)+\sqrt{\nu t}\,\xi\). Set
\[
\widehat E_h(t):=\frac12\norm[Y]{h(t)}^2,
\qquad
D_h(t):=\int_{\Rt}\pw_G
\left(\abs{\nabla h}^2+(1+\abs\xi^2)\abs h^2\right)\,d\xi.
\]
Since \(h\) has zero mass, the Gaussian spectral gap and the
skew-symmetry of \(\Lambda\) give
\[
-\inner{\cL h}{h}_Y\ge cD_h,
\qquad
\inner{\Lambda h}{h}_Y=0,
\]
as in \cite{Ga11}. In particular, \(D_h\) controls
\(\widehat E_h\).

The estimates needed for the remaining terms are
\[
\abs{b(\xi,t)}\le C_T\sqrt{\nu t}\,\abs\xi,
\qquad
\norm[L^\infty]{v^h}
\le C\widehat E_h^{1/4}D_h^{1/4},
\]
and
\[
\left\|
\sqrt{\frac t\nu}\,v^e(x,t)\cdot\nabla w_{2,\mathrm{app}}
+\Phi_{\mathrm{app}}
\right\|_Y
\le C_\nu t^{3/2}.
\]
The last bound follows from the wave estimate above,
\eqref{eq:composite-profile-bounds}, and
\eqref{eq:residual-bound}. Testing
\eqref{eq:unscaled-core-error} against \(h\) in \(Y\), integrating
the divergence-free transport terms by parts, and applying Young's
inequality therefore yields
\[
t\widehat E_h'(t)+cD_h(t)
\le C_\nu tD_h(t)
+C_\nu\widehat E_h(t)^3+C_\nu t^3.
\]
The cubic term comes from \(\nu^{-1}v^h\cdot\nabla h\); the
last term is the square of the forcing bound above.

By \eqref{eq:short-time-wave-trace},
\(\widehat E_h(t)\to0\). Since \(D_h\) controls
\(\widehat E_h\), decreasing \(t_\nu\) allows us to absorb both
\(C_\nu tD_h\) and \(C_\nu\widehat E_h^3\). Hence
\[
\widehat E_h'(t)\le C_\nu t^2,
\qquad 0<t\le t_\nu.
\]
Integrating from \(\delta\) to \(t\) and letting
\(\delta\to0^+\), using \(\widehat E_h(\delta)\to0\), gives
\[
\widehat E_h(t)\le C_\nu t^3,
\qquad
\norm[Y]{h(t)}\le C_\nu t^{3/2}.
\]
Together with the wave estimate, this proves
\eqref{eq:short-time-startup}.

Finally, \(h=\nu t\,\wb_2\),
\(e=\nu^{3/2}\wb_1\), and \(\pw\le2\pw_G\) imply
\[
E_{\pw}(t)
\le\frac{\norm[Y]{h(t)}^2}{(\nu t)^2}
\le C_\nu t,
\qquad
\norm{\wb_1(t)}
=\nu^{-3/2}\norm[\lbb]{e(t)}
\le C_\nu t,
\]
after increasing the constant depending on \(\nu\). This proves
\eqref{eq:short-time-remainder-traces}.
\end{proof}

\subsection{The uniform coupled estimate}

We now obtain bounds whose constants are independent of \(\nu\).
The two remainder estimates are coupled: the core energy inequality
contains \(\norm{\wb_1(t)}^2/t^2\), while the wave estimate contains
\(\int_0^t s^{3/2}\mathcal Q(s)^{1/2}\,ds\).
The key observation is that this integral is controlled by the
accumulated core dissipation
\(\int_0^t\mathcal Q(s)/s\,ds\), which already appears when the
core energy inequality is integrated in time. Combining the energy
with this accumulated dissipation will therefore close the two
estimates in a single differential inequality.

\begin{proposition}\label{prop:uniform-closure}
After decreasing \(\nu_T>0\) if necessary, the integral
\begin{equation}\label{eq:integrated-core-dissipation}
D(t):=\int_0^t\frac{\mathcal Q(s)}s\,ds
\end{equation}
is finite, and
\begin{equation}\label{eq:uniform-closure}
E_{\pw}(t)+\kappa_1D(t)\le C_T\nu t,
\qquad
\norm{\wb_1(t)}\le C_T\sqrt\nu\,t,
\end{equation}
for \(0<\nu\le\nu_T\) and \(0<t\le T\).
The constant \(C_T\) is independent of \(\nu\).
\end{proposition}

\begin{proof}
We take \(\nu_T\le1\). Dividing \eqref{eq:w2-apriori} by \(t\)
gives
\begin{equation}\label{eq:closure-core-energy}
\begin{aligned}
E_{\pw}'(t)+\kappa_1\frac{\mathcal Q(t)}t
\le C_T\biggl(
&E_{\pw}(t)+E_{\pw}(t)^2
+\frac{\norm{\wb_1(t)}^2}{t^2}+\nu
\\
&+\sqrt\nu\,
\frac{\norm{\wb_1(t)}}t\,\mathcal K(t)
\biggr).
\end{aligned}
\end{equation}

We first verify integrability at the initial time. For fixed \(\nu\),
Theorem~\ref{lem:short-time-startup} gives
\[
E_{\pw}(t)\le C_\nu t,
\qquad
\norm{\wb_1(t)}\le C_\nu t,
\qquad 0<t\le t_\nu.
\]
Together with \(\mathcal K\le\frac83E_{\pw}\), this shows that every
term on the right-hand side of \eqref{eq:closure-core-energy} is
integrable near zero. In particular,
\[
\frac{\norm{\wb_1(t)}^2}{t^2}\le C_\nu,
\qquad
\sqrt\nu\,\frac{\norm{\wb_1(t)}}t\,\mathcal K(t)
\le C_\nu t.
\]
Integrating \eqref{eq:closure-core-energy} from \(\delta\) to \(t\)
and letting \(\delta\to0^+\), nonnegativity of the energy,
\(E_{\pw}(\delta)\to0\), and monotone convergence for the
dissipation give \(D(t)<\infty\). Consequently,
\[
E_{\pw}(t)\to0,
\qquad
D(t)\to0,
\qquad t\to0^+.
\]

We can now pass to the initial time in the wave estimate.
Cauchy--Schwarz gives
\[
\begin{aligned}
\int_0^t s^{3/2}\mathcal Q(s)^{1/2}\,ds
&=\int_0^t s^2
\left(\frac{\mathcal Q(s)}s\right)^{1/2}\,ds
\\
&\le
\left(\int_0^t s^4\,ds\right)^{1/2}D(t)^{1/2}
\\
&=\frac{t^{5/2}}{\sqrt5}D(t)^{1/2}.
\end{aligned}
\]
Thus, letting \(t_0\to0^+\) in \eqref{eq:w1-estimate} and using
\(\norm{\wb_1(t_0)}\to0\), we obtain
\begin{equation}\label{eq:w1-from-zero}
\begin{aligned}
\norm{\wb_1(t)}
&\le C_T\left(
\int_0^t s^{3/2}\mathcal Q(s)^{1/2}\,ds
+\sqrt\nu\,t
\right)
\\
&\le C_T
\left(t^{5/2}D(t)^{1/2}+\sqrt\nu\,t\right),
\\
\frac{\norm{\wb_1(t)}^2}{t^2}
&\le C_T\left(\nu+t^3D(t)\right).
\end{aligned}
\end{equation}

The last inequality controls the singular wave source by
\(\nu\) and the accumulated core dissipation. We therefore combine
the energy and the dissipation by setting
\begin{equation}\label{eq:combined-bootstrap-norm}
Y_*(t):=E_{\pw}(t)+\kappa_1D(t),
\qquad
Y_*(0):=0.
\end{equation}
This function is continuous on \([0,T]\), locally absolutely
continuous on \((0,T]\), and satisfies
\[
Y_*'(t)
=E_{\pw}'(t)+\kappa_1\frac{\mathcal Q(t)}t
\]
for almost every \(t>0\).
Young's inequality and
\eqref{eq:weighted-energy-comparable} give
\[
\sqrt\nu\,
\frac{\norm{\wb_1(t)}}t\,\mathcal K(t)
\le
\frac12\frac{\norm{\wb_1(t)}^2}{t^2}
+C\nu E_{\pw}(t)^2.
\]
Substituting this bound and \eqref{eq:w1-from-zero} into
\eqref{eq:closure-core-energy}, and using
\[
\nu\le1,\qquad t\le T,\qquad
E_{\pw}(t)\le Y_*(t),\qquad
D(t)\le\kappa_1^{-1}Y_*(t),
\]
we obtain
\begin{equation}\label{eq:closed-bootstrap-ode}
Y_*'(t)
\le C_T\left(\nu+Y_*(t)+Y_*(t)^2\right)
\end{equation}
for almost every \(t\in(0,T]\).

On any interval starting at zero on which \(Y_*\le1\),
\eqref{eq:closed-bootstrap-ode} gives
\[
Y_*'(t)\le C_T\nu+2C_TY_*(t).
\]
Applying Gronwall's inequality from \(\delta>0\) and letting
\(\delta\to0^+\), we find
\[
Y_*(t)
\le\frac{\nu}{2}\left(e^{2C_Tt}-1\right)
\le C_T'\nu t.
\]
Choose \(\nu_T\) so that \(C_T'\nu_TT<1/2\).
Since \(Y_*(0)=0\) and \(Y_*\) is continuous, it cannot reach
\(1\) on \([0,T]\). Hence
\[
E_{\pw}(t)+\kappa_1D(t)\le C_T\nu t
\]
throughout this interval. Returning to \eqref{eq:w1-from-zero},
we obtain
\[
\frac{\norm{\wb_1(t)}^2}{t^2}
\le C_T\left(\nu+t^3D(t)\right)
\le C_T\nu,
\]
which proves \eqref{eq:uniform-closure}.
The fixed-\(\nu\) estimates were used only to justify integration
from zero; the constants in \eqref{eq:uniform-closure} are
independent of \(\nu\).
\end{proof}

\subsection{Proof of Theorem~\ref{thm:main}}

\begin{proof}
Under the normalization
\eqref{eq:normalized-limiting-separation},
the decomposition \eqref{eq:remainder-ansatz},
Proposition~\ref{prop:regular-corrector}, and
\eqref{eq:uniform-closure} give
\[
\begin{aligned}
\norm[\lbb]{\w^{E,\nu}(t)-\w^E(t)}
&\le
\nu\norm[\lbb]{w_{1,a}(t)}
+\nu^{3/2}\norm{\wb_1(t)}
\\
&\le C_T(\nu+\nu^2t)
\le C_T\nu.
\end{aligned}
\]
Both regular components equal \(\w_0^E\) at \(t=0\), so this
estimate holds on \([0,T]\).

The self-similar change of variables preserves the \(L^1\) norm.
Moreover, the lower bound in \eqref{eq:p-global-bounds} gives
\[
\begin{aligned}
\norm[L^1_\xi]{\wb_2(t)}
&\le
\left(\int_{\Rt}\pw(t,\xi)^{-1}\,d\xi\right)^{1/2}
\left(\int_{\Rt}\pw(t,\xi)|\wb_2(\xi,t)|^2\,d\xi\right)^{1/2}
\\
&\le C_T E_{\pw}(t)^{1/2}.
\end{aligned}
\]
Therefore, by \eqref{eq:concentrated-approximation-error},
\eqref{eq:remainder-ansatz}, and \eqref{eq:uniform-closure},
\[
\begin{aligned}
\norm[L^1]{\w^{B,\nu}(t)-\Gamma_{\nu t,\Zn(t)}}
&=\norm[L^1_\xi]{w_2(t)-G}
\\
&\le C_T\nu t+\nu t\norm[L^1_\xi]{\wb_2(t)}
\\
&\le C_T\nu t\left(1+E_{\pw}(t)^{1/2}\right)
\\
&\le C_T\nu t.
\end{aligned}
\]

The center estimate \eqref{eq:main-center} follows from
\eqref{eq:regular-corrector-bounds}:
\[
|\Zn(t)-z(t)|\le C_T\nu t.
\]
For \(s>0\) and \(a,b\in\Rt\), integration of the directional
derivative of the Gaussian along the segment joining \(a\) and
\(b\) gives
\[
\norm[L^1]{\Gamma_{s,a}-\Gamma_{s,b}}
\le\frac{|a-b|}{\sqrt{\pi s}}.
\]
Taking \(s=\nu t\), \(a=\Zn(t)\), and \(b=z(t)\), we obtain
\[
\norm[L^1]{\Gamma_{\nu t,\Zn(t)}-\Gamma_{\nu t,z(t)}}
\le C_T\sqrt{\nu t}.
\]
Since \(\nu t\le1\), the triangle inequality proves
\eqref{eq:main-z-centered}. For each fixed \(t>0\), the asserted
inviscid limits follow from these estimates and the weak-* convergence
\[
\Gamma_{\nu t,z(t)}\to\delta_{z(t)}
\qquad\text{as }\nu\to0.
\]

Finally, undoing the scaling
\eqref{eq:parabolic-normalization} and absorbing the fixed powers
of \(\ell\) into the constants gives the estimates in the original
variables.
\end{proof}
\appendix
\section{Transport and Biot--Savart estimates for the corrector}
\label{app:corrector-estimates}
We collect the estimates used to construct the localized corrector in
Section~\ref{sec:approx}.

\begin{lemma}
\label{lem:corrector-estimates}
Let \(g=g(t)\) be a smooth family of functions supported in a fixed
compact set, with all spatial derivatives bounded uniformly in time.

For every
\(f\in W^{4,4}(\Rt)\cap W^{3,1}(\Rt)\),
\begin{equation}\label{eq:BS-high}
\norm[W^{3,\infty}]{K\star f}
+\norm[W^{4,4}]{K\star f}
\lesssim
\norm[W^{4,4}]{f}
+\norm[W^{3,1}]{f}
,
\end{equation}
and
\begin{align}
&\norm[W^{4,4}]{(K\star f)\cdot\nabla g}
+\norm[W^{3,1}]{(K\star f)\cdot\nabla g}
\notag\\
&\qquad\lsim{g}
\norm[W^{4,4}]{f}
+\norm[W^{3,1}]{f}
.
\label{eq:nonlocal-product-high}
\end{align}
For every
\(f\in W^{2,4}(\Rt)\cap W^{1,1}(\Rt)\),
\begin{equation}\label{eq:BS-low}
\norm[W^{2,\infty}]{K\star f}
+\norm[W^{3,4}]{K\star f}
\lesssim
\norm[W^{2,4}]{f}
+\norm[W^{1,1}]{f}
,
\end{equation}
and
\begin{align}
&\norm[W^{2,4}]{(K\star f)\cdot\nabla g}
+\norm[W^{1,1}]{(K\star f)\cdot\nabla g}
\notag\\
&\qquad\lsim{g}
\norm[W^{2,4}]{f}
+\norm[W^{1,1}]{f}
.
\label{eq:nonlocal-product-low}
\end{align}

Let \(V=V(t)\) be divergence free.  If \(\partial_tf+V\cdot\nabla f=F\),
then
\begin{align}
\frac d{dt}
\left(
\norm[W^{4,4}]{f}
+\norm[W^{3,1}]{f}
\right)
&\lesssim
\norm[W^{5,\infty}]{V}
\left(
\norm[W^{4,4}]{f}
+\norm[W^{3,1}]{f}
\right)
\notag\\
&\quad
+\norm[W^{4,4}]{F}
+\norm[W^{3,1}]{F},
\label{eq:transport-high}
\\
\frac d{dt}
\left(
\norm[W^{2,4}]{f}
+\norm[W^{1,1}]{f}
\right)
&\lesssim
\norm[W^{3,\infty}]{V}
\left(
\norm[W^{2,4}]{f}
+\norm[W^{1,1}]{f}
\right)
\notag\\
&\quad
+\norm[W^{2,4}]{F}
+\norm[W^{1,1}]{F}.
\label{eq:transport-low}
\end{align}
These differential inequalities hold for almost every \(t\in[0,T]\),
and their integrated forms hold on every subinterval of \([0,T]\).
\end{lemma}

\begin{proof}
\noindent\emph{Biot--Savart estimates.}
We first claim that
\begin{equation}\label{eq:BS-zero-order}
\norm[L^4]{K\star f}
+\norm[L^\infty]{K\star f}
\lesssim
\norm[L^4]{f}+\norm[L^1]{f},
\end{equation}
and, for every integer \(j\ge1\),
\begin{equation}\label{eq:BS-derivatives}
\norm[L^4]{\nabla^j(K\star f)}
\lsim{j}
\norm[L^4]{\nabla^{j-1}f}.
\end{equation}

Since \(\abs{K(x)}=(2\pi\abs x)^{-1}\),
\[
K\mathbf 1_{\{\abs x<1\}}
\in L^1(\Rt)\cap L^{\frac{4}{3}}(\Rt),
\qquad
K\mathbf 1_{\{\abs x\ge1\}}
\in L^4(\Rt)\cap L^\infty(\Rt).
\]
Young's inequality therefore gives
\begin{align*}
\norm[L^4]{K\star f}
&\le
\norm[L^1]{K\mathbf 1_{\{\abs x<1\}}}
\norm[L^4]{f}
+
\norm[L^4]{K\mathbf 1_{\{\abs x\ge1\}}}
\norm[L^1]{f},
\\
\norm[L^\infty]{K\star f}
&\le
\norm[L^{\frac{4}{3}}]{K\mathbf 1_{\{\abs x<1\}}}
\norm[L^4]{f}
+
\norm[L^\infty]{K\mathbf 1_{\{\abs x\ge1\}}}
\norm[L^1]{f}.
\end{align*}
This proves \eqref{eq:BS-zero-order}.

To prove \eqref{eq:BS-derivatives}, let \(\alpha\) be a multi-index with
\(\abs\alpha=j\ge1\), and choose \(k\) such that \(\alpha_k\ge1\).
In Fourier variables,
\[
\widehat{\partial^\alpha(K\star f)}(\xi)
=
m_k(\xi)
\widehat{\partial^{\alpha-e_k}f}(\xi),
\]
where \(m_k\) is a homogeneous multiplier of degree zero, smooth on
\(\Rt\setminus\{0\}\).  It is therefore a Calder\'on--Zygmund
multiplier, and
\[
\norm[L^4]{\partial^\alpha(K\star f)}
\lsim{\alpha}
\norm[L^4]{\partial^{\alpha-e_k}f}.
\]
Summing over \(\abs\alpha=j\) proves
\eqref{eq:BS-derivatives}.

\[
\norm[W^{4,4}]{K\star f}
\stackrel{\substack{\eqref{eq:BS-zero-order}\\
\eqref{eq:BS-derivatives}}}{\lesssim}
\norm[W^{3,4}]{f}+\norm[L^1]{f}.
\]
Moreover, \(W^{1,4}(\Rt)\hookrightarrow L^\infty(\Rt)\), so, for
\(0\le j\le3\),
\[
\begin{aligned}
\norm[L^\infty]{\nabla^j(K\star f)}
&\lesssim
\norm[L^4]{\nabla^j(K\star f)}
+
\norm[L^4]{\nabla^{j+1}(K\star f)}
\\
&\stackrel{\substack{\eqref{eq:BS-zero-order}\\
\eqref{eq:BS-derivatives}}}{\lesssim}
\norm[W^{3,4}]{f}+\norm[L^1]{f}.
\end{aligned}
\]
Together with the preceding \(W^{4,4}\)-estimate, this proves
\eqref{eq:BS-high}.  The same argument through order three in \(L^4\)
and order two in \(L^\infty\) proves \eqref{eq:BS-low}.

\smallskip
\noindent\emph{Product estimates.}
The product estimates \eqref{eq:nonlocal-product-high} and
\eqref{eq:nonlocal-product-low} follow from Leibniz' rule, the
Biot--Savart estimates, and the compact support of \(g\).

Indeed, for every multi-index \(\alpha\),
\[
\partial^\alpha\bigl((K\star f)\cdot\nabla g\bigr)
=
\sum_{\beta\le\alpha}
\binom{\alpha}{\beta}
\partial^\beta(K\star f)\cdot
\partial^{\alpha-\beta}\nabla g.
\]
All derivatives of \(g\) are uniformly bounded.  Hence, for
\(\abs\alpha\le4\),
\[
\norm[L^4]{
\partial^\alpha\bigl((K\star f)\cdot\nabla g\bigr)}
\lsim{g}
\norm[W^{4,4}]{K\star f}.
\]
If \(\abs\alpha\le3\), every term is supported in the fixed compact
support of \(g\).  H\"older's inequality on that compact set gives
\[
\norm[L^1]{
\partial^\alpha\bigl((K\star f)\cdot\nabla g\bigr)}
\lsim{g}
\norm[W^{3,4}]{K\star f}.
\]
Summing these estimates gives
\[
\begin{aligned}
&\norm[W^{4,4}]{(K\star f)\cdot\nabla g}
+\norm[W^{3,1}]{(K\star f)\cdot\nabla g}
\\
&\qquad\lsim{g}
\norm[W^{4,4}]{K\star f}
+\norm[W^{3,4}]{K\star f}
\stackrel{\eqref{eq:BS-high}}{\lsim{g}}
\norm[W^{4,4}]{f}+\norm[W^{3,1}]{f}
,
\end{aligned}
\]
which is \eqref{eq:nonlocal-product-high}.

Taking instead \(\abs\alpha\le2\) in \(L^4\) and
\(\abs\alpha\le1\) in \(L^1\) gives
\[
\begin{aligned}
&\norm[W^{2,4}]{(K\star f)\cdot\nabla g}
+\norm[W^{1,1}]{(K\star f)\cdot\nabla g}
\\
&\qquad\lsim{g}
\norm[W^{2,4}]{K\star f}
+\norm[W^{1,4}]{K\star f}
\stackrel{\eqref{eq:BS-low}}{\lsim{g}}
\norm[W^{2,4}]{f}+\norm[W^{1,1}]{f}
,
\end{aligned}
\]
which proves \eqref{eq:nonlocal-product-low}.

\smallskip
\noindent\emph{Transport estimates.}
To prove \eqref{eq:transport-high} and \eqref{eq:transport-low}, we
commute spatial derivatives with the transport equation.

For any multi-index \(\alpha\),
\begin{equation}\label{eq:differentiated-transport}
\partial_t\partial^\alpha f
+V\cdot\nabla\partial^\alpha f
=
\partial^\alpha F
-
\sum_{0<\beta\le\alpha}
\binom{\alpha}{\beta}
\partial^\beta V\cdot
\nabla\partial^{\alpha-\beta}f.
\end{equation}
For \(p=4\), test by
\(\abs{\partial^\alpha f}^{2}\partial^\alpha f\) and integrate.
Since \(\nabla\cdot V=0\),
\[
\int_{\Rt}
V\cdot\nabla\partial^\alpha f\,
\abs{\partial^\alpha f}^{2}\partial^\alpha f\,dx
=
\frac14\int_{\Rt}
V\cdot\nabla\abs{\partial^\alpha f}^{4}\,dx
=0.
\]
Consequently,
\begin{align}
\frac d{dt}\norm[L^4]{\partial^\alpha f}
&\stackrel{\eqref{eq:differentiated-transport}}{\lsim{\alpha}}
\norm[L^4]{\partial^\alpha F}
\notag\\
&\quad+
\sum_{0<\beta\le\alpha}
\norm[L^\infty]{\partial^\beta V}
\norm[L^4]{
\nabla\partial^{\alpha-\beta}f}.
\label{eq:transport-L4-derivative}
\end{align}
Since \(\abs\beta\ge1\), \(\abs{\alpha-\beta}+1\le\abs\alpha\),
so the right-hand side contains no derivative of \(f\) above the order
being estimated.

For the \(L^1\)-estimate, use
\(\Phi_\varepsilon(s)=\sqrt{s^2+\varepsilon^2}-\varepsilon\).
Then
\[
0\le\Phi_\varepsilon(s)\le\abs s,
\qquad
\abs{\Phi_\varepsilon'(s)}\le1,
\qquad
\Phi_\varepsilon(s)\to\abs s.
\]
Testing by \(\Phi_\varepsilon'(\partial^\alpha f)\) gives
\[
\int_{\Rt}
V\cdot\nabla
\Phi_\varepsilon(\partial^\alpha f)\,dx
=0
\]
because \(V\) is divergence free.  Hence
\begin{align*}
\frac d{dt}
\int_{\Rt}\Phi_\varepsilon(\partial^\alpha f)\,dx
&\stackrel{\eqref{eq:differentiated-transport}}{\lsim{\alpha}}
\norm[L^1]{\partial^\alpha F}
\\
&\quad+
\sum_{0<\beta\le\alpha}
\norm[L^\infty]{\partial^\beta V}
\norm[L^1]{
\nabla\partial^{\alpha-\beta}f}.
\end{align*}
Letting \(\varepsilon\to0\) gives the corresponding estimate for
\(\norm[L^1]{\partial^\alpha f}\).

Summing the \(L^4\)-estimates for \(\abs\alpha\le4\) and the
\(L^1\)-estimates for \(\abs\alpha\le3\) proves
\eqref{eq:transport-high}.  Summing them for
\(\abs\alpha\le2\) and \(\abs\alpha\le1\), respectively, proves
\eqref{eq:transport-low}.
\end{proof}

\section*{Acknowledgments}
The author is supported by the NSFC Excellent Young Scientists Fund, the
Shanghai BYL Talent Program, and the Startup Research Fund of Shanghai Jiao
Tong University.

\section*{Statements and declarations}
\paragraph{Data availability.}
No datasets were generated or analyzed in this theoretical study.

\bibliographystyle{abbrv}

\end{document}